 \documentclass[reqno,11pt]{amsart}
\usepackage{amssymb,amsmath,amsthm,}
  \usepackage{a4wide} 
\usepackage{amscd}
\usepackage{amsfonts}
\usepackage{amssymb,mathrsfs}
\usepackage{latexsym}
\usepackage{color}
\usepackage{esint}

\usepackage{graphicx}
\graphicspath{ {images/} }

\usepackage[makeroom]{cancel}

\newtheorem{theorem}{Theorem}

\newtheorem{definition}{Definition}
\newtheorem{lemma}{Lemma}
\newtheorem{proposition}[theorem]{Proposition}
\newtheorem{remark}{Remark}

\let\e=\varepsilon
\let\d=\delta

\let\p=\partial

\let\o=\omega

\numberwithin{equation}{section}

\let\hide\iffalse
\let\unhide\fi

\newcommand{\R}{\mathbb{R}}
\renewcommand{\P}{\mathbf{P}}
\renewcommand{\S}{\mathbb{S}}

\newcommand{\be}{\begin{equation}}
\newcommand{\bm}{\begin{multline}}
\newcommand{\ee}{\end{equation}}
\newcommand{\dd}{\mathrm{d}}

\newcommand{\f}{\frac}
\newcommand{\pp}{\mathfrak{p}}
\newcommand{\T}{\mathbb{T}}

\newcommand{\Bes}{\begin{eqnarray*}}
\newcommand{\Ees}{\end{eqnarray*}}
\newcommand{\Be}{\begin{equation} }
\newcommand{\Ee}{\end{equation}}
\newcommand{\Bs}{\begin{split}}

	\usepackage{mathtools}

 \makeatletter
\def\munderbar#1{\underline{\sbox\tw@{$#1$}\dp\tw@\z@\box\tw@}}
\makeatother

\def\p{\partial}

\def\O{\T^2}
\def\R{\mathbb{R}}
\def\d{\mathrm{d}}
\def\B{\begin{equation}}
\def\E{\end{equation}}
\def\BN{\begin{eqnarray*}}
\def\EN{\end{eqnarray*}}

\begin{document}
 \date{ \today
 }

\title{Vorticity Convergence from Boltzmann to 2D incompressible Euler equations below Yudovich class}

 \author{Chanwoo Kim}
 \address{Department of Mathematics, University of Wisconsin-Madison, Madison, WI, United States, 53706; 
 	email: ckim.pde@gmail.com 
 }
 
  \author{Joonhyun La}
 \address{ Department of Mathematics, Stanford University, Stanford, CA, United States, 94305; \newline email:joonhyun@stanford.edu
 }

 
\maketitle


\begin{abstract} 
It is challenging to perform a multiscale analysis of mesoscopic systems exhibiting singularities at the macroscopic scale. In this paper, we study the hydrodynamic limit of the Boltzmann equations
\Be\label{Boltzmann}
\mathrm{St} \p_t F  +  v\cdot \nabla_x F  = \f{1}{\mathrm{Kn}} Q(F ,F )  
\Ee  
 toward the singular solutions of 2D incompressible Euler equations whose vorticity is unbounded
  \Be
\begin{split}\label{Euler_eqtn}
	\p_t u + u \cdot \nabla_x u + \nabla_x p = 0,\ \
	\text{div }u =0.
\end{split}
\Ee
 We obtain a microscopic description of the singularity through the so-called kinetic vorticity and understand its behavior in the vicinity of the macroscopic singularity. As a consequence of our new analysis, we settle affirmatively an open problem of the hydrodynamic limit toward Lagrangian solutions of the 2D incompressible Euler equation whose vorticity is unbounded ($\o \in L^\pp$ for any fixed $1 \leq \pp < \infty$). Moreover, we prove the convergence of kinetic vorticities toward the vorticity of the Lagrangian solution of the Euler equation. In particular, we obtain the rate of convergence when the vorticity blows up moderately in $L^\pp$ as $\pp \rightarrow \infty$ (localized Yudovich class).


\hide	
	The scaled Boltzmann equation with the \textit{Strouhal number} $\mathrm{St}$ and the \textit{Knudsen number} $\mathrm{Kn}$ is given by 
\Be\label{Boltzmann}
\mathrm{St} \p_t F  +  v\cdot \nabla_x F  = \f{1}{\mathrm{Kn}} Q(F ,F ) .
\Ee  
Through a scale of 
\Be\label{scale}
\mathrm{St}=\e = \mathrm{Ma}  \  \ \text{and} \ \ 
\mathrm{Kn} =\kappa \e  \ \ \text{with} \ 
\kappa= \kappa(\e) \rightarrow 0 \  \text{as} \   \e \rightarrow 0, 
\Ee
we construct a family of Boltzmann solution for (\ref{Boltzmann}) converging to incompressible Euler equation
\begin{align}
\p_t   {u} _E+   {u}  \cdot \nabla_x  {u}_E  +\nabla_x p_E  =0, \ 
\nabla_x \cdot   {u}_E  =0
 \ \ \text{in} \ \ \O, 
\label{Euler} 
\end{align} 
\unhide
\end{abstract}

\bigskip

 \noindent\textit{\textbf{Introduction. }}One of the fundamental questions in the area of partial differential equations is the Hilbert’s sixth problem, seeking a unified theory of the gas dynamics including different levels of descriptions from a mathematical standpoint by connecting the mesoscopic Boltzmann equations to the macroscopic
fluid models that arise in formal limits. The Boltzmann equation is a fundamental model of kinetic theory for dilute collections of gas particles,
which undergo elastic binary collisions. The dimensionless form of the equation is given as integro-differential equation \eqref{Boltzmann}, where $F(t,x,v)\geq 0$ is a density distribution of particles on the phase space. Here, the \textit{Strouhal number} and \textit{Knudsen number} are denoted by $\mathrm{St}$ and $\mathrm{Kn}$, which are a ratio of the characteristic length to the characteristic time and a ratio of mean free path to the characteristic length respectively. 

The effect of binary collision between particles is described by $Q(F,F)$, which takes various forms of the nonlocal-in-velocity operator depending on the nature of particles and its intermolecular interaction (\cite{CIP}). An intrinsic equilibrium, satisfying $Q( \cdot,  \cdot) = 0$, is given by the so-called local Maxwellian associated with $(R,U,\Theta) \in \R_+ \!\times\! \R^3 \!\times\! \R_+$
\Be\label{Maxwellian}
M_{R, U ,  \Theta} (v): = \f{R}{(2\pi \Theta)^{ {3} / {2}}}\exp \left\{- \frac{|v-U|^2}{2 \Theta}\right\}.
\Ee 
The collision operator enjoys so-called the collision invariance: 
$\int Q(F,G)   \begin{bmatrix} 1 & v  & |v|^2
\end{bmatrix}  \dd v=0$ for arbitrary $F,G$.
The celebrated Boltzmann’s H-theorem (entropy $H = \int F \ln F \dd v$) reveals the entropy dissipation: $\int Q(F,F) \ln F \dd v \leq 0 $. In this paper, we consider most basic hard-sphere collision cross section:  
\Be\label{Q}
\begin{split}
 Q(F,G) (v) = \frac{1}{2} \int_{\R^3} \int_{\S^2}   |(v-v_*) \cdot \sigma| &\{ 
 F(v^\prime) G(v^\prime_*) + G(v^\prime) F(v^\prime_*)  \\
 &
 - F(v ) G(v_* ) - G(v ) F(v_* ) 
\}\dd \sigma \dd v_*, \end{split}\Ee
where postcollision velocities are denoted by $v^\prime = v- ((v-v_*) \cdot  \sigma)  \sigma$ and $v_*^\prime = v_*+ ((v-v_*) \cdot  \sigma)  \sigma$.

Besides $\mathrm{St}$ and $\mathrm{Kn}$, we introduce the \textit{Mach number} $\mathrm{Ma}$ as a size of fluctuations of $F$ around the global Maxwellian $M_{1,0,1} (v)$ of the reference state $(1, 0, 1)$. Relations between $\mathrm{St}, \mathrm{Kn}$ and $\mathrm{Ma}$ are important. Naturally, $\mathrm{Ma}$ is bounded above by $\mathrm{St}/c$ where $c$ is denoted by the speed of sound. On the other hand, the famous \textit{Reynolds number} $\mathrm{Re}$ appears as a ratio between the Knudsen number and Mach number through the von Karman relation: $1/\mathrm{Re} = \mathrm{Kn}/\mathrm{Ma}$. By passing $\mathrm{Kn}$ to zero and choosing different $\mathrm{St} (\mathrm{Kn})$ and $\mathrm{Ma}(\mathrm{Kn})$ as functions of $\mathrm{Kn}$, we can formally derive various PDEs of macroscopic variables. 
\hide\Be
\begin{bmatrix}
	\rho (t,x)\\
	{u} (t,x)\\
	\theta  (t,x)
\end{bmatrix} 
= \lim_{\mathrm{Ma} \rightarrow 0}
\f{1}{\mathrm{Ma}}\int_{\R^3}
\{F  (t,x,v) -M_{1,0,1}  \}
\begin{bmatrix}
	1 \\
	{v}\\ 
	\frac{|v|^2-3}{\sqrt{6}}
\end{bmatrix}   
\dd v  
.\label{hydro_limit}
\Ee\unhide
 Formally the incompressible Euler limit can be realized in the following scaling of large Reynolds number limit
 \Be\label{scale}
 \mathrm{St}=\e = \mathrm{Ma}  \  \ \text{and} \ \ 
 \mathrm{Kn} =\kappa \e  \ \ \text{with} \ 
 \kappa= \kappa(\e) \rightarrow 0 \  \text{as} \   \e \rightarrow 0. 
 \Ee
 
  In the diffusive scaling, the same scaling of \eqref{scale} with $\kappa=1$, the corresponding macroscopic PDE is the incompressible Navier-Stokes-Fourier system. This scaling problem is better understood as a singular perturbation in \eqref{Boltzmann} is \textit{milder} than our case \eqref{scale} (see \cite{EGKM2,Guo2006,Golse,GS2004} and references therein). In \cite{EGKM2}, Esposito-Guo-Kim-Marra establish a uniform bound of a perturbation $f$ in $F = M_{1,0,1} + \e f \sqrt{M_{1,0,1}}$ without a priori information of the fluid solutions, and hence they derive (actually construct) a strong solution of incompressible Navier-Stokes-Fourier system for both steady and unsteady cases in the presence of boundary. One of key ingredients is to obtain an $L^6_x$ ($\hookleftarrow H^1_x$ in 3D) control of $f$, by realizing a hidden elliptic equations of the bulk velocity part of $f$ in 
  \Be\label{MMSB}
   v \cdot \nabla_x f  \sim \frac{1}{\e} Lf \ \  \   \ \  \  \textit{(macro-micro scale balance)}
  \Ee
 for a linearized operator $L$ of $Q$. Unfortunately, a uniform bound of $f$ in the Euler scaling seems not feasible even in 2D without a priori information of solutions of the incompressible Euler equations, due to additional singularity in both macro-micro scale balance and nonlinear perturbation, which are major obstacles in our analysis.

 The regularity of fluid solutions plays a crucial rule in the multiscale analysis in the Euler scaling \eqref{scale}, which has been revealed differently in a modulated entropy inequality by Saint-Raymond \cite{SR}, and an asymptotic expansion by Jang-Kim \cite{JK2020}. This effect appears as an growth in the microscopic scale (see  \eqref{Lipschitz}), which resembles the famous Beale-Kato-Majda result \cite{BKM1984}. For a spatially Lipschitz continuous velocity field, Saint-Raymond proves in \cite{SR} a hydrodynamic limit toward such solutions of the incompressible Euler equations \eqref{Euler_eqtn}. It has been an \textit{open problem} to study the hydrodynamic limit toward solutions of the Euler equations, which are not spatially Lipschitz continuous such as vortex sheet solutions. Due to the transport feature of 2D Euler equations, such singular solutions have been well-understood. For compactly supported initial vorticities in $L^\pp$ for $1 < \pp < \infty$, global existence theory was first proved by DiPerna-Majda in \cite{DipMa}. Using the so-called concentration-cancellation, the result was extended for a finite measure with distinguished sign by Delort in \cite{Delort}, and $L^1$ vorticities by Vecchi-Wu in \cite{VW}. Recently, Bohun-Bouchut-Crippa construct Lagrangian solutions of $\o \in L^1$ in \cite{BBC} using a stability estimate of \cite{BC2013}.

\hide
passage for both limits mathematically in various physical situations. The first aim supports this
objective in the hydrodynamic limits when free interfaces occur naturally: (i) The vortex patch
solutions are solutions to the incompressible Euler equation with special type of initial data, whose
vorticity is an indicator function of bounded subset with smooth boundary.
\unhide

\hide
This scale is known for the one to derive the incompressible Euler system \eqref{up} and 
\begin{align}
	\p_t \theta + u \cdot \nabla_x \theta =0 , \ \ 
	\theta (t,x) + \rho (t,x)= C&  \ \    \text{for some constant $C$}.\notag
\end{align}  
In order to focus on the velocity equation we set initial data $\theta_0(x)= 0$ and $\rho_0(x)= 0$ so that 
\Be\label{rho_theta=0}
\theta(t,x) \equiv 0, \ \ \rho(t,x) \equiv 0 \ \ \text{for all } t\geq0.
\Ee

Entropy

L

Spectral gap 

$\Gamma$
\unhide

\subsubsection*{\textbf{A. Main Theorems}} We recall the main object of this paper: the scaled Boltzmann equation \eqref{Boltzmann} of the scaling \eqref{scale}  
	\Be\label{eqtn_F}
	\e \p_t F^\e  + {v}\cdot \nabla_x F^\e  = \f{1}{\kappa \e } Q(F^\e ,F^\e ) \ \ \text{in} \ [0,T] \times \T^2 \times \R^3.
	\Ee
In this paper we set that the spatial variables and velocity variables belong to 2D periodic domain and 3D whole space respectively:
 \begin{align}
 x &= (x_1,x_2) \in 	\T^2  := \left[- \frac{1}{2},\frac{1}{2}\right] \times\left[- \frac{1}{2},\frac{1}{2}\right]  \ \text{with the periodic boundary},
 \label{T2} \\
 v &= (\munderbar{v}, v_3) :=(v_1,v_2,v_3) \in \R^3   .
 \label{bar v} 
\end{align} 
The existence and uniqueness of the Boltzmann equation with fixed scaling have been extensively studied in \cite{GrSt, Guo2002, Guo2003}; initial-boundary value problem \cite{Guo2010, KL2018_1, KL2018_2}; singularity formation \cite{K2011}; boundary regularity estimate \cite{GKTT2017, CKL2019}; non-equilibrium steady states \cite{EGKM}. For the weak solution contents, we refer to \cite{DiL_B, GS2004} and the reference therein.

As the main quantities in the hydrodynamic limit, we are interested in the following observables and their convergence toward the counterparts in fluid: 
\begin{definition}[Boltzmann's macroscopic velocity and vorticity]
\Be\begin{split}\label{vorticity_B}
	u^\e_B(t,x) &= \frac{1}{\e} \int_{\R^3}   (F^\e(t,x,v) - M_{1,0,1} (v)) \munderbar{v}  \dd v,\\
	\o^\e_B(t,x) &:=  \nabla^\perp \cdot u^\e_B(t,x)
	= \Big(-\frac{\p}{\p x_2}, \frac{\p}{\p x_1}\Big) \cdot u^\e_B(t,x) .
\end{split}	\Ee 
\end{definition}

In 2D, the incompressible Euler equations \eqref{Euler_eqtn} has the vorticity formulation:
\begin{align}
	\p_t \o + u \cdot \nabla \o =0 \ \ &\text{in} \ [0,T] \times \T^2,   
	\label{vorticity}\\
	u =- \nabla^\perp ( -\Delta)^{-1} \o 
	\ \ &\text{in} \ [0,T] \times \T^2
	,\label{BS}\\
	\o|_{t=0} = \o_0
	\ \ &\text{in} \  \T^2
	.\label{vorticity_initial}
\end{align} 
 We will present the Biot-Savart formula of \eqref{BS} in the periodic box $\T^2$ at \eqref{BS_periodic}. When a velocity field is Lipschitz continuous, there exists a Lagrangian flow $X (s;t,x)$ solving 
\Be\label{dX}
\frac{d}{ds} X(s;t,x)
= u(s, X(s;t,x)),\ \
X(s;t,x)|_{s=t}=x.
\Ee
Then a smooth solution of the vorticity equation \eqref{vorticity}-\eqref{vorticity_initial} is given by 
\Be
\begin{split}\label{Lag_sol}
	\o(t,x) = \o_0 (X(0;t,x)), \ \ 	u(t,x) = -\nabla^\perp (-\Delta)^{-1} \o (t,x).
\end{split}
\Ee
Out of the smooth context, a general notion of Lagrangian flow has been introduced: 
\begin{definition}[\cite{DiL_ODE,CDe}]\label{rLf}
	Let $u \in L^1([0,T] \times  \T^2; \mathbb{R}^2)$. A map $X: [0,T] \times \T^2 \rightarrow \T^2$ is a regular Lagrangian flow of \eqref{dX} if and only if for almost every $x \in \T^2$ and for any $t \in [0,T]$, the map $s \in [0,t] \mapsto X(s;t,x) \in \T^2$ is an absolutely continuous integral solution of \eqref{dX}; and there exists a constant $\mathfrak{C}>0$ such that for all $(s, t) \in [0,t] \times [0,T]$ there holds 
	\Be\label{compression}
	\int_{\T^2} \phi (X(s;t,x)) \dd x \leq 	\mathfrak{C}
	\int_{\T^2} \phi (x) \dd x,\Ee
for every measurable function $\phi: \T^2 \rightarrow [0,\infty]$.
\end{definition}
For a given regular Lagrangian flow to \eqref{dX}, we can define the \textit{Lagrangian solution} $(u, \o)$ along the regular Lagrangian flow as in \eqref{Lag_sol}. In fact, the existence and uniqueness (for a given $u$) of the regular Lagrangian flow is proved in \cite{DiL_ODE,CDe,BC2013} as long as \eqref{BS} holds while $\o \in L^\pp$ for $\pp \geq 1$.

\textit{Our first theorem is about the convergence of $\o_B^\e$ to the Lagrangian solution $\o$, when vorticities belong to $L^\pp (\T^2)$ when $\pp <\infty$.}

\begin{theorem}[Informal statement of Theorem \ref{main_theorem1}: Strong Convergence]\label{main_theorem1_summary}
	Let arbitrary $T>0$ and $(u_0, \o_0) \in L^2 (\mathbb{T}^2) \times L^\pp (\mathbb{T}^2)$ for $\pp \geq 1$. Let $(u, \o) \in L^\infty ((0, T); L^2 (\mathbb{T}^2) \times L^\pp (\mathbb{T}^2))$ be a Lagrangian solution of 2D incompressible Euler equations \eqref{vorticity}-\eqref{vorticity_initial} with initial data $(u_0, \o_0)$. Then we construct a family of solutions to the Boltzmann equation \eqref{eqtn_F} whose macroscopic velocity and vorticity $(u_B^\e, \o_B^\e)$ of \eqref{vorticity_B} converge to the Lagrangian solution. Moreover, we have 
	\Be\notag
\o_B^\e \rightarrow \o \  \ \text{strongly in } [0,T] \times \T^2.
	\Ee 
\end{theorem}

 \begin{remark}Uniqueness of the incompressible Euler equations in 2D is only known for vorticities with moderate growth of $L^\pp$ norm as $\pp \rightarrow \infty$ by Yudovich \cite{Y1963,Y1995}. In some sense, we can view the theorem as a ``selection principle'' of a Lagrangian solution of the incompressible Euler equations from the Boltzmann equation.  
 \end{remark}

\begin{remark}Our proof does not rely on a result of inviscid limit of the nonlinear Navier-Stokes equations (cf. \cite{JK2020}) nor the higher order Hilbert expansion (cf. see the results by Guo \cite{Guo2006} and de Masi-Esposito-Lebowitz \cite{DEL}). A direct approach we develop in this paper is based on stability analysis for both the Lagrangian solutions of the inviscid fluid and the Boltzmann solutions with a new corrector. 
\end{remark}

 
\textit{Our second theorem is about the quantitative rate of convergence/stability of $\o_B^\e$ to $\o$, when the uniqueness of fluid is guaranteed.} In \cite{Y1995}, Yudovich extend his uniquness result for bounded vorticities \cite{Y1963} to the so-called localized Yudovich class, namely $ \o_0 \in Y_{\mathrm{ul}}^\Theta (\Omega)$ with certain modulus of continuity for its velocity $u$. Here, 
\Be \notag
\| \o \|_{Y_{\mathrm{ul} }^\Theta  (\O) } := \sup_{1 \le \pp < \infty} \frac{ \| \o \|_{L^\pp (\O ) } }{\Theta (\pp) } \ \ \text{for some } \Theta(\pp) \rightarrow \infty \ \text{as} \  \pp \rightarrow \infty. 
\Ee
Here, we specify $\Theta:\mathbb{R}_{+} \rightarrow \mathbb{R}_{+}$: there exists $m \in \mathbb{Z}_{+}$ such that $
\Theta (\pp) =  \prod_{k=1} ^m \log_k \pp,$ for large $\pp >1$, where $\log_k \pp$ is defined inductively by $\log_0 \pp = 1$, $\log_1 \pp =  \log \pp$, and $
\log_{k+1} \pp = \log \log_{k} \pp$. Also, we denote the inverse function of $\log_m (\pp) $ (defined for large $\pp$) by $e_m$. Finally, we note that $
\int_{e_m(1) } ^\infty \frac{1}{\pp \Theta (\pp ) } = \infty$ which turns out to be important in uniqueness of the solution.

\begin{theorem}[Informal statement of Theorem \ref{main_theorem2}: Rate of Convergence] \label{main_theorem2_summary} 
If we further assume $\o_0 \in Y_{\mathrm{ul} }^\Theta  (\O)$ in addition to Theorem \ref{main_theorem1_summary}, then 
	\Be\notag
\o_B^\e \rightarrow \o \  \ \text{strongly in }  [0,T] \times \T^2 \  \text{with an explicit rate}. 
\Ee
\hide

\eqref{vorticity_B} converges to the vorticity of the Euler equations \eqref{Lag_sol}

$ \o_0 $ be a function with moderate growth of $L^p$ norm: $\| \o_0 \|_{L^\pp} \lesssim \prod_{k=1} ^m \log_k \pp$ (\eqref{Thetap}) for $\pp \in [1, \infty)$. Let $\o$ be the unique weak solution of 2D incompressible Euler equation with initial data $(u_0, \o_0)$ which propagates up to $T$, and $u$ is given by Biot-Savart law \eqref{BS}.Then there exists a family of solutions to Boltzmann equations \eqref{eqtn_F} whose bulk density, velocity, and temperature converges to $(1, u, 1)$ in the sense that $\frac{1}{\e} \left ( F^\e (t) - M_{1, \e u, 1} \right ) \rightarrow 0$ in a suitable space (see Theorem \ref{main_theorem2}). Furthermore, bulk velocity of $F^\e$ converges in $L^\infty (0, T; W^{1, \pp}(\mathbb{T}^2 ) )$ with an explicit rate (\eqref{conv:vorticity_B_rate}) depending on $\o_0$. Moreover, if $\o_0$ has additional regularity, the rate is uniform in $\o_0$ (\eqref{conv:vorticity_B_uniform}).
\unhide
\end{theorem}


\hide

Also, convolution on $\mathbb{T}^2$ is denoted by $*$, and it means $f*g (x) = \int_{\mathbb{T}^2} f(x-y) g(y) \dd y$ as in \eqref{convolution}.
Finally, we sometimes denote our domain $\mathbb{T}^2$ by $\T^2$.

Consider the scaled Boltzmann equation \eqref{Boltzmann} with our scaling \eqref{scale}:


The goal of this paper is to construct a family of solutions to \eqref{eqtn_F} such that


In \cite{JK2020}, Jang-Kim develops a new perturbation analysis around the local Maxwellian at this scaling, and 
identify a growth integral factor (replacing $t/ \sqrt{\kappa}$ factor in the perturbation around global Maxwellian)

which resembles the famous Beale-Kato-Majda result \cite{BKM1984}.

On the other hand, the Euler scaling \eqref{scale} suffers a loss of scale from the corresponding macro-micro balance .

Compared to the diffusive scaling, the hydrodynamic limit of the Euler scaling \eqref{scale} has two major difficulties. First, as the nonlinear perturbation is more singular 

Second, the incompressible Euler equations

\cite{SR}

However, when $u$ is not Lipschitz continuous 

\unhide

\subsubsection*{\textbf{B. Novelties, difficulties and idea}} The major novelty of this paper is to establish the incompressible Euler limit \textit{in the level of vorticity} \textit{without using inviscid limit} of the Navier-Stokes equations, in the vicinity of the \textit{macroscopic singularity} ($\o \notin L^\infty (\T^2)$). We study the convergence of Boltzmann's macroscopic vorticity toward the Euler's vorticity, as interesting singular behavior, e.g. interfaces in vortex patches, can be observed only in a stronger topology of velocities. We believe this new approach will shed the light to the validity of Euler equations more direct fashion. Possible application would be direct validity proof of Euler solutions from the kinetic theory without relying on the inviscid limit results. 
In addition, we are able to allow quite far-from-equilibrium initial data (see \eqref{Hilbertexpansion1}). 

There are two major difficulties in the proof.  
First, \textit{the macroscopic solutions are singular} and their singularity appears as growth in the microscopic level (\cite{JK2020}):
	\Be\label{Lipschitz}
	\exp \Big(\int^t_0 \| \nabla_x  u (s) \|_{L^\infty_x}  \dd s\Big).
	\Ee  
This factor becomes significantly difficult to control when we study the Boltzmann solutions close to the solution of Euler equations, instead of Navier-Stokes equations. The diffusion in the bulk velocity has a considerable magnitude, and causes a singular term due to the growth of \eqref{Lipschitz}. 
	Second, \textit{the macro-micro scale balance is singular} in the Euler scaling. As the transport effect is weaker, this results the lack of a scale factor of the hydrodynamic bound in the dissipation. In fact, an integrability gain in $L^p$ ($\hookleftarrow H^1_x$ in 2D) of \cite{EGKM2} or velocity average lemma \cite{glassey} are not useful to control the singular nonlinearity. In addition, the perturbation equations suffer a loss of scale due to the commutator of spatial derivatives and the linearized operator around a local Maxwellian associated with macroscopic solutions.

		\hide


\item

 However, if we introduce less expansion, we may not able to cancel out enough singular terms, which puts a serious obstacle in closing the estimate. Furthermore, we want to construct a sequence of Boltzmann solutions which are close to the solution of Euler equation, instead of Navier-Stokes equation: this is in particular desirable since we want to observe singular behavior of bulk velocity, without blurring from viscosity effect. However, the term which causes diffusion in the bulk velocity has a considerable magnitude, although not singular. Thus, we need an idea to overcome this effect. \unhide

To overcome the difficulties, we devise a novel \textit{viscosity-canceling correction} in an asymptotic expansion of the scaled Boltzmann equations. To handle the low regularity of fluid velocity fields, we regularize the initial data with scale $\beta$ and expand the Boltzmann equations around the local Maxwellian $M_{1, \e u^\beta, 1}$ associated with the Euler solution $u^\beta$ starting from $u^\beta_0$. At first place, one may try a form of the standard Hilbert expansion:
\Be \label{Hilbertexpansion1}
\begin{split}
	M_{1, \e u^\beta, 1} + \e^2 p^\beta M_{1, \e u^\beta, 1}  - \e^2 \kappa (\nabla_x u^\beta): \mathfrak{A} \sqrt{M_{1, \e u^\beta, 1}} + \e f_R \sqrt{M_{1, \e u^\beta, 1}},
\end{split}
\Ee
by matching to cancel most singular terms. The Euler equation is in the hierarchy of $O(\e^2)$: it comes from $\e \p_t M_{1, \e u^\beta, 1}$ and correctors. However, the third term of order $\e^2 \kappa$ introduces the viscosity contribution $-\e^2 \kappa \eta_0 \Delta_x u^\beta \cdot (v - \e u^\beta) M_{1, \e u^\beta, 1}$: and comparing fo $\e f_R \sqrt{M_{1, \e u^\beta, 1}}$, we see that if this term is not canceled, then it will drive the remainder to order $O(\e \kappa)$ - which is dangerous. Note that this term is hydrodynamic, so we cannot rely on coercivity provided by $L$: it provides additional $\e \sqrt{\kappa}$ smallness only for non-hydrodynamic terms.

A simple but useful observation is that still this term is in a lower hierarchy than that of Euler equation. Thus, if we introduce an additional corrector in $\e \kappa$ level, we may cancel out viscosity contribution. Of course one needs to be careful as we introduce $\e \kappa$-size term to cancel out $\e^2 \kappa$-size term! However, by carefully choosing the form of $\e \kappa$-size corrector:
\Be \label{F_e_inftro}
F^\e = \eqref{Hilbertexpansion1} +   \e \kappa \tilde{u}^\beta \cdot (v-\e u^\beta)  M_{1, \e u^\beta, 1}
+ \e^2 \kappa \tilde{p}^\beta  M_{1, \e u^\beta, 1},
\Ee
we can actually fulfill our goal.
\begin{enumerate}
	\item $\e \kappa \tilde{u}\cdot (v-\e u^\beta) M_{1, \e u^\beta, 1 }$ is fully hydrodynamic, and therefore the most singular term coming from collision with local Maxwellian vanishes. Then the largest term coming from collision is the collision of this corrector with itself, which is of size $\e \kappa$, but being non-hydrodynamic. Thus, it is in fact, small (due to $\e \sqrt{\kappa}$ gain for non-hydrodynamic term, non-hydrodynamic source terms of $\e \sqrt{\kappa}$ drives the remainder to order $O(\e \kappa)$. )
	\item By imposing $\nabla_x \cdot \tilde{u} = 0$, we can cancel out the hydrodynamic part for $v \cdot \nabla_x (\tilde{u}\cdot (v-\e u^\beta) M_{1, \e u^\beta, 1 } )$, which is of order $\e \kappa$. Also, by introducing additional corrector at $\e^2 \kappa$ level, one can cancel out all hydrodynamic terms of $\e^2 \kappa$ level by evolution equation for $\tilde{u}$, including $\Delta_x u$. Therefore, the remaining hydrodynamic terms are of order $o(\e^2 \kappa)$, and non-hydrodynamic terms are of order $O(\e \kappa)$, and both are small.
	\item Interaction of this corrector and the remainder also turns out to be innocuous as well.
\end{enumerate}

It is worth to remark that in this corrector-based Hilbert expansion we do not need to set up $\e = \kappa$ as in usual Hilbert expansion \cite{DEL}: we only need $\e/\kappa^2 \rightarrow 0$. This is satisfactory, in the sense that a regime which is close to Navier-Stokes regime (whose $\kappa$ vanishes slowly) should be more tractable in philosophy: and indeed for such regime we can allow larger deviation from the equilibrium. In addition, we note that this expansion in fact allows even more general data than \eqref{Hilbertexpansion1}: we have additional freedom in choosing $\tilde{u}_0$, so in principle a remainder with certain part of size $\e \kappa$ is in fact admissible, while in \eqref{Hilbertexpansion1} all parts of remainder should be of size $o(\e \kappa)$. We believe that this new idea of correction would have many applications.

\subsubsection*{Notations}For the sake of readers' convenience, we list notations used often in this paper. 
\begin{align} 
	\p \ : &  \ \p f  = \ \p_{x_1} f \text{ or } \p_{x_2}f   \label{pderivative} \\
	\p^s \ : & \  \p^s f = \sum_{\alpha_1 + \alpha_2 \leq s} \p_{x_1}^{\alpha_1}  \p_{x_2}^{\alpha_2}f \label{psderivative}  \\
	f*g \ : & \ 	f*g (x) := \int_{\mathbb{T}^2} f(x-y) g(y) \dd y\label{convolution}
	\\
	f*_{\R^2 }g  \ :  & \ 	f*_{\R^2 }g (x) = \int_{\mathbb{R}^2} f(x-y) g(y) \dd y \label{convolution_R2}
	\\
	(  \ \cdot \  )
	_+ \ :&  \  (  a )_+= \max \{a, 0\}
	\label{()+}
	\\
	\log_+ \  :&  \ \log _+ \hspace{-2pt}a = \max\{ \log  a, 0\}
	\label{log+}
	\\
	\lesssim \ : & \  \text{there exists } C>0 \text{ such that }
	a \lesssim b  \  \text{implies } a \leq C b
	\label{lesssim}\\
	a \simeq b \ : & \ \text{$a$ consists of an appropriate linear combination of the terms in $b$} \label{simeq}\\
	[\![  \cdot , \cdot]\!] \ : & \  [\![ A, B]\!] g :=  A  (Bg) -B (Ag)
	\  \ \text{(commutator)} 
	\label{commutator} \\
	\| \cdot \|_{L^p_\Box}  \ : & \ 
	\| f \|_{L^p_t} = \| f \|_{L^p (0,T)}, \ 	\| f \|_{L^p_x} = \| f \|_{L^p (\T^2)}, \ 
		\| f \|_{L^p_v} = \| f \|_{L^p (\R^3)}\label{Lp}
	\\
\| \cdot  \|_{L^p_x L^2_v} \ : & \ \|f \|_{L^p_x L^2_v} :=	\| f \|_{L^p(\mathbb{T}^2 ; L^2 (\mathbb{R}^3))} = \big \| 
 \| f(x,v ) \|_{L^2(\R^3_v)}
\big \|_{L^p(\mathbb{T}^2_x) }  \label{mixedLp}\\
	d_{\mathbb{T}^2} (x,y) \ : & \ \text{geodesic distance between $x$ and $y$ in $\T^2$, often abused as $|x-y|$} 
\end{align}

 \subsubsection*{Acknowledgements}The authors thank Juhi Jang and Theo Drivas for their interest to this project. This project is supported in part by National Science Foundation under Grant No. 1900923 and 2047681 (NSF-CAREER), and Brain Pool program funded by the Ministry of Science and ICT through the National Research Foundation of Korea (2021H1D3A2A01039047). CK thanks Professor Seung Yeal Ha for the kind hospitality during his stay at the Seoul National University. The authors thanks In-jee Jeong for his kind hospitality and support.

\hide

: therefore, instead of choosing initial data as a perturbation around the local Maxwellian $M_{1, \e u_0, 1}$, we choose initial data as a perturbation around the local Maxwellian $M_{1, \e u^\beta_0, 1}$, where $u^\beta_0$ is the initial data regularization of $u_0$ with scale $\beta$.

In the hydrodynamic limit proof, we need a stability of the Euler solution under the perturbation of initial data, as well as control of remaining small terms, we can construct a sequence of Boltzmann solutions whose bulk velocity converges to Euler solution. 

It turns out that this simple idea works well: in the class of solution of Euler equation we consider, we have a certain stability, so we can prove that solution $u^\beta$ starting from $u^\beta_0$ converges to the solution $u$ from $u_0$. 

Also, for the estimate of remainders, introduction of regularization scale $\beta$ gives an additional freedom in our analysis: by sacrificing the speed of regularization convergence, we can control the size of higher derivatives appearing in the remainder equation. In addition, many weak solutions of fluid equations are interpreted as a limit of smooth solutions. In that regard, this initial data regularization approach is quite natural.

\paragraph{\textbf{3. Viscosity-canceling correction}} The final piece of analysis is a Hilbert expansion particularly designed to keep source terms in remainder equation small. At first place, one may try a Hilbert expansion of the form
\Be \label{Hilbertexpansion1}
\begin{split}
	F^\e &= M_{1, \e u^\beta, 1} + \e^2 f_2 \sqrt{M_{1, \e u^\beta, 1}} + \e f_R \sqrt{M_{1, \e u^\beta, 1}} \\
	&= M_{1, \e u^\beta, 1} + \e^2 p^\beta M_{1, \e u^\beta, 1}  - \e^2 \kappa (\nabla_x u^\beta): \mathfrak{A} \sqrt{M_{1, \e u^\beta, 1}} + \e f_R \sqrt{M_{1, \e u^\beta, 1}},
\end{split}
\Ee
by matching to cancel most singular terms. The Euler equation is in the hierarchy of $O(\e^2)$: it comes from $\e \p_t M_{1, \e u^\beta, 1}$ and correctors. However, the third term of order $\e^2 \kappa$ introduces the viscosity contribution $-\e^2 \kappa \eta_0 \Delta_x u^\beta \cdot (v - \e u^\beta) M_{1, \e u^\beta, 1}$: and comparing fo $\e f_R \sqrt{M_{1, \e u^\beta, 1}}$, we see that if this term is not cancelled, then it will drive the remainder to order $O(\e \kappa)$ - which is dangerous. Note that this term is hydrodynamic, so we cannot rely on coercivity provided by $L$: it provides additional $\e \sqrt{\kappa}$ smallness only for non-hydrodynamic, i.e. perpendicular to hydrodynamic, terms.

\paragraph{\textbf{2. Control by Energy}} To admit far-from-equilibrium initial data, we need to keep the characteristic size of remainder as large as possible. A heuristic calculation suggests that the size $o(\e \kappa)$ for the remainder is the threshold: if the remainder becomes of the size $O(\e \kappa)$, we lose the control for nonlinearity of the remainder equation. Thus, we aim to keep our characteristic size of remainder to be slightly smaller than $\e \kappa$. As the macro-micro scale balance is singular, we naturally pursue a control of the nonlinearity by the energy in $H^2_xL^2_v$ and raw dissipation.  



It turns out that this idea gives a sharp scaling
: the commutator $[\![ \p^s, L ]\!]$ between spatial derivatives and $L$ forces us to lose $\sqrt{\kappa}$ scale for each derivative, but we do not lose scale in nonlinearity for 2-dimensional domain. Thus, by setting initial data decaying to $0$ in arbitrary slow rate as $\e \rightarrow 0$, we can keep $L^2_x L^2_v$ norms of remainder and its derivatives small, provided that the source terms are also small, which is the main point of the next idea.

Furthermore, we note that $H^2_x L^2_v$ suits very well with our goal to see convergence in a stronger topology: as we can control up to second derivatives of remainder small, we can keep our Boltzmann solution close to the local Maxwellian $M_{1, \e u^\beta, 1}$. Its zeroth and first derivatives may converge - they correspond to the velocity and vorticity. Its second derivatives may blow in general, which represents the formation of singular object, e.g. interfaces. 

\unhide

\hide
\begin{theorem} We construct a family of solutions $F^\e$ to (\ref{Boltzmann}) and (\ref{diffuse_BC}) with (\ref{scale}) such that for some $T\geq 0$
\Be
 \frac{F^\e(t,x,v)- M_{1, 0, 1} (v)}{\e} \rightarrow u_E(t,x) \ \ \text{strongly in } L^\infty_t (0,T)L^2_{x,v}. 
\Ee

\end{theorem}

In this note it is important to consider the incompressible Navier-Stokes equation with $\kappa$-viscosity 
\begin{align}
\p_t u+ u \cdot \nabla_x u - \kappa \eta_0 \Delta u + \nabla_x p  &=0 \ \ \text{in} \ \O, \label{NS_k}\\
 \nabla_x \cdot u &=0 \ \ \text{in} \ \O,
 \label{incomp} 
\end{align}
where a physical constant $\eta_0$ will be fixed later. We introduce a vorticity
 \Be\label{vorticity}
 \omega = \text{curl } u= \nabla^\perp \cdot u := \p_1 u_2 - \p_2 u_1   .
 \Ee

 \unhide

 \hide
 Next we consider a key change of variable formula:
 [[INTRO]]
 \begin{lemma}
 	Fix $N>0$. For any $s\geq s^\prime \geq 0$ and $(y,\underline u) \in \O \times \{\underline u \in \R^2: |\underline u|<N\}$, the map 
 	\Be
 	s^\prime \mapsto Y(s^\prime;s,y,\underline u) 
 	\in\O
 	\Ee
 	is $m$-to-one, where $m \leq \max\big\{ \big(2N\frac{s-s^\prime}{\e } \big)^2, 1\big\}$ at most.  There is a change of variables formula: for a non-negative function $A: z \in  \O   \mapsto A(z)  \geq 0$,
 	\Be\label{COV}
 	\int_{\{\underline u \in \R^2: |\underline u|< N\}} A( Y(s^\prime;s,y,\underline u)) \dd \underline u \leq 
 	\max\Big\{
 	N^2, \frac{\e^2}{|s-s^\prime|^2}
 	\Big\}
 	\int_{\O} A( z ) \dd z.
 	\Ee
 	
 \end{lemma}
 \begin{proof}It suffices to show \eqref{COV}, while others are obvious. From $\det \big(\frac{\p Y (s^\prime; s, y, \underline u)}{\p \underline u}\big)= \frac{|s-s^\prime|^2}{\e^2}$, 
 	\Be
 	\begin{split}\notag
 		&	\int_{\{\underline u \in \R^2: |\underline u|< N\}} A( Y(s^\prime;s,y,\underline u)) \dd \underline u\\
 		&\leq  \max\Big\{ \big(2N\frac{s-s^\prime}{\e } \Big)^2 , 1 \big\}\int_{\O} A(z) \frac{\e^2}{|s-s^\prime|^2} \dd z\\
 		&=	\max\Big\{
 		N^2, \frac{\e^2}{|s-s^\prime|^2}
 		\Big\}
 		\int_{\O} A( z ) \dd z.
 	\end{split}
 	\Ee
 	%
 \end{proof}
 
 \unhide

\tableofcontents

\section{Approximation of the Lagrangian solutions of the Euler equations}

As discussed in the introduction, we would like to obtain a limit to weak solutions, which does not have enough regularity in the framework of the standard Hilbert expansion in general, Moreover, we want a convergent sequence in a stronger topology than $L^\pp$ for velocity, as interesting singular behavior can be observed only in a stronger topology. However, control in stronger topology requires more regularity for velocity field as well. A straightforward remedy for low regularity of fluid velocity field is to regularize the initial data: therefore, instead of choosing initial data as a perturbation around the local Maxwellian $M_{1, \e u_0, 1}$, we choose initial data as a perturbation around the local Maxwellian $M_{1, \e u^\beta_0, 1}$, where $u^\beta_0$ is the initial data regularization of $u_0$ with scale $\beta$. Then if one can prove stability of the Euler solution under the perturbation of initial data, as well as control of remaining small terms, we can construct a sequence of Boltzmann solutions whose bulk velocity converges to Euler solution. 

It turns out that this simple idea works well: in the class of solution of Euler equation we consider, we have a certain stability, so we can prove that solution $u^\beta$ starting from $u^\beta_0$ converges to the solution $u$ from $u_0$. Also, for the estimate of remainders, introduction of regularization scale $\beta$ gives an additional freedom in our analysis: by sacrificing the speed of regularization convergence, we can control the size of higher derivatives appearing in the remainder equation. In addition, many weak solutions of fluid equations are interpreted as a limit of smooth solutions. In that regard, this initial data regularization approach is quite natural.

\hide
Here, there exists a smooth periodic function $R$ such that 
\Be\label{BS_periodic}
u(x)= \mathbf{b} * w = 
\Big( \frac{1}{2\pi} \frac{ x  ^\perp}{|x |^2} +  \nabla^\perp _xR \Big)
*  \o  .
\Ee
We record an explicit formula of $R$ in \eqref{reg_G}.  {\color{red}[[CK: Add the form of $R$ in the appendix]]}
\unhide

 \hide
Here, with $\mathbf{b}(x ) = \frac{1}{2\pi} \frac{ x  ^\perp}{|x |^2}$ for $x \in \R^2$,
\Be\label{b_star}
\mathbf{b} \star \o (x) := \int_{\R^2} \mathbf{b} (x-y) \o (y) \dd y,
\Ee
where $\o (y)$ should be understood as a periodic function in $\{y \in \R^2\}$. \unhide

\subsection{Regularization}
In our proof of the hydrodynamic limit from the Boltzmann equations, it is important to regularize the Largangian solutions of the Euler equation \eqref{vorticity}. We achieve this by regularizing the initial data using the standard mollifier. \hide Define a convolution on $\mathbb{T}^2$: 
\Be
 	f*g (x) := \int_{\mathbb{T}^2} f(x-y) g(y) \dd y. \label{convolution}
\Ee \unhide
\hide Finally, we sometimes denote our domain $\mathbb{T}^2$ by $\T^2$.  
\begin{align}
	\p_t \o + u \cdot \nabla \o =0 \ \ &\text{in} \ [0,T] \times \T^2,   
	\label{vorticity}\\
	u =- \nabla^\perp ( -\Delta)^{-1} \o 
	\ \ &\text{in} \ [0,T] \times \T^2
	,\label{BS}\\
	\o|_{t=0} = \o_0
	\ \ &\text{in} \  \T^2
	.\label{vorticity_initial}
\end{align} \unhide
Let $\varphi \in C_c ^\infty (\mathbb{R}^2)$ be a smooth non-negative function with $\int_{\mathbb{R}^2} \varphi(x) dx = 1$ and $\varphi(x) = 0 
$ for $|x- (0,0)|\geq \frac{1}{4}$. For $\beta \in (0,1)$, we define 
\be\label{mollifier}
\varphi^\beta(x) :=  \frac{1}{\beta^2} \varphi \big( \frac{x}{\beta}\big) \ \text{ for }  \ x \in \left [ -\frac{1}{2}, \frac{1}{2} \right ]^2.
\ee
Note that $\varphi^\beta$ can be extended periodically so that $\varphi^\beta \in C^\infty (\mathbb{T}^2)$ and $\int_{\mathbb{T}^2} \varphi^\beta (x) dx = 1$ as well. Also, $\varphi^\beta$ is supported on $B_{\frac{\beta}{4} } (0)$. 
Note that $\{\varphi^\beta\}_{\beta}$ are approximate identities: thus
, for $1 \le p < \infty$ and $\psi \in L^p (\mathbb{T}^2)$, we have
\Be\label{mol->1}
\lim_{\beta \rightarrow 0} \| \varphi^\beta * \psi - \psi \|_{L^p (\mathbb{T}^2) } = 0.
\Ee 
Note that we cannot expect a universal rate of convergence, which is independent of $\psi$ if $\psi$ is merely in $L^p(\mathbb{T}^2)$ or $p=\infty$. However, if we have a certain regularity for $\psi$, we have the rate of convergence: for example, if $\psi \in W^{1,2} (\O)$, we have
\Be\begin{split}\label{mol->1_rate}
 &\|\varphi^\beta * \psi - \psi \|_{L^2 (\mathbb{T}^2)} 
= \left ( \int_{\mathbb{T}^2} \left | \int_{\mathbb{T}^2} \varphi^\beta(y) ( \psi(x-y) -  \psi(x) ) \dd y \right |^2 \dd x \right )^\frac{1}{2} \\
 &\le  \int_{\mathbb{T}^2} |\varphi^\beta (y) | \left ( \int_{\mathbb{T}^2} |  \psi(x-y) -  \psi(x) |^2 dx \right )^{\frac{1}{2} } dy \\
 & \le C \int_{\mathbb{T}^2} |y| |\varphi^\beta (y) | dy \|  \psi \|_{W^{1,2} (\mathbb{T}^2 ) }  \le C \beta \|  \psi \|_{W^{1,2} (\mathbb{T}^2 ) }.
\end{split} \Ee


	We consider approximation solutions $(u^\beta, \o^\beta)$ for the mollified initial data: 
\begin{align}
	\p_t \o^\beta + u^\beta \cdot \nabla \o^\beta =0
	\ \ &\text{in} \ [0,T] \times \T^2,
 \label{vorticity_eqtn}\\
		u^\beta = - \nabla^\perp (-\Delta)^{-1}\o^\beta 
		\ \ &\text{in} \   [0,T] \times \T^2,
	 \label{BS_beta}\\
	\o^\beta|_{t=0}= \o^\beta_0 := \varphi^\beta  * \o_0
	\ \ &\text{in} \   \T^2 
	.\label{vorticity_beta} 
\end{align}
Note that, for each $\beta\in (0,1)$ this problem \eqref{vorticity_eqtn}-\eqref{vorticity_beta} has a smooth (therefore unique) solution, which is the Lagrangian solution:  
\begin{align}
	&\o^\beta (t,x) = \o_0^\beta (X^\beta (0;t,x)),\label{Lag_beta}\\ 
	&	\frac{d}{ds} X^\beta (s;t,x) = u^\beta(s, X^\beta (s;t,x) ),   \ \ X^\beta (s;t,x) |_{s=t} = x.
	\label{ODE:X_beta}
\end{align}

\begin{remark}
	If $u^\beta$ is obtained from \eqref{BS} with $\omega^\beta \in C^\infty (\T)$, $u^\beta$ is incompressible and thus associated flow $X^\beta$ by \eqref{dX} satisfies \eqref{compression} with an equality and $\mathfrak{C} = 1$ (measure-preserving).
\end{remark}

We define a pressure as a unique solution of $- \Delta p^\beta = \text{div}(\text{div}(u^\beta \otimes u^\beta))$ with $\fint_{\O} p^\beta=0$. Then we have 

\Be \label{up} 
\begin{split}
(\p_t + u^\beta \cdot \nabla_x) u^\beta  + \nabla_x p^\beta = 0 	\ \ &\text{in} \ [0,T] \times \T^2, \\
\nabla_x \cdot u^\beta  = 0  	\ \ &\text{in} \ [0,T] \times \T^2, \\
u^\beta (x,0)  = u_0^\beta (x) 	\ \ &\text{in} \  \T^2.
\end{split}  
\Ee

Also, we will consider the following auxiliary linear equation.
\Be \label{uptilde} 
\begin{split}
(\p_t + u^\beta \cdot \nabla_x) \tilde{u}^\beta  + \tilde{u}^\beta \cdot \nabla_x u^\beta + \nabla_x \tilde{p}^\beta - \eta_0 \Delta_x u^\beta =0	\ \ &\text{in} \ [0,T] \times \T^2,\\
\nabla_x \cdot \tilde{u}^\beta  = 0	\ \ &\text{in} \ [0,T] \times \T^2, \\
\tilde{u}^\beta(0, x)  = \tilde{u}_0 (x) 	\ \ &\text{in} \  \T^2.
\end{split}  
\Ee
Here $\eta_0$ is given by Lemma \ref{Ainn}.

\subsection{Biot-Savart law in a periodic domain}

In this part, we discuss the asymptotic form of kernel for Biot-Savart law which gives $u$ from $\o$, and the singular integral which gives $\nabla_ x u $ from $\o$ in our setting, the periodic domain $\mathbb{T}^2 = \left [-\frac{1}{2}, \frac{1}{2} \right ]^2$. This is important, since the compactness results we have used, in particular \cite{BC2013}, have the $\mathbb{R}^N$ setting: in particular, the key estimate, weak $L^1$ estimate for $\nabla_x u$ relies on the form of Calderon-Zygmund kernel of Riesz transform. Therefore, we need an asymptotic form of Biot-Savart kernels and Riesz transforms. 

We start from \cite{CO2007}: 
\begin{proposition}[\cite{CO2007}, Lemma 1]
	The function $G$, defined on $\mathbb{R}^2 \simeq \mathbb{C}$ by
	\Be \label{Greentorus}
	\begin{split}
		G(z) &:= \mathrm{Im} \left ( \frac{|z|^2 - z^2}{-4 i } - \frac{z}{2} + \frac{i}{12} \right )  \\
		& - \frac{1}{2\pi} \log  \left |  \left (1 - e (z) \right )\times \prod_{n=1} ^\infty \left (1 - e(n i+z ) \right )\left ( 1 - e(ni - z) \right )   \right |,
	\end{split}
	\Ee
	where $e(z) = e^{2 \pi i z}$, is $\mathbb{Z}^2$-periodic and is the Green's function with mass on $\mathbb{Z}^2$, that is, 
	\Be
	\begin{split}
		- \Delta_x G(x)  = \sum_{\zeta \in \mathbb{Z}^2} \delta( x - \zeta )- 1 \text{ for } x \in \mathbb{R}^2, \ \
		\int
		G(x) dx  = 0.
	\end{split}
	\Ee
\end{proposition} 
		In particular, the infinite product inside converges absolutely and $G$ is of the form
	\Be \label{Greenform}
	G(z) = \frac{|z|^2}{4} - \frac{1}{2\pi}\log |\mathfrak{h}(z) |,
	\Ee
	where $\mathfrak{h}$ is a holomorphic function with simple zeros exactly on $\mathbb{Z}^2$.
	
	
	For the sake of completeness, we briefly reason \eqref{Greenform}. We recall the following result from complex analysis:
	\begin{proposition}[Theorem 15.5 of \cite{RCA}]
		Suppose that $\{ g_n \}$ is a sequence of non-zero holomorphic functions on $\mathbb{C}$ such that
		\Be
		\sum_{n=1} ^\infty |1 - g_n (z) |
		\Ee
		converges uniformly on compact subsets of $\mathbb{C}$. Then the product 
		\Be
		g(z) = \prod_{n=1} ^\infty g_n (z)
		\Ee
		converges uniformly on compact subsets of $\mathbb{C}$, and thus $g$ is holomorphic on $\mathbb{C}$. Furthermore, the multiplicity of $g$ at $z_0$ (i.e. the smallest nonnegative integer $k$ such that $\lim_{z \rightarrow z_0} \frac{g(z)}{(z-z_0)^k} \ne 0$) is the sum of multiplicities of $g_n$ at $z_0$.
	\end{proposition}
	Now we see that $\mathfrak{h}(z)$ is the product of $1 - e(z) = 1 - e^{2 \pi i z}$, $1 - e(ni + z) = 1 - e^{-2 \pi n + 2 \pi i z}$, and $1 - e(ni - z) = 1 - e^{-2 \pi n - 2 \pi i z} $. Note that $| 1 - (1 - e (ni+z) ) |= | 1 - (1 - e(ni - z) ) | = e^{- 2 \pi n}$ so the premise of the proposiiton is satisfied. Thus $\mathfrak{h} (z)$ is holomorphic. Furthermore, the zeros of $\mathfrak{h}$ is exactly the union of zeros of $1 - e(z)$, which is $\{ m i | m \in \mathbb{Z} \}$, zeros of $1 - e (ni + z) $, which is $ \{ m - ni | m \in \mathbb{Z} \}$, and zeros of $1 - e(ni - z) $, which is $\{ m + ni | m \in \mathbb{Z} \}$, for each integer $n \ge 1$. The union os exactly $\mathbb{Z}^2$. Moreover, the multiplicity of each point in $\mathbb{Z}^2$ is 1, in other words, all roots are simple.
	

Thus, on $\mathbb{R}^2 \setminus \mathbb{Z}^2$, $G$ is infinitely differentiable. Furthermore, let $\zeta \in \mathbb{Z}^2$. Then there exists a $\mathfrak{r}_\zeta > 0$ such that 
\Be
\mathfrak{h} (z) = (z - \zeta) \mathfrak{H} (z),
\Ee
where $\mathfrak{H} (z) = \frac{\mathfrak{h} (z) }{z-\zeta}$ is an holomorphic function on $B_{\mathfrak{r}_\zeta} (\zeta)$ and $\inf_{z \in B_{\mathfrak{r}_\zeta} (\zeta)} |\mathfrak{H} (z) | \ge c_\zeta > 0$. Therefore, we can rewrite \eqref{Greenform} in the following form and differentiate: for $z \in B_{\mathfrak{r}_\zeta} (\zeta)$,
\Be
\begin{split}
	G(z) &= -\frac{1}{2\pi} \log | z - \zeta| + \mathfrak{B}_\zeta (z), \\
	\nabla G(z) &= -\frac{1}{2\pi} \frac{z - \zeta} {|z- \zeta |^2} + \nabla  \mathfrak{B}_\zeta (z), \\
	\nabla^2 G(z) &= \frac{1}{4\pi} \frac{ (z-\zeta) \otimes (z-\zeta) - \frac{1}{2} |z-\zeta|^2 \mathbb{I}_2 }{|z-\zeta|^4} + \nabla^2  \mathfrak{B}_\zeta (z),
\end{split}
\Ee
where $z = x+i y$ is identified with $(x,y)$, $\nabla = (\partial_x, \partial_y )$, and
\Be
 \mathfrak{B}_\zeta (z) = \frac{|z|^2}{4} - \frac{1}{2\pi} \log | \mathfrak{H} (z) |
\Ee
is a smooth function (in $x, y$) whose all derivatives are bounded. In particular, taking $\zeta = 0 = (0,0)$ and taking $\mathfrak{r} = \mathfrak{r}_0$, we have the following:
\begin{proposition} \label{PeriodicKernel}
	Let $G$ be defined by \eqref{Greentorus}, so that the solution to Poisson equation $- \Delta_x q = h - \int_{\mathbb{T}^2} h$ is given by $q = G * h$, and the Biot-Savart law by 
	\Be\label{BS_periodic}
u(x)= \mathbf{b} * w = 
\Big( \frac{1}{2\pi} \frac{ x  ^\perp}{|x |^2} +  \nabla^\perp _x \mathfrak{B} \Big)
*  \o  .
\Ee
Then there exists a $\mathfrak{r}>0$ such that $G, \nabla_x G, \nabla_x^2 G$ are smooth and bounded in $\mathbb{T}^2 \setminus B_{\mathfrak{r} } (0) = \left [ -\frac{1}{2}, \frac{1}{2} \right ]^2 \setminus B_{\mathfrak{r} } (0)$ and in $B_{\mathfrak{r} } (0)$ we have
	\Be
	\begin{split}
		G(x) &= -\frac{1}{2\pi} \log |x| +  \mathfrak{B}(x), x \in B_{\mathfrak{r}} (0) , \\
		\nabla_x G(x) &= -\frac{1}{2\pi} \frac{x}{|x|^2} + \nabla_x  \mathfrak{B}(x), x \in B_{\mathfrak{r}} (0), \\
		\nabla_x^2 G(x) &= \frac{1}{4\pi} \frac{x\otimes x - \frac{1}{2} |x|^2 \mathbb{I}_2 }{|x|^4} + \nabla_x^2  \mathfrak{B}(x), x \in B_{\mathfrak{r}} (0),
	\end{split}
	\Ee
	where $\nabla_x ^k  \mathfrak{B}$ are bounded in $B_{\mathfrak{r}} (0) $ for all $k\ge 0$.
\end{proposition}


\subsection{Higher Regularity of the Approximations $(u^\beta,\o^\beta)$}

In this section we establish the regularity estimate of $(u^\beta, \o^\beta)$ solving \eqref{up} and \eqref{vorticity_eqtn}-\eqref{vorticity_beta} and $(\tilde{u}^\beta, \tilde{p}^\beta)$ solving \eqref{uptilde}.

First we prove that, for $1 \leq r, p \leq \infty$,
\begin{align}
\| \o_0^\beta \|_{L^r (\T^2)} &  
\lesssim \beta^{-2 \big( \frac{1}{p} - \frac{1}{r}\big)_{\hspace{-2pt}+}}
\| \o_0 \|_{L^p},
\label{growth:o_ell}\\
\|  \nabla^k\o _0^\beta \|_{L^r (\T^2)} & 
\lesssim \beta^{- k -2 \big( \frac{1}{p} - \frac{1}{r}\big)_{\hspace{-2pt}+}} \| \o_0 \|_{L^p}  .\label{growth:Do_ell}
\end{align}
From the Young's inequality, for $1+1/r = 1/p+ 1/q$ and $r, p, q \in [1, \infty]$,
\Be\notag
\| \o_0^\beta \|_{L^r (\T^2)} \leq \| \varphi^\beta \|_{L^q (\T^2)} \| \o_0 \|_{L^p (\T^2)}  
\lesssim \beta^{-2 \big( \frac{1}{p} - \frac{1}{r}\big)}
  \| \o_0 \|_{L^p} \ \ \text{for }  r \geq p, 
\Ee
\Be\notag
\|  \nabla^k\o _0^\beta \|_{L^r (\T^2)} \leq  \| \nabla^k \varphi^\beta \|_{L^q (\T^2)} \| \o_0 \|_{L^p (\T^2)}  
\leq \beta^{- k -2 \big( \frac{1}{p} - \frac{1}{r}\big)}   \| \o_0 \|_{L^p} \ \ \text{for }  r \geq p.
\Ee
For both, we have used 
%
\Be\notag
\left(\int_{\T^2} | \nabla^k_x\varphi^\beta|^q \dd x\right)^{1/q}= \left(  \frac{\beta^2}{\beta^{q(2+k)}}\int_{\T^2} |\nabla^k \varphi (\frac{x}{\beta})|^q \dd \frac{x_1}{\beta}
\dd \frac{x_2}{\beta}
\right)^{1/q}
= \beta^{ -k- \frac{
			2(q-1)	
		}{q}} \| \nabla^k \varphi \|_{L^q (\T^2)}.
\Ee
Using $|\T^2| =1$, we have  
\begin{align*}
\| \o_0^\beta \|_{L^r (\T^2)} &\leq \| \o_0^\beta \|_{L^p (\T^2)} \lesssim \| \o_0 \|_{L^p (\T^2)}  \ \ \text{for }  p \geq r,
\\
\| \nabla^k \o_0 ^\beta \|_{L^r (\T^2)} &\leq \| \nabla^k \o_0 ^\beta \|_{L^p (\T^2)} \lesssim 
 \beta^{- k} \| \o _0 \|_{L^p (\T^2)}  \ \ \text{for }  p \geq r .
\end{align*} 
Collecting the bounds, we conclude \eqref{growth:o_ell} and \eqref{growth:Do_ell}

\subsubsection{Bounds for $\| \nabla_x u^\beta (t) \|_{L^\infty(\mathbb{T}^2 ) }$}

\begin{theorem}\label{thm:Dptu}
Let $(u^\beta,\o^\beta)$ be the Lagrangian solution of \eqref{Lag_beta} supplemented with \eqref{ODE:X_beta} and \eqref{BS_beta}. For $\pp \in [1, \infty]$ and $\beta  \ll \color{black}{} \| \o_0 \|_{L^\pp}$, we have the following estimate for all $t\geq0$,
	\begin{align}
		\| \nabla u^\beta (t, \cdot )\|_{L^\infty} 
		&	\lesssim  
		\mathfrak{Lip}(\beta, \pp) :=
		\Big(  \beta^{- \frac{2}{\pp}} 
		\log_+ \hspace{-2pt}  \frac{1}{\beta}
		\Big)  \| \o_0 \|_{L^\pp}   e^{ t C\color{black}{} \beta^{- \frac{2}{\pp}}\| \o_0 \|_{L^\pp} }  {\text{ for some } C>1}\color{black}{}.\label{est:Du} 
	\end{align}
\end{theorem}

We will estimate $\nabla_x X$ by applying the Gronwall's inequality to the differentiation of \eqref{ODE:X_beta}:  
	\Be\label{dspX}
		\frac{d}{ds} \nabla_x X^\beta (s; t, x)= \nabla_x X(s; t, x) \cdot (\nabla_x u) (s, X(s; t, x) ). 
	\Ee

The initial condition for each purely spatial derivative can be driven from \eqref{dX}: 
\Be\label{initial:DX}
\nabla_x X(s;t,x)|_{s=t} = id   
\Ee

 We use a following version of Gronwall's inequality.  
\begin{lemma}[\cite{BCD2011}, Lemma 3.3]\label{lem:gronwall}
Let $q$ and $z$ be two $C^0$ (resp. $C^1$) nonnegative functions on $[t_0, T]$. Let $\mathcal{G}$ be a continuous function on $[t_0, T]$. Suppose that, for $t \in[t_0, T]$,
\Be
\frac{d}{dt} z(t) \le \mathcal{G}(t) z(t) + q(t).
\Ee
For any time $t \in [t_0, T]$, we have
\Be
z(t) \le z(t_0) \exp \left (\int_{t_0} ^t \mathcal{G} (\tau) d\tau \right ) + \int_{t_0} ^t q(\tau) \exp \left ( \int_\tau ^t  \mathcal{G} (\tau') d\tau ' \right ) d \tau.
\Ee
\end{lemma}

\color{black}{}

\begin{lemma}\label{lem:pX} For any $r \in [1,\infty]$ and $0 \leq s \leq t$, 
	\begin{align}\label{est:DX}
			\|  \nabla_x X^\beta (s;t,\cdot ) \|_{L^r (\T^2)}  \leq
		e^{\int_s ^t \|\nabla_x u (t') \|_{L^\infty_x} \dd t'	 }.
		\end{align}
\end{lemma}
\begin{proof}The proof is immediate from the Gronwall's inequality to \eqref{dspX} and the initial condition $\|\nabla X^\beta (t;t,x)\|_{L^r (\T^2)}= \|\nabla x\|_{L^r (\T^2)}= \|id\|_{L^r (\T^2)}=1$ from \eqref{initial:DX}.
\end{proof}

Next, using the Morrey's inequality 
\Be\label{Sob_emb}
W^{1, r} (\T^2) \subset C^{0,1- \frac{2}{r}} (\T^2)  \ \ \text{ for } \   r>2.
\Ee 
we estimate the Holder seminorm of $\o^\beta$.
\begin{lemma}\label{lem:estH:o}
For $r\in (2,\infty) $, 
	\Be
	\begin{split}\label{estH:o}
		[\o^\beta (t, \cdot)]_{C^{0,1- \frac{2}{r}} (\T^2)} \lesssim 
		\beta^{-1-2   \big( \frac{1}{\pp} - \frac{1}{r}   \big)_{\hspace{-2pt} +}}  \| \o_0 \|_{L^\pp (\T^2)} e^{\left (1- \frac{2}{r} \right )\int_0 ^t \| \nabla_x u^\beta (t') \|_{L^\infty_x} \dd t'}.
	\end{split}
	\Ee
\end{lemma}
\begin{proof}
We note that
\Be
\begin{split}
[ \o^\beta (t, \cdot ) ]_{C^{0, 1 - \frac{2}{r} }  (\T^2)} &= \sup_{x\ne y \in \T^2} \frac{|\o_0 ^\beta (X^\beta(0; t, x) ) - \o_0 ^\beta (X^\beta (0; t, y) ) | }{|x-y|^{\left (1 - \frac{2}{r} \right ) } } \\
&\le [\o_0 ^\beta ]_{C^{0, 1 - \frac{2}{r} }  (\T)} \| \nabla_x X^\beta (0; t, \cdot ) \|_{L^\infty_x} ^{\left (1 - \frac{2}{r} \right )},
\end{split}
\Ee
where we slightly abused the notation by
\Be \label{geodistance1}
|x-y| = \mathrm{dist}_{\mathbb{T}^2} (x,y).
\Ee
Applying Morrey's inequality \eqref{Sob_emb} for $[\o_0 ^\beta]_{C^{0, 1 - \frac{2}{r} }  (\T^2)}$ and applying \eqref{est:DX} gives the result.
\end{proof}

The following standard estimate is important in the proof: 
 \begin{lemma}
Let $(u^\beta, \o^\beta)$ satisfy \eqref{BS_beta}.
Then, for any $\gamma>0$,
\Be\label{potential1_k}
\| \nabla_x   u \|_{L^\infty (\O)} \lesssim 1+ 
\|  \o \|_{L^1(\O)} + \|  \o \|_{L^\infty(\O) }
\log_+ ( [ \o]_{C^{0,\gamma} (\O)} )
. \Ee\hide\Be\label{potential_k}
\| \nabla \p^\alpha  u \|_{L^\infty (\O)} \lesssim 
\| \p^\alpha  \o \|_{L^1(\O)} + \| \p^\alpha \o \|_{L^\infty(\O) }
\log   
\left(e+ \frac{[\p^\alpha \o]_{C^{0,\gamma} (\O)}}{\| \p ^\alpha \o \|_{L^\infty(\O)}}
\right). 
\Ee 
 \Be 
B\log 
\left( e+ \frac{A}{B}
\right)
\lesssim B\log ( eA+ B)
-B \log B \lesssim B \max ( \log_+ A, \log_+ B)
.
\Ee\unhide\end{lemma}
\begin{proof} The result is well known from the potential theory (e.g. \cite{Rein}) so we just briefly sketch the proof. Assume that $\o \in L^1(\O) \cap C^{0,\gamma} (\O)$. From \eqref{BS} and \eqref{BS_periodic}, for $R\geq d>0$, there exists $C_2>0$ only depending on the spatial dimension (2 in our case), 
		\Be
		\begin{split}
	\p_{x_j} 	u_i (x) =& \ 
	\int_{|x-y| \geq R} \p_{j} \mathbf{b}_i (x-y) \o (y) \dd y
	+ 	\int_{d \leq |x-y| \leq R} \p_{j} \mathbf{b}_i (x-y) \o (y) \dd y
	\\
	&
	+ 	\int_{|x-y| \leq d}\p_j \mathbf{b}_i (x-y)  [ \o (y) - \o(x)] \dd y +C_2 \delta_{i+1, j} \o(x) ,
		\label{exp:Du}
		\end{split}
		\Ee
		for 
		\Be\label{b_j}
		\p_j \mathbf{b} (x-y) :=  \frac{1}{2\pi }\bigg(	\frac{2 (x_{i+1} - y_{i+1}) (x_j - y_j)}{|x-y|^4} - \frac{\delta_{i+1, j}}{|x-y|^2}\bigg)
		+ \p_j \mathfrak{B}(x-y)
		.
		\Ee
		 Here, the index $i+1$ should be understood on a modulus of 2; and $\delta_{i+1,j}=1$ if $i+1=j$ mod 2 and $\delta_{i+1,j}=0$ if $i+1\neq j$ mod 2. We bound $\eqref{exp:Du}$ as 
		\Be\begin{split}\label{exp:Du_decomp}
	| \eqref{exp:Du}| 	& \leq \int_{|x-y|\geq R} \frac{4}{|x-y|^2} |\o (y)| \dd y
	+ \int_{d \leq |x-y|\leq R} \frac{4}{|x-y|^2} |\o (y)| \dd y\\
	&+ [\o]_{C^{0, \gamma} (\O)} \int_{|x-y| \leq d } \frac{4}{|x-y|^{2-\gamma}} \dd y 
	+ C_2 |\o (x) |  \\
	& \lesssim R^{-1/2} \| \o \|_{L^1(\O)} + \ln \left(\frac{R}{d} \right) \| \o \|_{L^\infty(\O)} 
	+ d^\gamma [\o]_{C^{0, \gamma} (\O)}    + \| \o \|_{L^\infty(\O)}.
					\end{split}\Ee
				We finalize the proof by choosing $R=1$ and $d=\max\big(1,  [\o]_{C^{0,\gamma} (\O)}^{1/\gamma}\big)$. 
 	\end{proof}

\begin{proof}[\textbf{Proof of Theorem \ref{thm:Dptu}}] 
	To prove \eqref{est:Du}, we apply $\eqref{growth:o_ell}|_{r=1, \infty}$ and $\eqref{estH:o}|_{r>2}$ to $\eqref{potential1_k}|$ to conclude that 
	\Be\begin{split}\label{est1:Du}
		&\| \nabla u^\beta (t, \cdot )\|_{L^\infty}/ \| \o_0 \|_{L^\pp}\\
		&\lesssim 1 + \beta^{- \frac{2}{\pp}}
		\log_+ (
		\beta^{-1 -2   (\frac{1}{\pp} - \frac{1}{r})_+} \| \o_0 \|_{L^\pp} e^{\int_0 ^t \| \nabla_x u^\beta (s) \|_{L^\infty_x} \dd s }
		)\\
		&\lesssim 1+ \beta^{- \frac{2}{\pp}} 
		\Big\{
		\log_+ \hspace{-2pt} \frac{1}{\beta} 
		+   \log_+ \hspace{-2pt} \| \o_0 \|_{L^\pp}
		+   \int^t_0\| \nabla u^\beta(s, \cdot ) \|_{L^\infty} \dd s\Big\}.
	\end{split}\Ee
Applying Gronwall's inequality gives the result.
\end{proof}

\subsubsection{Bounds for $V(\beta)$}

We introduce the growth-of-estimate function for $(u^\beta, p^\beta, \tilde{u}^\beta, \tilde{p}^\beta )$, which is a function of $\beta$:
\Be \label{Vbeta}
\begin{split}
V(\beta) :=   \sum_{s_1 + s_2 \le 2, D \in \{\p_t, \p\} }& \| \p^{s_1} D (u^\beta, \p u^\beta, p^\beta, \tilde{u}^\beta, \tilde{p}^\beta) \|_{L^\infty_{t,x} } \\
& \times  \left ( 1 + \| \p^{s_2} (\tilde{u}^\beta, u^\beta ) \|_{L^\infty_{t,x} } \right ) (1 + \sum_{j\le 2} \| \p^j u^\beta \|_{L^\infty_{t,x}} )^2.
\end{split}
\Ee

This is a pointwise bound for all derivatives of $(u^\beta, p^\beta, \tilde{u}^\beta, \tilde{p}^\beta)$ appearing in the remainder estimates in section \ref{sec:re}. 

We have a following explicit bound for $V(\beta)$.

\begin{theorem} \label{Vbound}
Suppose that $\o_0 \in L^\pp (\mathbb{T}^2)$. Then 
$$V(\beta) \lesssim \left (\| \tilde{u}_0 \|_{H^6(\mathbb{T}^2 ) }  + TU(\beta, \pp) e^{T U(\beta, \pp ) } + U(\beta , \pp) \right )^6,$$ 
where $U(\beta)$ is as defined in \eqref{uutildeb}.
\end{theorem}
\begin{proof}
By Sobolev embedding, and the formula for $p^\beta$, $\tilde{p}^\beta$, $\p_t {u}^\beta$, and $\p_t \tilde{u}^\beta$, we have a bound
\Be\notag
V(\beta) \lesssim  \left (\| u^\beta  \|_{L^\infty ((0, T); H^8(\mathbb{T}^2 ) )}  +\| \tilde{u}^\beta  \|_{L^\infty ((0, T); H^6(\mathbb{T}^2 ) )} \right ) ^6
\Ee
We invoke the standard energy, commutator estimate and algebra property of $H^s(\mathbb{T}^2)$, $s > 1$: 
\Be
\begin{split}\notag
&\frac{\dd}{2 \dd t} \| \p^8 u^\beta (t) \|_{L^2 (\mathbb{T}^2 ) }^2 \\
&\le \| \p^8 u^\beta (t) \|_{L^2 (\mathbb{T}^2 ) } \| [\![ \p^8, u^\beta \cdot \nabla_x ]\!] u \|_{L^2 (\mathbb{T}^2)} \lesssim \| \nabla_x u^\beta \|_{L^\infty (\mathbb{T}^2 ) }\| \p^8 u^\beta (t) \|_{L^2 (\mathbb{T}^2 ) } ^2 , \\
&\frac{\dd}{2 \dd t} \| \p^6 \tilde{u}^\beta (t) \|_{L^2 (\mathbb{T}^2 ) }^2\\
 & \lesssim \| \p^6 \tilde{u}^\beta (t) \|_{L^2 (\mathbb{T}^2) } \\
& \ \   {\tiny   \times \bigg( \| [\![ \p^6, u^\beta \cdot \nabla_x  ]\!] \tilde{u} (t) \|_{L^2 (\mathbb{T}^2 ) } + \| \p^6 \tilde{u}^\beta (t) \|_{L^2 (\mathbb{T}^2 ) } \| \p^7 u^\beta (t) \|_{L^2 (\mathbb{T}^2 ) } + \| \p^8 u^\beta (t) \|_{L^2 (\mathbb{T}^2 ) } \bigg)}  \\
& \lesssim \| \p ^8 u^\beta (t) \|_{L^2 (\mathbb{T}^2 )} \| \p^6 \tilde{u}^\beta (t ) \|_{L^2 (\mathbb{T}^ 2 ) }^2 + \| \p^8 \tilde{u}^\beta (t) \|_{L^2 (\mathbb{T} ^2 ) }^2.
\end{split}
\Ee
Therefore, we have
\Be \label{uutildeb}
\begin{split}
\| u^\beta \|_{L^\infty ((0, T); H^8 (\mathbb{T}^2 ) )} &\lesssim e^{\| \nabla_x u^\beta \|_{L^\infty ((0, T) \times \mathbb{T}^2 )}} \| u^\beta (0) \|_{L^\infty ((0, T); H^8 (\mathbb{T}^2 ) )} \\ 
&\lesssim e^{\mathfrak{Lip} (\beta, \pp)} \beta^{-8 - 2 \left ( \frac{1}{\pp} - \frac{1}{2} \right )_+} \| \o_0 \|_{L^\pp} =: U(\beta, \pp) , \\
\| \tilde{u}^\beta \|_{L^\infty ((0, T); H^6 (\mathbb{T}^2 ) )} &\lesssim e^{\|u^\beta \|_{L^\infty ((0, T); H^8 (\mathbb{T}^2 ) ) } T } \left ( \| \tilde{u}_0 \|_{H^6(\mathbb{T}^2 ) } +  T  \|u\|_{L^\infty ((0, T); H^8 (\mathbb{T}^2 ) ) } \right ) \\
& \lesssim (\| \tilde{u}_0 \|_{H^6(\mathbb{T}^2 ) } + TU(\beta, \pp) ) e^{T U(\beta, \pp ) }.
\end{split}
\Ee\end{proof}

\section{Hilbert-type Expansion with Viscosity-canceling corrector}

\subsection{Formulation around a local Maxwellian}

We denote a local Maxwellian corresponding to $(1 , \e u^\beta ,1 )$ by
\Be\label{mu_e}
\mu : = M_{1 , \e u^\beta , 1 }.
\Ee

We try to construct a family of solutions $F^\e$ in a form of 
\Be \label{F_e}
F^\e =  
 \mu+ \e^2 p^\beta  \mu - \e^2 \kappa (\nabla_x u^\beta) : \mathfrak{A} \sqrt{ \mu} +
   \{\e \kappa \tilde{u}^\beta \cdot (v-\e u^\beta)  
  + \e^2 \kappa \tilde{p}^\beta\}  \mu
  + \e f_R \sqrt{ \mu},
\Ee
where $p^\beta$, $\tilde{u}^\beta$, and $\tilde{p}^\beta$ satisfy \eqref{up} and \eqref{uptilde}, and $\mathfrak{A}$ will be defined in \eqref{AB}.

Also, we assume the following assumption on the relative maginitudes on $\e, \kappa = \kappa(\e), \beta =  \beta(\e) $:
\Be \label{ekappabeta}
\begin{split}
\lim_{\e \rightarrow 0} \frac{\e}{\kappa^2} = 0, \\
\lim_{\e \rightarrow 0} \kappa^{\frac{1}{4}} V(\beta ) = 0, \\
\lim_{\e \rightarrow 0} \kappa^{\frac{1}{2} } e^{2 \mathbf{C}_0 T \| \nabla_x u^\beta \|_{L^\infty ((0, T) \times \mathbb{T}^2 ) } ^2 } = 0,
\end{split}
\Ee
where $\mathbf{C}_0$ is specified in Section \ref{sec:re}.

We define 
\Be\label{L_Gamma}
L  f = \frac{-2}{\sqrt{\mu }} Q(\mu , \sqrt{\mu }f)   ,\ \
\Gamma (f,g) = \frac{1}{ \sqrt{\mu}}  Q(\sqrt{\mu}f, \sqrt{\mu}g)  .
\Ee
From the collision invariance, a null space of $L$, denoted by $\mathcal{N}$, 
has five orthonormal bases $\{\varphi_i \sqrt{\mu }\}_{i=1}^5$ with 
\hide  \Be\label{basis}
\varphi_0 := \frac{1}{\sqrt{1+ \delta \sigma}}
,   \  \ \ \varphi_i: = \frac{1}{\sqrt{1+ \delta \sigma}} \frac{v_i -\delta u_i }{\sqrt{1+ \delta \theta}}
\ \ \text{for} \ i=1,2,3 
,   \  \ \ \varphi_4: =  \frac{1}{ \sqrt{1+ \delta \sigma}}
\frac{ \big|\f{v-\delta u}{\sqrt{1+ \delta \theta}}\big|^2-3}{\sqrt{6}}.
\Ee\unhide
\Be\label{basis}
\begin{split}
	&\varphi = (\varphi_0, \varphi_1, \varphi_2, \varphi_3, \varphi_4), \\
&\varphi_0 := 1
,   \  \ \ \varphi_i: =   {v_i -\e u_i^\beta  } 
\ \ \text{for} \ i=1,2,3 
,   \  \ \ \varphi_4: =   
\frac{  | {v-\e u^\beta } |^2-3}{\sqrt{6}}.
\end{split}
\Ee
We define $\mathbf{P}$, an $L^2_v$-projection on $\mathcal{N}$, as 
\Be\label{P}
\begin{split}
	P g&:= (P_0 g, P_1 g, P_2 g, P_3 g, P_4 g )
, \ \  
{P}_j g:= \int_{\R^3} g \varphi_j \sqrt{\mu } \dd v \ \ \text{for } j=0,1, \cdots,4,\\
\mathbf{P} g&:= \sum_{j=0}^4  ( {P}_j g) \varphi_j \sqrt{\mu } = Pg \cdot \varphi \sqrt{\mu}.
\end{split}
\Ee

We record the exact form of $L$ and $\Gamma$ for the later purpose: the calculation is due to Grad \cite{G1963}: one can also read \cite{glassey} for details of derivations. Also the exact form of formulae were excerpted from \cite{JK2020}: For certain positive constants $c_1, c_2, c_3$,
\Be \label{Lgammaform}
\begin{split}
	Lf(v) &= \nu f (v) - Kf (v) =  \nu(v) f(v) - \int_{\mathbb{R}^3}\mathbf{k}(v, v_*) f(v_*) \dd v_*, \\
	\nu(v) &= c_1 \left ( (2 |v-\e u^\beta | + \frac{1}{|v-\e u^\beta |} ) \int_0 ^{|v-\e u^\beta | } e^{-\frac{z^2}{2}} \dd z + e^{-\frac{|v-\e u^\beta | ^2 }{2 }} \right ), \\
	\mathbf{k} (v, v_*) &= c_2 |v-v_* | e^{-\frac{|v-\e u^\beta |^2 + |v_* - \e u^\beta |^2 }{4}} - \frac{c_3}{|v-v_*| }e^{-\frac{1}{8} |v-v_*|^2 - \frac{1}{8} \frac{(|v-\e u^\beta |^2 - |v_* - \e u^\beta | ^2 )^2}{|v-v_* |^2 } }, \\
	\Gamma (f, g) ( v) &= \int_{\mathbb{R}^3} \int_{\mathbb{S}^2} |(v- v_*) \cdot \o | \sqrt{\mu (v_*) } (f(v')g(v_*') + g(v')f(v_*') ) \dd \o \dd v_*  \\
	&- \int_{\mathbb{R}^3} \int_{\mathbb{S}^2} |(v- v_*) \cdot \o| \sqrt{\mu (v_* ) } ( f(v) g(v_*) + g(v) f(v_*) ) \dd \o \dd v_*, 
\end{split}
\Ee
where $v^\prime = v- ((v-v_*) \cdot \o) \o$, $v_*^\prime = v_*+ ((v-v_*) \cdot \o) \o$. Here, all $\nu, \mathbf{k}, \Gamma$ also depend on $x$ and $t$ in a straightforward manner, that is, $Lf(x, t, v)$ and $\Gamma (f, g)(x, t, v)$ depends on $f(x, t, \cdot)$, $g(x, t, \cdot)$, and $u^\beta(x, t)$; we omitted them for the sake of simplicity.
 
Also, we define ${}_{\p^s} L$ and ${}_{\p^s} \Gamma$ for $s \ge 1$: 
\Be \label{psLGamma}
\begin{split}
{}_{\p^s} L f (v) &= \p^s (\nu ) (v) f(v) - \int_{\mathbb{R}^3} \p^s (\mathbf{k})(v, v_*) f(v_*) \dd v_*, \\
{}_{\p^s} \Gamma (f, g) ( v) &= \int_{\mathbb{R}^3} \int_{\mathbb{S}^2} |(v- v_*) \cdot \o | \p^s(\sqrt{\mu (v_*) }) (f(v')g(v_*') + g(v')f(v_*') ) \dd \o \dd v_*  \\
	&- \int_{\mathbb{R}^3} \int_{\mathbb{S}^2} |(v- v_*) \cdot \o| \p^s (\sqrt{\mu (v_* ) } ) ( f(v) g(v_*) + g(v) f(v_*) )\dd \o \dd v_*.
\end{split}
\Ee

We list standard results which will be used later in this section for the sake of readers' convenience. First we note that 
\Be\label{Lnonhydro}
Q(\mu, \mu)=0=\mathbf{P}L=L \mathbf{P}=\mathbf{P} \Gamma,
\Ee
from the collision invariance.

\begin{lemma}[\cite{EGKM2,Guo2006,Guo2002}]
	Suppose that \eqref{ekappabeta} holds. Then 
	\label{lemma_L}
	\Be\label{est:L}
	\begin{split}
	&\| \nu^{-1/2}L f \|_{L^2(\O \times \R^3)} \lesssim 	\|  \sqrt{\nu}  (\mathbf{I} - \mathbf{P}) f \|_{L^2(\O \times \R^3)}, \\
	&\| \nu^{\frac{1}{2}} (\mathbf{I} - \mathbf{P} ) f \|_{L^2_v }^2 \lesssim \left | \int Lf(v) f(v) \dd v \right | , \\ 
	&\left | \int {}_{\p^s} L f(v) g(v) \dd v \right | \\
	&\lesssim \e \| \p^s u^\beta \|_{L^\infty_{t,x} }  \left (\| \mathbf{P} f \|_{L^2_v} + \| \nu^{\frac{1}{2}} (\mathbf{I} - \mathbf{P} ) f \|_{L^2_v} \right ) \left (\| \mathbf{P} g \|_{L^2_v} + \| \nu^{\frac{1}{2}} (\mathbf{I} - \mathbf{P} ) g \|_{L^2_v} \right ),
	\end{split}
	\Ee
	\Be \label{est:Gamma}
	\begin{split}
&	\left | \int \Gamma ( f, g) h \d v \d x \d t \right |\\
&  \lesssim \int \left [ \left ( \| \mathbf{P} f  \|_{L^2_v} + \|\nu^{\frac{1}{2} } (\mathbf{I} - \mathbf{P} ) f\|_{L^2_v} \right ) \| g \|_{L^2_v} \right. \\
	& \ \ \ \ \ \ \ \ \   + \left. \left ( \| \mathbf{P} g \|_{L^2_v} + \| \nu^{\frac{1}{2} } (\mathbf{I} - \mathbf{P} ) g  \|_{L^2_v} \right ) \| f \|_{L^2_v}   \right ] \| \nu^{\frac{1}{2} } (\mathbf{I} - \mathbf{P} ) h  \|_{L^2_v}  \d x \d t, \\
&	\left | \int {}_{\p^s} \Gamma ( f, g) h \d v \d x \d t \right | \\
& \lesssim \e \| \p ^s u \|_{L^\infty_{t,x} }  \int \left [ \left ( \| \mathbf{P} f  \|_{L^2_v} + \|\nu^{\frac{1}{2} } (\mathbf{I} - \mathbf{P} ) f \|_{L^2_v} \right ) \| g \|_{L^2_v} \right. \\
	& \ \ \ \ \ \ \ \ \   \ \ \ \  \ \ \  \ \ \ \ \ \    + \left. \left ( \| \mathbf{P} g \|_{L^2_v} + \| \nu^{\frac{1}{2} } (\mathbf{I} - \mathbf{P} ) g  \|_{L^2_v} \right ) \| f \|_{L^2_v}   \right ] \\
	& \ \ \ \ \ \ \ \ \   \ \ \ \  \ \ \  \ \ \  \times  \left ( \| \mathbf{P} h  \|_{L^2_v} + \| \nu^{\frac{1}{2} } (\mathbf{I} - \mathbf{P} ) h  \|_{L^2_v} \right )  \d x \d t.
	\end{split}
	\Ee
\hide
	\Be\label{est:Gamma}
	\begin{split}
		\| \nu^{-1/2} \Gamma (f,g) \|_{L^2((0,t) \times \O \times \R^3)} 
		\lesssim & \  \| wg \|_{L^\infty([0,t] \times \O \times \R^3)}\|  \sqrt{\nu} (\mathbf{I} - \P) f \|_{L^2 ((0,t) \times \O \times \R^3)} 
		%
		\\	&
		+ \|  w g \|_{L^2(0,t;L^\infty(\O  ))}   \|    f \|_{L^\infty(0,t;L^2(\O \times \R^3  ))} .
	\end{split}
		\Ee
\unhide
		\end{lemma}

Next, we introduce a lemma illustrating the structure of higher derivatives of $Lf$. Recall the notation $[\![  \cdot , \cdot]\!]$ for the commutator \eqref{commutator}.
\begin{lemma} \label{Lpcomm}
For $s \ge 1$, $[\![ \p^s, L ]\!] f$ is a linear combination, whose coefficient depends only on $s$, of the terms having one of the following forms:
\begin{enumerate}
\item ${}_{\p^j} L (\mathbf{I} - \mathbf{P} ) \p^{s-j} f$, where $1 \le j \le s$, 
\item $L \p \cdots [\![\mathbf{P}, \p ]\!] \cdots \p f$, where $\p \cdots [\![\mathbf{P}, \p ]\!] \cdots \p f$ is an application of $s-1$ $\p$ and one $ [\![\mathbf{P}, \p ]\!]$ at $j$-th order to $f$ ($0 \le j \le s$), or
\item ${}_{\p^j} L \p \cdots [\![ \mathbf{P}, \p ]\!] \cdots \p f$, where $1 \le j \le s-1$, and $\p \cdots [\![ \mathbf{P}, \p ]\!] \cdots \p f$ is an application of $s-j-1$ $\p$ and one $ [\![\mathbf{P}, \p ]\!]$ at $i$-th order to $f$ ($0 \le i \le s-j$).
\end{enumerate}
\end{lemma}

\begin{proof}
We proceed by the induction on $s$: first we note that
\Be\notag
\begin{split}
 \p (Lf) &= \p L (\mathbf{I} - \mathbf{P}) f = {}_{\p} L (\mathbf{I} - \mathbf{P} ) f + L \p (\mathbf{I} - \mathbf{P} ) f \\
 &=  {}_{\p} L (\mathbf{I} - \mathbf{P} ) f + L [\![ \mathbf{P}, \p ]\!] f + L (\mathbf{I} - \mathbf{P} ) \p f, \\
[\![ \p, L ]\!] f &= {}_{\p} L (\mathbf{I} - \mathbf{P} ) f + L [\![ \mathbf{P}, \p ]\!] f,
\end{split}
\Ee
which proves the claim for $s=1$. Next, for $s\ge 1$, we have
\Be\notag
[\![ \p^{s+1}, L ]\!] f = \p^{s+1} L f - L \p^{s+1} f = \p [\![\p^s, L]\!] f  + [\![ \p, L ]\!] \p^s f ,
\Ee
and by the first step $[\![ \p, L ]\!] \p^s f $ consists of terms in the lemma. Also, application of $\p$ to the terms of the second and third form of the lemma produces terms of the second and third form again, while application of $\p$ to the first form produces
\Be\notag
\begin{split}
\p {}_{\p^j} L (\mathbf{I} - \mathbf{P} ) \p^{s-j} f &= {}_{\p^{j+1}} L (\mathbf{I} - \mathbf{P} ) \p^{s-j} f + {}_{\p^j} L \p (\mathbf{I} - \mathbf{P} ) \p^{s-j} f \\
&= {}_{\p^{j+1}} L (\mathbf{I} - \mathbf{P} ) \p^{s-j} f + {}_{\p^j} L [\![ \mathbf{P}, \p ]\!] \p^{s-j} f  + \cdots +  {}_{\p^j} L \p^{s-j} [\![ \mathbf{P}, \p ]\!]  f \\
& +  {}_{\p^j} L (\mathbf{I} - \mathbf{P} ) \p^{s-j+1} f,
\end{split}
\Ee
which proves the claim.
\end{proof}

Also, we have the following straightforward estimate for $[\![\mathbf{P}, \p ]\!] f$:
\begin{lemma} \label{pPcomm}
Suppose that \eqref{ekappabeta} holds. For $ s_1 + s_2 \le 1$, the following holds:
\Be
\begin{split}\notag
[\![\mathbf{P}, \p ]\!] f &= - \sum_{i=0} ^4 \langle f, \varphi_i \sqrt{\mu} \rangle_{L^2_v} \p (\varphi_i \sqrt{\mu} ),\\
\| [\![\mathbf{P}, \p ]\!] f \|_{L^2_v} &\lesssim \e \| \nabla_x u^\beta \|_{L^\infty_{t,x} } \| f \|_{L^2_v }, \\
\| \p^{s_1} [\![\mathbf{P}, \p ]\!] \p^{s_2} f \|_{L^2_v} &\lesssim \e V(\beta)  \| \p^{s_1 + s_2} f \|_{L^2_v }. 
\end{split}
\Ee
\end{lemma}

Next, we introduce anisotropic spaces: this will be key to our analysis. For $p \in [1, \infty]$, we recall the space $L^p(\mathbb{T}^2 ; L^2 (\mathbb{R}^3))$ by the norm $\| f \|_{L^p(\mathbb{T}^2 ; L^2 (\mathbb{R}^3))} $ in \eqref{mixedLp}. 
For $p, q \in [1, \infty]$, $L^q([0, T]; L^p(\mathbb{T}^2; L^2 (\mathbb{R}^3 ) ))$ is defined similarly. We have the following anisotropic interpolations:

\begin{lemma}\label{anint}
We have the following:
\begin{enumerate}
\item (Anisotropic Ladyzhenskaya) $\| f \|_{L^4_x L^2_v} \lesssim \| f \|_{L^2_x L^2_v} ^{\frac{1}{2} } \| \p f \|_{L^2_x L^2_v} ^{\frac{1}{2} }$, and 
\item (Anisotropic Agmon) $\| f \|_{L^\infty_x L^2_v} \lesssim \| f \|_{L^2_x L^2_v} ^{\frac{1}{2} } \| \p^2 f \|_{L^2_x L^2_v} ^{\frac{1}{2} }$.
\end{enumerate}
\end{lemma}

\begin{proof}
We only prove the former: the latter is derived in a similar manner.
\Be
\begin{split}\notag
\| f \|_{L^4_x L^2_v} &= \left ( \int_{\mathbb{T}^2} \left ( \int_{\mathbb{R}^3} |f(x,v)|^2 \dd v \right )^{\frac{4}{2} } \dd x \right )^{\frac{1}{2\cdot 2} } \le  \left (\int_{\mathbb{R}^3} \left (\int_{\mathbb{T}^2} |f(x,v) |^4 \dd x \right )^{\frac{1}{2}} \dd v \right )^{\frac{1}{2}} \\
& = \left ( \int_{\mathbb{R}^3} \| f(\cdot ,v )\|_{L^4_x} ^2 \dd v \right )^{\frac{1}{2}} \lesssim \left ( \int_{\mathbb{R}^3} \| f (\cdot ,v) \|_{L^2_x} \| \p f (\cdot , v ) \|_{L^2 _x} \dd v \right )^{\frac{1}{2}} \\
& \le \left ( \int_{\mathbb{R}^3} \int_{\mathbb{T}^2} |f(x,v)|^2 \dd x \dd v \right )^{\frac{1}{2\cdot 2}} \left ( \int_{\mathbb{R}^3} \int_{\mathbb{T}^2} |\p f(x,v)|^2 \dd x \dd v \right )^{\frac{1}{2\cdot 2}} \\
&= \| f \|_{L^2_x L^2_v} ^{\frac{1}{2} } \| \p f \|_{L^2_x L^2_v} ^{\frac{1}{2} }.
\end{split}
\Ee
where we applied Minkowski for the first, usual Ladyzhenskaya for the second, and Holder for the last inequalities.
\end{proof}

From Lemma \ref{anint}, we have the following.
\begin{lemma} \label{nonhydroLp}
\Be\notag
\begin{split}
&\| \nu^{\frac{1}{2} } (\mathbf{I} - \mathbf{P} ) f \|_{L^4_x L^2_v}\\
 \lesssim& \  \e^{\frac{1}{2}}  \| \nu^{\frac{1}{2} } (\mathbf{I} - \mathbf{P} ) f \|_{L^2_x L^2_v} ^{\frac{1}{2} } \\
& \ \  \times \left (\|\p u^\beta \|_{L^\infty_{x} }  \| f \|_{L^2_x L^2_v} + \| \e^{-1} \nu^{\frac{1}{2} } (\mathbf{I} - \mathbf{P}) \p f \|_{L^2_x L^2_v} +  V(\beta) \| \e^{-1} \nu^{\frac{1}{2} } (\mathbf{I} - \mathbf{P} ) f \|_{L^2_x L^2_v }  \right )^{\frac{1}{2}}, \\
&\| \nu^{\frac{1}{2} } (\mathbf{I} - \mathbf{P} ) f \|_{L^\infty_x L^2_v}\\
 \lesssim & \ \e^{\frac{1}{2}} \| \nu^{\frac{1}{2} } (\mathbf{I} - \mathbf{P} ) f \|_{L^2_x L^2_v} ^{\frac{1}{2} } \\
& \ \  \times \left [\| \p u^\beta \|_{L^\infty_x} \| \p f \|_{L^2_x L^2_v }+ \|\e^{-1} \nu^{\frac{1}{2} } (\mathbf{I} - \mathbf{P} ) \p^2 f \|_{L^2_x L^2_v} \right. \\
& \ \ \ \  \ \ \  \left. + V(\beta) \left (\| \e^{-1} \nu^{\frac{1}{2}} (\mathbf{I} - \mathbf{P} ) f \|_{L^2_x L^2_v } + \|\e^{-1}  \nu^{\frac{1}{2} } (\mathbf{I} - \mathbf{P}) \p f \|_{L^2_x L^2_v } + \| f \|_{L^2_x L^2_v} \right ) \right ]^{\frac{1}{2} }.
\end{split}
\Ee
\end{lemma}
\begin{proof}
We only give proof for the first inequality: the second inequality can be proved by a similar argument. By Lemma \ref{anint}, it suffices to control $\p (\nu^{\frac{1}{2} } (\mathbf{I} -\mathbf{P} ) f )$: we have
\Be\notag
\p (\nu^{\frac{1}{2} } (\mathbf{I} -\mathbf{P} ) f ) = \frac{1}{2} \nu^{-1} \p (\nu) \nu^{\frac{1}{2}} (\mathbf{I} - \mathbf{P} ) f + \nu^{\frac{1}{2}}[\![\mathbf{P},\p]\!] f + \nu^{\frac{1}{2}}(\mathbf{I} - \mathbf{P} ) \p f.
\Ee
One can easily check that $ \sup_{x,v} | \nu^{-1} \p (\nu) | \lesssim \e \| \p u^\beta \|_{L^\infty_x}$, and thus the inequality follows.
\end{proof}



\begin{lemma}[\cite{CIP,Guo2006}] \label{Linverse}
	$L|_{\mathcal{N}^\perp} : \mathcal{N}^\perp \rightarrow \mathcal{N}^\perp$ is a bijection, and thus $L^{-1} : \mathcal{N}^\perp \rightarrow \mathcal{N}^\perp$ is well-defined. Also, $L^{-1}$ is symmetric under any orthonormal transformation. In particular, if $f \in \mathcal{N}^\perp $ is an even (resp. odd) function, then so is $L^{-1} f$.
\end{lemma}
\begin{proof}
The proof follows the Fredholm alternative and rotational invariance of $Q$. We refer to \cite{CIP,Guo2006} for the proof. 
	\end{proof}


The term $
(v-\e u^\beta) \otimes (v -\e u^\beta) 
 \sqrt{\mu}
$ and its image over $L^{-1}$ turns out to play an important role in Hilbert expansion. Note that 
	\begin{align}\label{vvproj}
	&(\mathbf{I} - \mathbf{P} )  \left (  (v-\e u^\beta) \otimes (v -\e u^\beta) \sqrt{\mu}\right ) =  \left ((v-\e u^\beta) \otimes (v -\e u^\beta) - \frac{1}{3} |v-\e u^\beta|^2 \mathbf{I}_3 \right ) \sqrt{\mu}.
\end{align}
\hide
\begin{lemma}[\cite{BGL1, JK2020}] 
	We have the following:

\end{lemma}
\begin{proof}
	The proof is straightforward, via Isserlis theorem.{\color{red}\huge[[Add details]]}
\end{proof}\unhide
Thus, we define $\mathfrak{A}:= \mathfrak{A}(t,x) \in \mathbb{M}_{3 \times 3} (\mathbb{R})$ by (see \cite{BGL1})
\begin{equation} \label{AB}
	\mathfrak{A}_{ij} = L^{-1} \left ( \left ((v-\e u^\beta)_i (v-\e u^\beta)_j - \frac{|v- \e u^\beta |^2}{3} \delta_{ij} \right ) \sqrt{\mu} \right ).
\end{equation}

Regarding $\mathfrak{A}$, we have the following useful lemma.
\begin{lemma}[\cite{BGL1, BGL2}] \label{Ainn}
	$\langle L \mathfrak{A}_{\ell k}, \mathfrak{A}_{ij} \rangle = \eta_0 (\delta_{ik} \delta_{j\ell} + \delta_{i\ell} \delta_{jk} ) - \frac{2}{3} \eta_{0} \delta_{ij} \delta_{k\ell}.$
\end{lemma}
\begin{proof}
We refer to \cite{BGL1, BGL2} for the proof. 
	\end{proof}

From explicit calculation, we can also establish the following result:
\begin{lemma} \label{hydroph3}
For $i,j,k \in \{ 1,2,3 \}$, 
\Be\notag
\mathbf{P} (\varphi_i \varphi_j \varphi_k \sqrt{\mu} ) = \sum_{\ell=1} ^3 (\delta_{ij}\delta_{k\ell} + \delta_{ik}\delta_{j\ell} + \delta_{jk}\delta_{i\ell} ) \varphi_\ell \sqrt{\mu}.
\Ee
\end{lemma}

We also have the following useful pointwise estimates. First, we have the following pointwise estimates on $\p^s \left ( f \frac{(\partial_t + \frac{v}{\e} \cdot \nabla_x) \sqrt{\mu} }{\sqrt{\mu} } \right )$:

\begin{lemma} \label{Momentstreambound}
Suppose that \eqref{ekappabeta} holds. Then for $s \le 2$, we have
\Be \label{mombound}
\begin{split}
\p^s \left ( f \frac{(\partial_t + \frac{v}{\e} \cdot \nabla_x) \sqrt{\mu} }{\sqrt{\mu} } \right ) = & \  \p^s f\left ( \frac{(\partial_t + \frac{v}{\e} \cdot \nabla_x) \sqrt{\mu} }{\sqrt{\mu} } \right ) \\& + \sum_{s' < s} (\p^{s'} f) \frac{1}{2} \sum_{i,j} (\partial^{s-s'} \p_{x_i} u^\beta_j) \varphi_i \varphi_j  + R,
\end{split}
\Ee
where $
|R| \lesssim \e V(\beta) \nu(v) \sum_{s' < s}| \p^{s'} f |.$
\end{lemma}
\begin{proof}
It suffices to notice that
\Be\notag
  \frac{(\partial_t + \frac{v}{\e} \cdot \nabla_x) \sqrt{\mu} }{\sqrt{\mu} } = \frac{1}{2}     \sum_{i,j} \p_{x_i} u^{\beta}_j \varphi_i \varphi_j +  \frac{1}{2}   \e \sum_{i} (\p_t u^\beta + u^\beta \cdot \nabla_x u^\beta)_i \varphi_i ,
\Ee
and that the first two terms of the right-hand side of \eqref{mombound} correspond to the terms where all $\p$ are applied to either $f$ or $\partial_{x_i} u^\beta_j$, and $R$ are all others.
\end{proof}

Next, we present pointwise estimates on $A$ and its derivatives (\cite{JK2020})
\begin{lemma}[Lemma 3 of \cite{JK2020}]\label{Abound}
	Suppose that \eqref{ekappabeta} holds. For $\varrho \in (0, 1/4)$, 
	\begin{equation}\notag
		\begin{split}
			|\mathfrak{A}_{ij} (v) | &\lesssim e^{-\varrho |v-\e u^\beta |^2 }, \\
			\sum_{s \le 2, D \in \{ \p_t, \p\} }\left |  \p^s \left ( (1+(u^\beta, \tilde{u}^\beta) )D \mathfrak{A}_{ij} (v) \right ) \right |   & \lesssim \e V(\beta) e^{-\varrho |v-\e u^\beta |^2 }. \\
			\hide
			| \nabla_x \mathfrak{A}_{ij} (v) | &\lesssim \e |\nabla_x u^\beta| e^{-\varrho |v-\e u^\beta |^2 },\\ |\partial_t \mathfrak{A}_{ij} (v) | &\lesssim \e |\p_t u^\beta| e^{-\varrho |v-\e u^\beta |^2 }, \\
			|(\p_t + u^\beta \cdot \nabla_x)  \mathfrak{A}_{ij} (v) | &\lesssim \e |(\p_t + u^\beta \cdot \nabla_x ) u^\beta| e^{-\varrho |v-\e u^\beta |^2 }, \\
|  \nabla_x \partial_t   \mathfrak{A}_{ij} (v) | &\lesssim \e ( | \nabla_x \partial_t u | + \e |\nabla_x u| \partial_t u| ) e^{-\varrho |v- \e u^\beta |^2}, \\
|\nabla_x ^2 \mathfrak{A}_{ij} (v) | &\lesssim \e( | \nabla_x^2 u | + \e |\nabla_x u|^2  ) e^{-\varrho |v- \e u^\beta |^2}. \unhide
		\end{split}
	\end{equation}
\end{lemma}
 
Next, we have the following pointwise estimates on $\Gamma$ and $L$:
\begin{lemma}[Lemma 4 of \cite{JK2020}] \label{Gammabound}
	Suppose that $\e |u^\beta (x,t) | \lesssim 1$. For $0< \varrho < 1/4$, $C \in \mathbb{R}^3$ and $s \le 2$, we have
	\begin{equation}\notag
	\begin{split}
		|\Gamma(f,g)(v) | &\lesssim  \| e^{\varrho|v|^2 + C \cdot v }f(v) \|_{L^\infty_v} \| e^{\varrho|v|^2 + C \cdot v }g(v) \|_{L^\infty_v} \frac{\nu(v)}{e^{\varrho|v|^2+C\cdot v}}, \\
		|{}_{\p^s} \Gamma (f,g)(v) | &\lesssim \e V(\beta) \| e^{\varrho|v|^2 + C \cdot v }f(v) \|_{L^\infty_v} \| e^{\varrho|v|^2 + C \cdot v }g(v) \|_{L^\infty_v} \frac{\nu(v)}{e^{\varrho|v|^2+C\cdot v}}, \\
		| {}_{\p^s} L f (v) | & \lesssim \e V(\beta) \| e^{\varrho|v|^2 + C \cdot v }f(v) \|_{L^\infty_v}  \frac{\nu(v)^2 }{e^{\varrho|v|^2+C\cdot v}}.
	\end{split}
	\end{equation}
Here we can choose the constant for the bound uniformly for $\{ |C | \le 1 \}$.
\end{lemma}

Finally, we present pointwise estimates regarding projections $\mathbf{P}$ and $\mathbf{I} - \mathbf{P}$. 

\begin{lemma} \label{projbound}
	Suppose that $f(t,x,v)  \in L^2_v$ satisfies $|f(t,x,v) | \le C(t,x ) \exp \left ( - \varrho |v- \e u^\beta(t,x) |^2 \right )$ for some constant $C(t,x)$ independent of $v$ and $\varrho \in (0, 1/4)$. Then
\Be	\begin{split}\label{hydrop}
		| \mathbf{P} f (t,x,v) | &\lesssim C(t,x ) \exp \left ( - \varrho |v- \e u^\beta(t,x) |^2 \right ), \\
		|(\mathbf{I} - \mathbf{P} ) f(t,x,v) | &\lesssim C(t,x ) \exp \left ( - \varrho |v- \e u^\beta(t,x) |^2 \right ),
	\end{split}\Ee
	where the constants for inequalities are independent of $t,x,v$ but depend on $\varrho$. 
\end{lemma}
\begin{proof}
	It suffices to show (\ref{hydrop}) only: the other follows from $|(\mathbf{I} - \mathbf{P} ) f(t,x,v) | \le | \mathbf{P} f (t,x,v) | + |f(t,x,v) |$. Note that, from \eqref{P}, 
	\begin{align*}
		&| \mathbf{P} f(t,x,v)| \\
		&\le \sum_{\ell=1} ^5 C(t,x) \int \langle v - \e u^\beta \rangle ^2 \exp \left ( - \left (\varrho + \frac{1}{4} \right ) |v- \e u^\beta(t,x) |^2 \right ) dv \langle v - \e u^\beta \rangle ^2 \sqrt{\mu} \\
		& \le C(t,x) C_{\varrho}  \exp \left ( - \varrho |v- \e u^\beta(t,x) |^2 \right ).
	\end{align*}\end{proof}

\subsection{New Hilbert-type Expansion}

We recall an explicit form of derivatives of ${\mu}^k$:  
\Be\begin{split}\label{vDmu}
	& \left [ \p_t + u^\beta \cdot \nabla_x \right ] {\mu}^k = \e k (\p_t u^\beta + u^\beta \cdot \nabla_x u^\beta ) \cdot (v - \e u^\beta ) \mu^k, \\
	&(v - \e u^\beta) \cdot \nabla_x \mu^k = \e k ( \nabla_x u^\beta ) : ( (v-\e u^\beta ) \otimes (v - \e u^\beta)) \mu^k,
\end{split}\Ee
where $k > 0$ and $A:B = \mathrm{tr} (AB) = \sum_{i,j=1} ^3 A_{ij} B_{ji}$ for arbitrary rank 2 tensors $A$, $B$. 

Now we derive an equation of $f_R$. First, we plug (\ref{F_e}) into (\ref{eqtn_F}) to obtain
\begin{align}
&(\underline{v}-\e u^\beta) \cdot \nabla_x \left ({\mu} {+ \e^2 p^\beta \mu} {- \e^2 \kappa (\nabla_x u^\beta): \mathfrak{A}\sqrt{\mu} } {+ \e \kappa \tilde{u}^\beta \cdot (v-\e u^\beta) \mu} {+ \e^2 \kappa \tilde{p}^\beta \mu} \right )  \label{HE1} \\
&+ \e (\p_t + u^\beta \cdot \nabla_x ) \left ({\mu} + \e^2 p^\beta \mu - \e^2 \kappa (\nabla_x u^\beta): \mathfrak{A}\sqrt{\mu} {+ \e \kappa \tilde{u}^\beta \cdot (v-\e u^\beta) \mu }+ \e^2 \kappa \tilde{p}^\beta \mu \right ) \label{HE2}\\
& - \frac{1}{\kappa \e}  Q (\mu + \e^2 p^\beta \mu - \e^2 \kappa (\nabla_x u^\beta): \mathfrak{A}\sqrt{\mu} + \e \kappa \tilde{u}^\beta \cdot (v-\e u^\beta) \mu + \e^2 \kappa \tilde{p}^\beta \mu ) \label{HE3}\\  
&+ \e^2 \Big\{ 
\p_t (f_R\sqrt{\mu} ) + \frac{\underline{v}}{\e} \cdot \nabla_x (f_R\sqrt{\mu} ) - \frac{1}{\e \kappa} Q(f_R \sqrt{\mu}, f_R \sqrt{\mu} )\Big\}  \label{HE5}\\
&  -  \frac{2}{  \kappa} Q(\mu + \e^2 p^\beta \mu - \e^2 \kappa (\nabla_x u^\beta): \mathfrak{A}\sqrt{\mu} + \e \kappa \tilde{u}^\beta \cdot (v-\e u^\beta) \mu + \e^2 \kappa \tilde{p}^\beta \mu, f_R\sqrt{\mu} )  = 0 \label{HE6},
\end{align}
where we have used an abbreviation $Q(g)=Q(g,g)$ in \eqref{HE3}. 

We group the source terms \eqref{HE1}, \eqref{HE2}, \eqref{HE3} with corresponding order of magnitudes: it is good to keep in mind that in our method, all hydrodynamic terms of order of magnitude less than $\e^2 \kappa$ are considered small, and all non-hydrodynamic terms of order of magnitude less than $\e \sqrt{\kappa}$ are considered small. In the end, we will group all small terms altogether.

\paragraph{0. Terms which are greater than $\e$:} Among terms which are independent of $f_R$, There are no terms whose magnitude is greater than $\e$: for terms in \eqref{HE1} and \eqref{HE2} this is obvious: the largest term comes from $(\underline{v} - \e u^\beta) \cdot \nabla_x \mu$, which is of order $\e$. For terms in \eqref{HE3}, we note that since $(v-\e u^\beta ) \sqrt{\mu}, \sqrt{\mu} \in \mathcal{N}$, in fact \eqref{HE3} can be rewritten as
\Be \label{HEQ}
\begin{split}
&{ 2\e Q(\mu (1 {\color{black}+ \e^2 p^\beta + \e^2 \kappa \tilde{p}^\beta }) , (\nabla_x u^\beta): \mathfrak{A} \sqrt{\mu} )} {- \kappa \e Q(\tilde{u}^\beta \cdot (v -\e u^\beta) \mu,\tilde{u}^\beta \cdot (v -\e u^\beta) \mu)} \\
& {+ 2 \e^2 \kappa Q(\tilde{u}^\beta \cdot (v-\e u^\beta) \mu, (\nabla_x u^\beta) : \mathfrak{A} \sqrt{\mu} )} - \e^3 \kappa Q((\nabla_x u^\beta): \mathfrak{A} \sqrt{\mu}, (\nabla_x u^\beta): \mathfrak{A} \sqrt{\mu}),
\end{split}
\Ee
whose leading order is $\e$.

\subsubsection{Order $\e$:} Among terms which are independent of $f_R$, there are two terms of order $\e$:
\Be\notag
\begin{split}
&(\underline{v} - \e u^\beta) \cdot \nabla_x \mu + \frac{2}{\kappa \e} Q(\mu, \e^2 \kappa (\nabla_x u^\beta):\mathfrak{A} \sqrt{\mu} ) \\
&=\e  \nabla_{x} u^\beta : (v-\e u^\beta) \otimes (v-\e u^\beta) \mu - \e  (\nabla_x u^\beta): L\mathfrak{A} \sqrt{\mu}  = 0,
\end{split}
\Ee
as $\nabla_x \cdot u^\beta = 0.$

\subsubsection{Order $\e \kappa$:} Among terms which are independent of $f_R$, there are two terms of order $\e \kappa$.
\begin{align}
&\e \kappa (\underline{v} - \e u^\beta) \cdot \nabla_x ( \tilde{u}^\beta \cdot (v- \e u^\beta) \mu ) - \e \kappa Q(\tilde{u}^\beta \cdot (v -\e u^\beta) \mu,\tilde{u}^\beta \cdot (v -\e u^\beta) \mu) \notag\\
&= \e \kappa \left ( (\nabla_x \tilde{u}^\beta): L \mathfrak{A}  - \Gamma \left ( \tilde{u}^\beta \cdot (\underline{v} - \e u^\beta) \sqrt{\mu},\tilde{u}^\beta \cdot (\underline{v} - \e u^\beta) \sqrt{\mu} \right) \right ) \sqrt{\mu}  \label{ekappa} \\
& + \e^2 \kappa ( \sum_{i, j } (\underline{v} - \e u^\beta)_i \tilde{u}^\beta_j \left (- \partial_{x_i} u^\beta_j + (\underline{v} - \e u^\beta)_j (\underline{v} - \e u^\beta)_k \partial_{x_i} u^\beta_k \right )) \mu ,\label{e2kappa1}
\end{align}
as $\nabla_x \cdot \tilde{u}^\beta = 0.$ Note that terms of order $\e \kappa$ are non-hydrodynamic: $\frac{1}{\sqrt{\mu} } \eqref{ekappa} \in \mathcal{N}^\perp$.

\subsubsection{Order $\e^2$:} The following terms are of order $\e^2$:
\begin{align}
& \e (\p_t + u^\beta \cdot \nabla_x) \mu + \e^2 (\underline{v} - \e u^\beta) \cdot \nabla_x (p^\beta \mu)\notag \\
&= \e^2 \left ((\p_t + u^\beta \cdot \nabla_x )u^\beta + \nabla_x p^\beta \right ) \cdot (\underline{v} - \e u^\beta) \mu \notag \\
&+ \e^3 p^\beta \nabla_{x_i} u^\beta_j (\underline{v} - \e u^\beta)_i (\underline{v} - \e u^\beta)_j \mu = \e^3 p^\beta \nabla_{x_i} u^\beta_j (\underline{v} - \e u^\beta)_i (\underline{v} - \e u^\beta)_j \mu,  \label{e3}
\end{align}
since $(\p_t + u^\beta \cdot \nabla_x )u^\beta + \nabla_x p^\beta= 0$.

\subsubsection{Order $\e^2 \kappa$:} The key reason to introduce correctors $\e \kappa \tilde{u}^\beta \cdot (v - \e u^\beta) \mu$ and $\e^2 \kappa \tilde{p}^\beta \mu$ is to get rid of hydrodynamic terms of order $\e^2 \kappa$: as a payback, we obtained terms of order $\e \kappa$, which is larger, but all of them are non-hydrodynamic, so they are small in our scale. The following is the collection of all terms of order $\e^2 \kappa$:
\begin{align}
&-\e^2 \kappa (\underline{v} - \e u^\beta) \cdot \nabla_x ( (\nabla_x u^\beta) : \mathfrak{A} \sqrt{\mu} ) + \e^2 \kappa (\underline{v} - \e u^\beta) \cdot \nabla_x (\tilde{p}^\beta \mu ) + \eqref{e2kappa1} \notag \\
& + \e^2 \kappa (\p_t + u^\beta \cdot \nabla_x ) ( \tilde{u}^\beta \cdot (\underline{v} - \e u^\beta) \mu )  + 2 \e^2 \kappa \Gamma ( \tilde{u}^\beta \cdot (\underline{v} - \e u^\beta) \sqrt{\mu}, (\nabla_x u^\beta): \mathfrak{A}) \sqrt{\mu} \notag \\
&   = \e^2 \kappa  
\{-\eta_0 \Delta_x u^\beta 
+ 
\nabla_x \tilde{p}^\beta 
  +
   (\p_t + u^\beta \cdot \nabla_x) \tilde{u}^\beta \} 
   \cdot (\underline{v} - \e u^\beta) \mu  \label{e2kappahydro1} \\
& {\tiny+ \e^2 \kappa \bigg (- \sum_{i,j} \tilde{u}^\beta_j \partial_{x_i} u^\beta_j (\underline{v} - \e u^\beta)_i + \sum_{i,j,k, \ell} \tilde{u}^\beta_j \partial_{x_i} u^\beta_k (\delta_{ij} \delta_{k\ell} + \delta_{ik} \delta_{j\ell} + \delta_{jk} \delta_{i\ell} ) (\underline{v} - \e u^\beta)_\ell \bigg ) \mu} \label{e2kappahydro2} \\
& +  \e^2 \kappa \left (2 \Gamma ( \tilde{u}^\beta \cdot (\underline{v} - \e u^\beta) \sqrt{\mu}, (\nabla_x u^\beta): \mathfrak{A} ) - (\nabla_x ^2 u^\beta):(\mathbf{I} - \mathbf{P}) (\underline{v} - \e u^\beta) \mathfrak{A}  \right ) \sqrt{\mu} \label{e2kappanonhydro1}\\
& + \e^2 \kappa \sum_{i,j, k} \tilde{u}^\beta_j \p_{x_i} u^\beta_k  (\mathbf{I} - \mathbf{P})\left ( (\underline{v} - \e u^\beta)_i (\underline{v} - \e u^\beta)_j (\underline{v} - \e u^\beta)_k \sqrt{\mu} \right ) \sqrt{\mu} \label{e2kappanonhydro2} \\
&  {\tiny+ \e^2 \kappa \left ( - (\nabla_x u^\beta): (\underline{v} - \e u^\beta) \cdot \nabla_x (\mathfrak{A} \sqrt{\mu} ) + \tilde{p}^\beta (\underline{v} - \e u^\beta) \cdot \nabla_x \mu + \tilde{u}^\beta \cdot (\p_t + u^\beta \cdot \nabla_x) ((\underline v - \e u^\beta )\mu ) \right ) } \label{e3kappa1} \\
&= \eqref{e2kappanonhydro1}+ \eqref{e2kappanonhydro2} + \eqref{e3kappa1}\notag.
\end{align}
Here, we have used Lemma \ref{hydroph3}, and that \eqref{e2kappahydro1} and \eqref{e2kappahydro2} can be gathered to form
\Be\notag
\eqref{e2kappahydro1} + \eqref{e2kappahydro2}  = \e^2 \kappa ( (\p_t + u^\beta \cdot \nabla_x) \tilde{u}^\beta + \tilde{u}^\beta \cdot \nabla_x u^\beta - \eta_0 \Delta_x u^\beta + \nabla_x \tilde{p}^\beta ) \cdot (\underline{v} - \e u^\beta) \mu = 0.
\Ee
Note that $\frac{1}{\sqrt{\mu} }\left (\eqref{e2kappanonhydro1} + \eqref{e2kappanonhydro2}\right ) \in \mathcal{N}^\perp$, that is, it is non-hydrodynamic so small in our scales, and \eqref{e3kappa1} is small: in fact, it is of order $\e^3 \kappa$.

\subsubsection{Small, non-necessarily non-hydrodynamic remainders}
The remaining terms are small in our scales: the following gathers all remaining terms.
\Be
\begin{split} \label{R1}
 \e^3 \sqrt{\mu} \mathfrak{R}_1  
 = & \   \eqref{e3} + \eqref{e3kappa1} + \e^3 (\p_t + u^\beta \cdot \nabla_x) (p^\beta \mu) \\
& + \e^3 \kappa (\p_t + u^\beta \cdot \nabla_x) ( - (\nabla_x u^\beta):\mathfrak{A} \sqrt{\mu} + \tilde{p}^\beta \mu ) \\
&- \e^3 p^\beta (L (\nabla_x u^\beta): \mathfrak{A} )\sqrt{\mu} - \e^3 \kappa \tilde{p}^\beta (L(\nabla_x u^\beta): \mathfrak{A} )\sqrt{\mu} \\
&- \e^3 \kappa \Gamma ((\nabla_x u^\beta): \mathfrak{A},(\nabla_x u^\beta): \mathfrak{A}) \sqrt{\mu}.
\end{split}
\Ee
One can easily observe the following:
\begin{proposition}
Suppose that \eqref{ekappabeta} holds. $\mathfrak{R}_1$ consists of a linear combination of the terms in the following tensor product:
\Be
\left (
\begin{matrix}
1 \\ \kappa \\ \e \\ \e \kappa
\end{matrix}
\right ) \otimes
\left (
\begin{matrix}
1 \\
p^\beta \\
\nabla_x u^\beta \\ \tilde{p}^\beta \\ \tilde{u}^\beta \\ \tilde{u}^\beta \otimes u^\beta \\ u^\beta
\end{matrix}
\right ) \otimes
D \left (
\begin{matrix}
p^\beta \\ u^\beta \\ \nabla_x u^\beta \\ \tilde{p}^\beta 
\end{matrix}
\right ) \otimes 
\mathfrak{P}^{\le 2} ( (\underline v - \e u^\beta ) )
\left ( 
\begin{matrix}
\sqrt{\mu} \\ \frac{1}{\e} \p_t \mathfrak{A} \\\frac{1}{\e} \p \mathfrak{A} \\ L \mathfrak{A} \\ \Gamma(\mathfrak{A}, \mathfrak{A} )
\end{matrix}
\right ), \notag
\Ee
where $D$ is either $\p_t$ or $\p$, which is applied to $p^\beta, u^\beta, \nabla_x u^\beta, \tilde{p}^\beta$, and $\mathfrak{P}^{\le 2}$ is a polynomial of degree $\le 2$ of its arguments. In particular, for $\varrho \in (0,\frac{1}{4})$ and $s \le 2$, we have the following pointwise estimate:
\Be \label{R1point}
| \p^s \mathfrak{R}_1 | \lesssim V(\beta) e^{-\varrho |v- \e u^\beta|^2 }. 
\Ee
\end{proposition}

\subsubsection{Small non-hydrodynamic remainders} \eqref{ekappa}, \eqref{e2kappanonhydro1}, \eqref{e2kappanonhydro2} are non-hydrodynamic remainders. We group them to obtain the following proposition:

\begin{proposition}
Suppose that \eqref{ekappabeta} holds. Let $ \mathfrak{R}_2$ be defined by
\Be \label{R2}
\e \kappa \sqrt{\mu} \mathfrak{R}_2 = \eqref{ekappa}+ \eqref{e2kappanonhydro1}+ \eqref{e2kappanonhydro2}.
\Ee
Then $\mathfrak{R}_2$ consists of a linear combination of the terms in the following tensor product:
\Be
\left (
\begin{matrix}
1 \\ \e
\end{matrix}
\right ) \otimes 
\left (
\begin{matrix}
\nabla_x \tilde{u}^\beta \\ \tilde{u}^\beta \otimes \tilde{u}^\beta \\ \tilde{u}^\beta \otimes \nabla_x u^\beta \\ \nabla_x^2 u^\beta 
\end{matrix}
\right ) \otimes
\left (
\begin{matrix}
L\mathfrak{A} \\
\Gamma((\underline v - \e u^\beta) \sqrt{\mu}, (\underline v - \e u^\beta) \sqrt{\mu} ) \\
\Gamma((\underline v - \e u^\beta) \sqrt{\mu}, \mathfrak{A} ) \\
(\mathbf{I} - \mathbf{P} ) (\underline v - \e u^\beta) \mathfrak{A} \\
(\mathbf{I} - \mathbf{P} ) (\underline v - \e u^\beta)^{\otimes^3} \sqrt{\mu}.
\end{matrix}
\right ). \notag
\Ee
In particular, $\mathfrak{R}_2 \in \mathcal{N}^\perp$, and for $\varrho \in (0,\frac{1}{4})$ and $s \le 2$, we have the following pointwise estimate:
\Be \label{R2point}
\begin{split}
| (\mathbf{I} - \mathbf{P}) \p^s \mathfrak{R}_2 | \lesssim V(\beta) e^{-\varrho |v- \e u^\beta|^2 }, \\
| \mathbf{P} \p^s \mathfrak{R}_2 | \lesssim \e V(\beta) e^{-\varrho |v- \e u^\beta|^2 }.
\end{split}
\Ee
\end{proposition}
\begin{proof}
It suffices to show \eqref{R2point}: we see that if all $\p^s$ are applied to macroscopic quantities $\nabla_x \tilde{u}^\beta, \cdots, \nabla_x ^2 u^\beta$, then the resulting term is still non-hydrodynamic. In that case, the first inequality of \eqref{R2point} applies. On the other hand, if some of $\p$ are applied to microscopic quantities $g = L\mathfrak{A}, \cdots, (\mathbf{I} - \mathbf{P}) (\underline{v} - \e u^\beta)^{\otimes^3} \sqrt{\mu}$, we note that 
\Be\notag
\p^{s'} g = \p^{s'} (\mathbf{I} - \mathbf{P} ) g = (\mathbf{I} - \mathbf{P} ) \p^{s'} g +  [\![ \mathbf{P}, \p^{s'} ]\!] g.
\Ee
The first term belongs to $\mathcal{N}^\perp$, and the second term belongs to $\mathcal{N}$ and is bounded by $\e (1+ \sum_{s''\le s'} \|\p^{s''} u\|_{L^\infty} ) \| \p^{s'-1} g \|_{L^\infty_x L^2_v } e^{-\varrho |v- \e u^\beta|^2 }$. In both cases, \eqref{R2point} is valid.
\end{proof}

Also, we can collect terms in \eqref{HE6} except for $\mu$ and $f_R$ by $\mathfrak{R}_3$: 
\begin{proposition}
Suppose that \eqref{ekappabeta} holds. Let $\mathfrak{R}_3$ be defined by
\Be \label{R3}
\e \kappa \sqrt{\mu} \mathfrak{R}_3 = 2 \e \kappa \tilde{u}^\beta \cdot (\underline{v} - \e u^\beta) \mu + \e^2 p^\beta \mu - \e^2 \kappa (\nabla_x u^\beta): \mathfrak{A} \sqrt{\mu} + \e^2 \kappa \tilde{p}^\beta \mu.
\Ee
Then for $\varrho \in (0,\frac{1}{4})$ and $s \le 2$, we have the following pointwise estimate:
\Be \label{R3point}
| \p^s \mathfrak{R}_3 | \lesssim V(\beta) e^{-\varrho |v- \e u^\beta|^2 }.
\Ee
\end{proposition}

\subsection{Remainder equation and its derivatives} We have simplified \eqref{HE1}-\eqref{HE6} so far. Finally, by dividing \eqref{HE1}-\eqref{HE6} by $\e^2 \sqrt{\mu}$, we obtain
\Be \label{Remainder}
\begin{split}
&\p_t f_R + \frac{\underline v}{\e} \cdot \nabla_x f_R + f_R \left ( \frac{(\p_t + \frac{\underline v}{\e} \cdot \nabla_x ) \sqrt{\mu} } {\sqrt{\mu} } \right ) + \frac{1}{\e^2 \kappa} L f_R\\
& = \frac{1}{\e\kappa} \Gamma (f_R, f_R) + \frac{1}{\e} \Gamma ( \mathfrak{R}_3, f_R) - \e \mathfrak{R}_1 - \frac{\kappa}{\e} \mathfrak{R}_2,
\end{split}
\Ee
where $\mathfrak{R}_1, \mathfrak{R}_2$, and $\mathfrak{R}_3$ are defined by \eqref{R1}, \eqref{R2}, and \eqref{R3} respectively.

Also, we have the equation for $\p^s f$, for $s \le 2$: by Lemma \ref{Momentstreambound}, 

\Be \label{pRemainder}
\begin{split}
& \p_t \p^s f_R + \frac{\underline{v}}{\e} \cdot \nabla_x \p^s f_R + \p^s f_R \left ( \frac{(\p_t + \frac{\underline v}{\e} \cdot \nabla_x ) \sqrt{\mu} } {\sqrt{\mu} }  \right ) + \frac{1}{\e^2\kappa} L \p^s f_R \\
&=  - \sum_{s' < s} \p^{s'} f_R \frac{1}{2} \sum_{i,j} (\p^{s-s'} \partial_{x_i} u^\beta_j ) \varphi_i \varphi_j  + R_s \\
& + \frac{1}{\e^2 \kappa} [\![\p^s, L]\!] f_R+ \frac{1}{\e \kappa} \p^s \Gamma (f_R, f_R) + \frac{1}{\e} \p^s \Gamma(\mathfrak{R}_3, f_R) \\
& - \e \p^s \mathfrak{R}_1 - \frac{\kappa}{\e} (\mathbf{I} - \mathbf{P} ) \p^s \mathfrak{R}_2 - \frac{\kappa}{\e} \mathbf{P} \p^s \mathfrak{R}_2,
\end{split}
\Ee
where $
|R_s | \lesssim \e V(\beta) \nu(v) \sum_{s'< s} |\p^{s'} f |.$

\subsection{Scaled $L^\infty$-estimate}
In this section, we prove a pointwise estimate (with a weight \eqref{weight}) of an $L^p$ solution of the linear Boltzmann equation with a force term. We consider the following transport equation with 
\eqref{stream} term: 
\Be \label{transport_f_stream}
[\partial_t + \e^{-1} \underline v \cdot \nabla_x ] f  + \frac{1}{\e^2 \kappa} L f 
- \frac{\left ( \partial_t + \frac{1}{\e} \underline v \cdot \nabla_x \right ) \sqrt{\mu} }{\sqrt{\mu} } f_R 
= \tilde H \text { in } [0, T] \times \O \times \mathbb{R}^3.
\Ee
Also, we have an issue of momentum stream: the remainder equations \eqref{Remainder} in our case contains the term 
\Be \label{stream}
\frac{\left ( \partial_t + \frac{1}{\e} \underline v \cdot \nabla_x \right ) \sqrt{\mu} }{\sqrt{\mu} } f_R
\Ee 
which cannot be controlled by $f_R$ for large $v$: this term precisely comes from that we expand around local Maxwellian, not global one. In \cite{JK2020}, a weight function of the form
\Be
w(x, v) := \exp ( \vartheta |v|^2 - Z(x) \cdot v),
\Ee
where $Z(x)$ is a suitable vector field, was introduced to bound \eqref{stream} term in the expansion around the local Maxwellian: 
\Be
w \left ( \partial_t + \frac{1}{\e} \underline v \cdot \nabla_x \right )  f_R =\left ( \partial_t + \frac{1}{\e} \underline v \cdot \nabla_x \right ) (w f_R) + \frac{1}{\e} \left ( \underline v \cdot \nabla_x Z(x) \cdot \underline v \right ) w f_R,
\Ee
and if $Z(x)$ is chosen so that $\underline v\cdot \nabla_x Z(x)\cdot \underline v > 0$ for any $v$ ($Z(x) = z(x) x$ for a suitably chosen function $z(x)$ works), one may control the most problematic term in \eqref{stream}: $(\nabla_x u^\beta : \underline v \otimes \underline v) w f_R.$  

Inspired by this, we introduce a suitable weight function, which is appropriate for periodic domain. Unlike the whole Euclidean space, existence of such $Z(x)$ in $\mathbb{T}^2$ is less obvious: in fact, if $Z = (Z_1, Z_2)$ is smooth, then since $\int_{\mathbb{T}^1}  \partial_1 Z_1 (x_1, x_2) dx_1 = 0$, $\partial_1 Z_1 $ will have a mixed sign along the circle $\mathbb{T}^1 \times \{x_2 \}$ for each $x_2 \in \mathbb{T}^1$, unless being 0 over whole circle. Thus, $\nabla_x Z + (\nabla_x Z)^T$ is neither positive definite nor negative definite over whole domain $\mathbb{T}^2$.

To overcome this difficulty, we introduce a weight function which cancels the most problematic term of \eqref{stream} instead of controlling it: we introduce
\Be \label{weight}
w(t,x,v) := \exp \left (\vartheta |v|^2 - \frac{1}{2} \e u^\beta (t,x) \cdot \underline v \right ),
\Ee
where $\vartheta \in (0, \frac{1}{4} )$, under the assumption 
\begin{equation} \label{smallnessubeta}
	\e | u^\beta (t,x) | = o(1). 
\end{equation}
In our scale regime \eqref{ekappabeta}, \eqref{smallnessubeta} holds.

\begin{proposition}\label{prop_infty}For an arbitrary $T>0$, suppose $f(t,x,v)$ is a distribution solution to \eqref{transport_f_stream}. Also, suppose that \eqref{ekappabeta} holds.

	Then, for $w= e^{\vartheta |v|^2 - \frac{1}{2} \e u^\beta (t,x) \cdot \underline v }$ with $\vartheta \in (0, \frac{1}{4})$ in \eqref{weight}, 	
	\Be \label{Ltinfty_2D}
	\begin{split}
		&	\e \kappa  \sup_{t \in [0,T]} 	\| wf (t) \|_{L^\infty (\T^2 \times \R^3)}\\
		\lesssim   	& \ 
		\e \kappa \| w f_0 \|_{L^\infty (\T^2 \times \R^3)}+ \sup_{t \in [0,T]}  \| f(s) \|_{L^2 (\T^2 \times \R^3)}
		+ \e^3 \kappa ^2 \sup_{t \in [0,T]} 	  	\| \nu^{-1}w\tilde H \|_{L^\infty (\T^2 \times \R^3)}.
	\end{split}	\Ee
		
\end{proposition}

The proof is based on the Duhamel formula \eqref{form:h} along the trajectory with scaled variables, and the $L^p$-$L^\infty$ interpolation argument based on the change of variable.


Let, with $w$ of \eqref{weight}, 
\Be\label{def:h}
h= w f.
\Ee
From \eqref{transport_f_stream}, we can write the evolution equation of $h$: 
\Be
\begin{split}\label{eqtn_h}
	&[ \partial_t + \frac{1}{\e} \underline v \cdot \nabla_x ] h = w [ \partial_t + \frac{1}{\e} \underline v \cdot \nabla_x ]  f + f[ \partial_t + \frac{1}{\e} \underline v \cdot \nabla_x ] w \\
	&= -\frac{1}{\e^2 \kappa} w L f + \frac{[ \partial_t + \frac{1}{\e} \underline v \cdot \nabla_x ]  \sqrt{\mu} }{\sqrt{\mu}} h
	+ w\tilde H + h [ \partial_t - \frac{1}{\e} \underline v \cdot \nabla_x ]   \frac{1}{2} \e u^\beta \cdot \underline v   \\
	&= - \frac{1}{\e^2 \kappa } w L f 
	+ w\tilde H 
\\
& \ \ 
	+ h \left ( -\frac{1}{2} (\underline v-\e u^\beta ) \cdot [ \partial_t + \frac{1}{\e} \underline v \cdot \nabla_x ] (-\e u^\beta ) - \frac{1}{2}  [ \partial_t + \frac{1}{\e} \underline v \cdot \nabla_x ]\e u^\beta \cdot \underline v \right ) \\
	&=  - \frac{1}{\e^2 \kappa } w L \Big(\frac{h}{w}\Big)
	 - h \left ( \frac{\e^2}{2} u^\beta \cdot \partial_t u^\beta + \frac{\e}{2} \underline v \cdot (\nabla_x u^\beta ) \cdot u^\beta \right ) + w\tilde H .
\end{split}
\Ee

Next, we recall that $Lf = \nu  f - Kf$ from \eqref{Lgammaform}. From the explicit form of $\nu$ in \eqref{Lgammaform}, we have a positive constant $\nu_0 >0$ such that 
\Be \label{nu0}
\nu_0 (|v-\e u^\beta| +1 ) \le \nu(v) \le 2 \nu_0 (|v-\e u^\beta |  + 1 ).
\Ee

In particular, \eqref{nu0} and \eqref{ekappabeta} implies that
\Be \label{nutilde}
\tilde{\nu} (t,x, v) := \nu(t,x,v) +  \frac{\e^4 \kappa }{2} u^\beta \cdot \partial_t u^\beta +  \frac{\e^3 \kappa }{2} \underline v \cdot (\nabla_x u^\beta ) \cdot u^\beta 
\Ee
satisfies
\Be \label{nutilde0}
\frac{1}{2}  \nu_0 (|v|+1) \le \tilde{\nu} (t,x,v) \le \frac{5}{2} \nu_0 (|v| + 1).
\Ee
With $\tilde{\nu}$, we can write the evolution equation for $h$:
\Be \label{hevolution}
\left ( \partial_t + \frac{1}{\e} \underline{v} \cdot \nabla_x \right ) h + \frac{1}{\e^2 \kappa} \tilde{\nu} h = \frac{1}{\e^2 \kappa} w K \frac{h}{w} + w\tilde H.
\Ee

Let $K_w h (v)= \int_{\R^3} \mathbf{k}_w (v,v_*) h(v_*) \dd v_*$ with $\mathbf{k}_w (v,v_*):= \mathbf{k}(v,v_*) \frac{w (v)}{w(v_*)}$. Then
\Be
w(v) K \frac{h}{w} (v) = \int_{\mathbb{R}^3} \mathbf{k} (v, v_*) \frac{w(v) }{w(v_*) } h(v_*) \dd v_* = K_w h (v).
\Ee
We will need the following estimate for $\mathbf{k}_w$: 
\begin{lemma}[Lemma 2 of\cite{JK2020}; also \cite{EGKM2}]
	
	Suppose that \eqref{smallnessubeta} holds. For $w= e^{\vartheta |v|^2 - \frac{1}{2} \e u^\beta \cdot v}$ with $\vartheta  \in (0, \frac{1}{4})$, there exists $C_{\vartheta }>0$ such that 
	\Be\label{k_w}
	\mathbf{k}_w(v,v_*)   \lesssim \frac{1}{|v-v_*| } e^{-C_{\vartheta } \frac{|v- v_*|^2}{2} } =: \mathbf{k}^\vartheta (v-v_*),
	\Ee
	\Be\label{int:k_w}
	\int_{\R^3} (1+|v-v_*|) \mathbf{k}_w(v,v_*) \dd v_* \lesssim \frac{1}{\nu(v)} \lesssim \frac{1}{1+|v|},
	\Ee
	\Be\label{int:k_w2}
	\int_{\R^3} \frac{1}{|v-v_*| } \mathbf{k}_w(v,v_*) \dd v_* \lesssim \frac{1}{\nu(v)} \lesssim 1.
	\Ee
	Note that $\mathbf{k}^\vartheta \in L^1 (\mathbb{R}^3)$. 
\end{lemma} 

We solve \eqref{hevolution} along the characteristics:
\hide 
\Be\Bs\label{Duhamel_linear}
&\frac{d}{ds}\Big\{ h(s , Y(s;t,x,v), v)
e^{- \frac{\nu(v)}{\e^2 \kappa} (t-s) }
\Big\}\\
=& \ \Big\{  \frac{1}{\e^2 \kappa}K_w h (s , Y(s;t,x,v), v)
+ H(s , Y(s;t,x,v), v)\Big\}
e^{-\frac{\nu(v)}{\e^2 \kappa}  (t-s) },
\end{split}\Ee

Along the characteristics,
\Be\Bs\label{Duhamel_linear}
&\frac{d}{ds}\Big\{ h(s , Y(s;t,x,v), v)
e^{- \frac{\nu(v)}{\e^2 \kappa} (t-s) }
\Big\}\\
=& \ \Big\{  \frac{1}{\e^2 \kappa}K_w h (s , Y(s;t,x,v), v)
+ H(s , Y(s;t,x,v), v)\Big\}
e^{-\frac{\nu(v)}{\e^2 \kappa}  (t-s) },
\end{split}\Ee
and hence  
\Be\Bs\label{form:h}
h(t,x,v) =& \ h_0 (Y(0;t,x,v), v) e^{-  
\frac{\nu(v)}{\e^2 \kappa}  t  }
\\
&+ \int^t_0 
\frac{ e^{
-  \frac{\nu(v)}{\e^2 \kappa} (t-s) }}{\e^2 \kappa }
\int_{\R^3} \mathbf{k}_w (v,u ) 
\underline{	h (s , Y(s;t,x,v), u) }
\dd u
\dd s\\
& + \int^t_0 
e^{
-  \frac{\nu(v)}{\e^2 \kappa} (t-s) }
H (s , Y(s;t,x,v), v) 
\dd s.
\end{split}\Ee

then we have\unhide

\Be\Bs\label{form:h}
&h(t,x,v) =  h_0 (Y(0;t,x, \underline v), v) \exp \left ({-  
\int_0 ^t \frac{\tilde{\nu}(\tau, Y(\tau; t, x, \underline v ), v)}{\e^2 \kappa}   \dd \tau   } \right )
\\
&+ \int^t_0 
\frac{ e^{ - \int_s^t  
\frac{\tilde{\nu}(\tau, Y(\tau; t, x, \underline v), v) \dd \tau }{\e^2 \kappa}  }}{\e^2 \kappa }
\int_{\R^3} \mathbf{k}_w (s, Y(s; t, x, \underline v), v, v_*  ) 
h (s , Y(s;t,x,\underline v), v_*) 
\dd v_*
\dd s\\
& + \int^t_0 
e^{
- \int_s ^t  \frac{\tilde{\nu}(\tau, Y(\tau; t, x, \underline v ), v) \dd \tau }{\e^2 \kappa}  }
(w \tilde H) (s , Y(s;t,x,\underline v), v) 
\dd s.
\end{split}\Ee

\hide
\begin{proposition}\label{prop_infty}For an arbitrary $T>0$, suppose $f(t,x,v)$ is a distribution solution to \eqref{transport_f} in $[0,T] \times \O \times  \R^3$.
	Then, for $w= e^{\vartheta |v|^2}$ with $\vartheta \in (0, \frac{1}{4})$, 	
	\Be \label{Linfty_2D}
	\begin{split}
		&	\e \kappa  \sup_{t \in [0,T]} 	\| wf (t) \|_{L^\infty (\O \times \R^3)}\\
	 \lesssim   	& \ 
			\e \kappa \| w f_0 \|_{L^\infty (\O \times \R^3)}+ \sup_{t \in [0,T]}  \| f(s) \|_{L^2 (\O \times \R^3)}
		\\
		& 
		+ \e^3 \kappa ^2 \sup_{t \in [0,T]} 	  	\| \nu^{-1}wH \|_{L^\infty (\O \times \R^3)}
		.  
	\end{split}	\Ee

\end{proposition}
The proof is based on (1) the Duhamel formula of \eqref{form:h} along the trajectory with scaled variables, and (2) the change of variable \eqref{COV}. 

First, 

Let 
\Be\label{def:h}
h= w f,
\Ee

and $K_w h (v)= \int_{\R^3} \mathbf{k}_w (v,u) h(u) \dd u$ with $\mathbf{k}_w (v,u):= \mathbf{k}(v,u) \frac{w (v)}{w(u)}$. 
\begin{lemma}[\cite{EGKM2}]
	
	For $w= e^{\vartheta |v|^2}$ with $\vartheta  \in (1, \frac{1}{4})$,  
	\Be\label{k_w}
	\mathbf{k}_w(v,u)   \lesssim 
	\Ee
	\Be\label{int:k_w}
	\int_{\R^3} \mathbf{k}_w(v,u) \dd u \lesssim \frac{1}{\nu(v)}.
	\Ee
\end{lemma}

Along the characteristics,
\Be\Bs\label{Duhamel_linear}
&\frac{d}{ds}\Big\{ h(s , Y(s;t,x,v), v)
e^{- \frac{\nu(v)}{\e^2 \kappa} (t-s) }
\Big\}\\
=& \ \Big\{  \frac{1}{\e^2 \kappa}K_w h (s , Y(s;t,x,v), v)
+ H(s , Y(s;t,x,v), v)\Big\}
e^{-\frac{\nu(v)}{\e^2 \kappa}  (t-s) },
\end{split}\Ee
and hence  
\Be\Bs\label{form:h}
h(t,x,v) =& \ h_0 (Y(0;t,x,v), v) e^{-  
\frac{\nu(v)}{\e^2 \kappa}  t  }
\\
&+ \int^t_0 
\frac{ e^{
	-  \frac{\nu(v)}{\e^2 \kappa} (t-s) }}{\e^2 \kappa }
\int_{\R^3} \mathbf{k}_w (v,u ) 
\underline{	h (s , Y(s;t,x,v), u) }
\dd u
\dd s\\
& + \int^t_0 
e^{
-  \frac{\nu(v)}{\e^2 \kappa} (t-s) }
H (s , Y(s;t,x,v), v) 
\dd s.
\end{split}\Ee

\unhide

\hide
Next we consider the change of variable formula:
\begin{lemma}
Fix $N>0$. For any $s\geq s^\prime \geq 0$ and $(y,\underline u) \in \O \times \{\underline u \in \R^2: |\underline u|<N\}$, the map 
\Be
s^\prime \mapsto Y(s^\prime;s,y,\underline u) 
\in\O
\Ee
is $m$-to-one, where $m \leq \max\big\{ \big(2N\frac{s-s^\prime}{\e } \big)^2, 1\big\}$.  There is a change of variables formula: for a non-negative function $A: z \in  \O   \mapsto A(z)  \geq 0$,
\Be\label{COV}
\int_{\{\underline u \in \R^2: |\underline u|< N\}} A( Y(s^\prime;s,y,\underline u)) \dd \underline u \leq 
\max\Big\{
N^2, \frac{\e^2}{|s-s^\prime|^2}
\Big\}
\int_{\O} A( z ) \dd z.
\Ee

\end{lemma}
\begin{proof}It suffices to show \eqref{COV}, while others are obvious. From $\det \big(\frac{\p Y (s^\prime; s, y, \underline u)}{\p \underline u}\big)= \frac{|s-s^\prime|^2}{\e^2}$, 
\Be
\begin{split}\notag
&	\int_{\{\underline u \in \R^2: |\underline u|< N\}} A( Y(s^\prime;s,y,\underline u)) \dd \underline u\\
&\leq  \max\Big\{ \big(2N\frac{s-s^\prime}{\e } \Big)^2 , 1 \big\}\int_{\O} A(z) \frac{\e^2}{|s-s^\prime|^2} \dd z\\
&=	\max\Big\{
N^2, \frac{\e^2}{|s-s^\prime|^2}
\Big\}
\int_{\O} A( z ) \dd z.
\end{split}
\Ee
%
\end{proof}

\unhide

\begin{proof}[\textbf{Proof of Proposition \ref{prop_infty}}]  

We again apply \eqref{form:h} to the second term in the right-hand side of \eqref{form:h}:
\Be\Bs\notag
&h(t,x,v) =
\ h_0 (Y(0;t,x, \underline v), v) \exp \left ({-  
\int_0 ^t \frac{\tilde{\nu}(\tau, Y(\tau; t, x, \underline v ), v)}{\e^2 \kappa}   \dd \tau   } \right )
\\
& + \int^t_0 
e^{
- \int_s ^t  \frac{\tilde{\nu}(\tau, Y(\tau; t, x, \underline v ), v) \dd \tau }{\e^2 \kappa}  }
(w H) (s , Y(s;t,x,\underline v), v) 
\dd s \\
& + \int_0 ^t \frac{ e^{ - \int_s^t  
	  \frac{\tilde{\nu}(\tau, Y(\tau; t, x, \underline v), v) \dd \tau }{\e^2 \kappa}  }}{\e^2 \kappa }
\int_{\R^3} \mathbf{k}_w (s, Y(s; t, x, \underline v), v, v_*  ) \\
& \ \ \times h_0 (Y(0; s, Y(s; t, x, \underline v), \underline v_* ), v_* ) e^{-\int_0 ^s \frac{\tilde{\nu} (\tau', Y(\tau'; s, Y(s; t, x, \underline v), \underline v_* ), v_* ) }{\e^2 \kappa } \dd \tau' } \dd v_* \dd s \\
&+ \int_0 ^t \frac{ e^{ - \int_s^t  
	  \frac{\tilde{\nu}(\tau, Y(\tau; t, x, \underline v), v) \dd \tau }{\e^2 \kappa}  }}{\e^2 \kappa }
\int_{\R^3} \mathbf{k}_w (s, Y(s; t, x, \underline v), v, v_*  ) \\
&  \ \ \times \int_0 ^s e^{-\int_{\tau} ^s \frac{ \tilde{\nu} (\tau', Y(\tau'; s, Y(s; t, x, \underline v), \underline v_*), v_* ) \dd \tau' }{\e^2 \kappa} } (wH) (\tau, Y(\tau; s, Y(s; t, x, \underline v), \underline v_*), v_* ) \dd \tau \dd v_* \dd s \\
& +  \int_0 ^t \frac{ e^{ - \int_s^t  
	  \frac{\tilde{\nu}(\tau, Y(\tau; t, x, \underline v), v) \dd \tau }{\e^2 \kappa}  }}{\e^2 \kappa }
\int_{\R^3} \mathbf{k}_w (s, Y(s; t, x, \underline v), v, v_*  ) \\
& \ \ \times \int_0 ^s \frac{e^{-\int_\tau ^s \frac{\tilde{\nu} (\tau', Y(\tau'; s, Y(s; t, x, \underline v), \underline v_* ), v_* ) \dd \tau' }{\e^2 \kappa } } }{\e^2 \kappa } \int_{\mathbb{R}^3} \mathbf{k}_w (\tau, Y(\tau; s, Y(s; t, x, \underline v ), \underline v_* ), v_*, v_{**} ) \\
&  \ \ \times h (\tau, Y(\tau; s, Y(s; t, x, \underline v ), \underline v_* ), v_{**} ) \dd v_{**} \dd \tau \dd v_* \dd s \\
& =: I_1 + I_2 + I_3 + I_4 + I_K.
\end{split}\Ee

First, we control $I_0 := I_1 + I_3$, contribution from initial data. We easily notice from \eqref{nutilde0}, \eqref{int:k_w} that
\Be\notag
\begin{split}
& |I_1| \le \| h_0 \|_{L^\infty(\O \times \mathbb{R}^3)} e^{- \frac{\nu_0 (|v|+1) t }{2 \e^2 \kappa } }\le \| h_0 \|_{L^\infty(\O \times \mathbb{R}^3)}, \\
& |I_3 | \le \int_0 ^t \frac{e^{-\frac{\nu_0 (|v|+1) (t-s) }{2 \e^2 \kappa } }}{\e^2 \kappa } \int_{\mathbb{R}^3} \mathbf{k}_w ( v, v_* ) e^{- \frac{\nu_0 (|v_*| + 1)s  }{2 \e^2 \kappa } } \| h_0 \|_{L^\infty (\O \times \mathbb{R}^3 ) }  \dd v_* \dd s \lesssim \| h_0 \|_{L^\infty (\O \times \mathbb{R}^3 ) } .
\end{split}
\Ee
In the second inequality, the dependence of $\mathbf{k}_w$ on $t,x$ variable is omitted as the bound is uniform on them.

Next, we control $I_H := I_2 + I_4$, the contribution from source $H$. Again from \eqref{nutilde0}, \eqref{int:k_w} we have
\Be\notag
\begin{split}
&|I_2 | \le \int_0 ^t e^{- \frac{\nu_0 (|v|+1) (t-s) }{2 \e^2 \kappa } } | w H (v) | \dd s \lesssim \e^2 \kappa  \| \nu^{-1} w H \|_{L^\infty ([0, T] \times \O \times \mathbb{R}^3 ) } , \\
&|I_4| \le \int_0 ^t \frac{e^{-\frac{\nu_0 (|v|+1) (t-s) }{2 \e^2 \kappa } }}{\e^2 \kappa } \int_{\mathbb{R}^3} \mathbf{k}_w ( v, v_* ) \int_0 ^s e^{- \frac{\nu_0 (|v_*| + 1)(s-\tau)  }{2 \e^2 \kappa } } | wH  (v_*) |\dd \tau  \dd v_* \dd s \\
& \lesssim \e^2 \kappa \| \nu^{-1} wH \|_{L^\infty ([0, T] \times \O \times \mathbb{R}^3 ) }.
\end{split}
\Ee

Finally, we control $I_K$. The idea is the following: we decompose the time interval $[0, s]$ into $[0, s-\e^2 \kappa o(1)]$ and $[s - \e^2 \kappa o(1) , s]$: the first integral is controlled using the change of variables $\underline v_* \rightarrow Y(\tau; s, Y(s; t, x, \underline v ), \underline v_*)$ and thus we can rewrite the integral of $h$ with respect to $v_*, v_{**}$ variables into the space-time integral of $f$: for that reason we plugged \eqref{form:h} into itself. Also, the splitting of the time gives control of the Jacobian factor obtained from change of variables. On the other hand, the second term is controlled by the fact that it is a short time integral: this gives smallness and thus we can bound the integral with $o(1) \|h\|_{L^\infty ([0, T] \times \O \times \mathbb{R}^3 )}$. 

For this purpose, we introduce a small positive number $\eta>0$, which is to be determined. Using  \eqref{nutilde0}, \eqref{int:k_w}, we have the following:

\Be
\begin{split}\notag
& |I_K| \le \int_0 ^t \frac{e^{-\frac{\nu_0 (|v|+1) (t-s) }{2\e^2 \kappa} } }{\e^2 \kappa} \int_0 ^s \frac{e^{-\frac{\nu_0 (|v_*|+1) (s-\tau ) }{2\e^2\kappa} } }{\e^2 \kappa } \\
& \  \times \int_{\mathbb{R}^3} \int_{\mathbb{R}^3} \mathbf{k}^\vartheta (v- v_*) \mathbf{k}^\vartheta (v_* - v_{**} ) | h (\tau, Y(\tau; s, Y(s; t, x, \underline v), \underline v_*), v_{**} ) | \dd v_{**} \dd v_* \dd \tau \dd s \\
& = \int_0 ^t \frac{e^{-\frac{\nu_0 (|v|+1) (t-s) }{2\e^2 \kappa} } }{\e^2 \kappa} \int_0 ^{s- \e^2 \kappa \eta} \frac{e^{-\frac{\nu_0 (|v_*|+1) (s-\tau ) }{2\e^2\kappa} } }{\e^2 \kappa } \\
 &  \ \times \int_{\mathbb{R}^3} \int_{\mathbb{R}^3} \mathbf{k}^\vartheta (v- v_*) \mathbf{k}^\vartheta (v_* - v_{**} ) | h (\tau, Y(\tau; s, Y(s; t, x, \underline v), \underline v_*), v_{**} ) | \dd v_{**} \dd v_* \dd \tau \dd s \\
&+ \int_0 ^t \frac{e^{-\frac{\nu_0 (|v|+1) (t-s) }{2\e^2 \kappa} } }{\e^2 \kappa} \int_{s - \e^2 \kappa \eta}  ^s \frac{e^{-\frac{\nu_0 (|v_*|+1) (s-\tau ) }{2\e^2\kappa} } }{\e^2 \kappa } \\
& \  \times \int_{\mathbb{R}^3} \int_{\mathbb{R}^3} \mathbf{k}^\vartheta (v- v_*) \mathbf{k}^\vartheta (v_* - v_{**} ) | h (\tau, Y(\tau; s, Y(s; t, x, \underline v), \underline v_*), v_{**} ) | \dd v_{**} \dd v_* \dd \tau \dd s \\
&=: I_{5,1} + I_{5,2}.
\end{split}
\Ee

We first bound $I_{5,2}$. From the integrability of $\mathbf{k}^\vartheta$ we have

\Be\notag
I_{5,2} \le \int_0 ^t  \frac{e^{-\frac{\nu_0 (|v|+1) (t-s) }{2\e^2 \kappa} } }{\e^2 \kappa} \dd s\frac{\e^2 \kappa \eta}{\e^2 \kappa} \| \mathbf{k}^\vartheta \|_{L^1 (\mathbb{R}^3) }^2 \| h \|_{L^\infty ([0, T] \times \O \times \mathbb{R}^3 ) } \lesssim \eta \| h \|_{L^\infty ([0, T] \times \O \times \mathbb{R}^3 ) } .
\Ee

Next, to treat $I_{5,1}$, we introduce the following decomposition of $\mathbf{k}^\vartheta (v - v_*)$: for a given $N>0$, 
\Be\notag
\begin{split}
&\mathbf{k}^\vartheta (v- v_*) = \mathbf{k}_N ^\vartheta (v , v_*) + \mathbf{k}_R ^\vartheta (v, v_*), \text{ where } \\
&\mathbf{k}_N ^\vartheta (v , v_*) = \mathbf{k}^\vartheta (v- v_*) \mathbf{1}_{B_N (0) \setminus B_{\frac{1}{N} } (0) } (v-v_*) \mathbf{1}_{B_N (0) } (v_*), \text{ and } \\
& \mathbf{k}_R ^\vartheta (v, v_*) = \mathbf{k}^\vartheta (v- v_*) - \mathbf{k}_N ^\vartheta (v , v_*).
\end{split}
\Ee

With this decomposition, we can split $I_{5,1}$ by
\Be
\begin{split}\notag
& I_{5,1} =  \int_0 ^t \frac{e^{-\frac{\nu_0 (|v|+1) (t-s) }{2\e^2 \kappa} } }{\e^2 \kappa} \int_0 ^{s- \e^2 \kappa \eta} \frac{e^{-\frac{\nu_0 (|v_*|+1) (s-\tau ) }{2\e^2\kappa} } }{\e^2 \kappa } \\
& \  \times \int_{\mathbb{R}^3} \int_{\mathbb{R}^3}
 \mathbf{k}_N ^\vartheta (v, v_*) \mathbf{k}_N ^\vartheta (v_*, v_{**} )
 | h (\tau, Y(\tau; s, Y(s; t, x, \underline v), \underline v_*), v_{**} ) | \dd v_{**} \dd v_* \dd \tau \dd s  \\
& + \int_0 ^t \frac{e^{-\frac{\nu_0 (|v|+1) (t-s) }{2\e^2 \kappa} } }{\e^2 \kappa} \int_0 ^{s- \e^2 \kappa \eta} \frac{e^{-\frac{\nu_0 (|v_*|+1) (s-\tau ) }{2\e^2\kappa} } }{\e^2 \kappa } \\
& \  \times \int_{\mathbb{R}^3} \int_{\mathbb{R}^3}
 \mathbf{k}_N ^\vartheta (v, v_*) \mathbf{k}_R ^\vartheta (v_* , v_{**} )
 | h (\tau, Y(\tau; s, Y(s; t, x, \underline v), \underline v_*), v_{**} ) | \dd v_{**} \dd v_* \dd \tau \dd s  \\
& + \int_0 ^t \frac{e^{-\frac{\nu_0 (|v|+1) (t-s) }{2\e^2 \kappa} } }{\e^2 \kappa} \int_0 ^{s- \e^2 \kappa \eta} \frac{e^{-\frac{\nu_0 (|v_*|+1) (s-\tau ) }{2\e^2\kappa} } }{\e^2 \kappa } \\
& \  \times \int_{\mathbb{R}^3} \int_{\mathbb{R}^3}
 \mathbf{k}_R ^\vartheta (v,v_*) \mathbf{k}_N ^\vartheta (v_* , v_{**} )
 | h (\tau, Y(\tau; s, Y(s; t, x, \underline v), \underline v_*), v_{**} ) | \dd v_{**} \dd v_* \dd \tau \dd s  \\
& + \int_0 ^t \frac{e^{-\frac{\nu_0 (|v|+1) (t-s) }{2\e^2 \kappa} } }{\e^2 \kappa} \int_0 ^{s- \e^2 \kappa \eta} \frac{e^{-\frac{\nu_0 (|v_*|+1) (s-\tau ) }{2\e^2\kappa} } }{\e^2 \kappa } \\
& \  \times \int_{\mathbb{R}^3} \int_{\mathbb{R}^3}
 \mathbf{k}_R ^\vartheta (v, v_*) \mathbf{k}_R ^\vartheta (v_*, v_{**} )
 | h (\tau, Y(\tau; s, Y(s; t, x, \underline v), \underline v_*), v_{**} ) | \dd v_{**} \dd v_* \dd \tau \dd s  \\
& =: I_{5,1} ^{NN} + I_{5,1} ^{NR} + I_{5,1} ^{RN} + I_{5,1} ^{RR}.
\end{split}
\Ee

Since $\int_{\mathbb{R}^3} \mathbf{k}_N ^\vartheta (v, v_*) \dd v_* \uparrow \| \mathbf{k}^\vartheta \|_{L^1 (\mathbb{R}^3)}$ as $N\rightarrow \infty$ and thus $A_N := \int_{\mathbb{R}^3} \mathbf{k}^\vartheta _R (v, v_{*} ) \dd v_* \rightarrow 0$ as $N\rightarrow \infty$ by Monotone convergence theorem, we have
\Be
\begin{split}\notag
&I_{5,1} ^{NR} \lesssim A_N \| \mathbf{k}^\vartheta \|_{L^1 (\mathbb{R}^3)} \| h \|_{L^\infty  ([0, T] \times \O \times \mathbb{R}^3 ) }, \\
&I_{5,1} ^{RN}\lesssim A_N \| \mathbf{k}^\vartheta \|_{L^1 (\mathbb{R}^3)} \| h \|_{L^\infty  ([0, T] \times \O \times \mathbb{R}^3 ) }, \\
&I_{5,1}^{RR} \lesssim A_N^2  \| h \|_{L^\infty  ([0, T] \times \O \times \mathbb{R}^3 ) }.
\end{split}
\Ee

Finally, we estimate $I_{5,1} ^{NN}$. First, we recall that $\mathbf{k}^\vartheta _N (v, v_*)$ is supported on $\{\frac{1}{N} <  |v-v_*| < N \}$ and therefore is bounded by some constant $C_N$. Thus, we have 
\Be
\begin{split}\notag
&\mathbf{k}^\vartheta _N (v, v_*) \le C_N \mathbf{1}_{B_N (0) } (v_*), \\
&\mathbf{k}^\vartheta _N (v_*, v_{**}) \le C_N \mathbf{1}_{B_N (0) } (v_{**}).
\end{split}
\Ee
Next, we expand $|h(\tau, Y(\tau; s, Y(s; t, x, \underline v), \underline v_*), v_{**} )|$: in the support of $\mathbf{k}^\vartheta _N (v, v_*) \mathbf{k}^\vartheta _N (v_*, v_{**})$, $|v_*|, |v_{**} | < N$. Note that this implies $|\underline v_*| < N$ and $|(v_*)_3| < N$, where $v_* = (\underline v_*, (v_*)_3).$ Together with \eqref{smallnessubeta}, we have
\Be
\begin{split}\notag
& |h(\tau, Y(\tau; s, Y(s; t, x, \underline v), \underline v_*), v_{**} )| = |w f (\tau, Y(\tau; s, Y(s; t, x, \underline v), \underline v_*), v_{**} )| \\
& \le e^{\vartheta |v_{**}|^2 + \frac{1}{2} \e \| u^\beta\|_{L^\infty ([0, T] \times \O ) }| \underline{v_{**}} |} |f (\tau, Y(\tau; s, Y(s; t, x, \underline v), \underline v_*), v_{**} )| \\
& \le C_N |f (\tau, Y(\tau; s, Y(s; t, x, \underline v), \underline v_*), v_{**} )|.
\end{split}
\Ee
Also, we rewrite $Y(\tau; s, Y(s; t, x, \underline v), \underline v_*)$: we have
\Be\notag
Y(\tau; s, Y(s; t, x, \underline v), \underline v_*) = x - \frac{t-s}{\e} \underline v - \frac{s-\tau} {\e} \underline v_* / \mathbb{Z}^2 \in \O.
\Ee
Finally, we remark that since $\tau \in [0, s - \e^2 \kappa \eta]$, we have $\frac{s-\tau}{\e} > \e \kappa \eta$. Combining these altogether, for $\tau \in [0, s - \e^2 \kappa \eta]$, we have
\Be
\begin{split}\notag
&\int_{\mathbb{R}^3} \int_{\mathbb{R}^3}
 \mathbf{k}_N ^\vartheta (v, v_*) \mathbf{k}_N ^\vartheta (v_*, v_{**} )
 | h (\tau, Y(\tau; s, Y(s; t, x, \underline v), \underline v_*), v_{**} ) | \dd v_{**} \dd v_*   \\
& \lesssim_{N} \int_{|(v_*)_3 | < N}  \int_{|\underline v_*| < N } \int_{|v_{**}|< N } | f_R (\tau, x - \frac{t-s}{\e} \underline v - \frac{s-\tau} {\e} \underline v_* / \mathbb{Z}^2, v_{**} ) | \dd v_{**} \dd v_* \\
& \lesssim_{N} \left ( \int_{|\underline v_*| < N } \int_{\mathbb{R}^3} \left | f_R (\tau, x - \frac{t-s}{\e} \underline v - \frac{s-\tau} {\e} \underline v_* / \mathbb{Z}^2, v_{**} ) \right |^2 \dd v_{**} \dd v_* \right )^{\frac{1}{2} },
\end{split}
\Ee
where we have used that the integrand is independent of $(v_*)_3$ and $\| \mathbf{1}_{\{|\underline v_*| < N\} \times \{|v_{**}| < N \} } \|_{L^2 (\mathbb{R}^2 \times \mathbb{R}^3 ) } \lesssim_N 1$.

Next, we apply the change of variables $\underline v_* \rightarrow y = x - \frac{t-s}{\e} \underline v - \frac{s-\tau} {\e} \underline v_* \in \mathbb{R}^2$. This map is one-to-one, and maps $\underline v_* \in B_N (0)$ onto $y \in B_{\frac{s-\tau}{\e} N } (x - \frac{t-s}{\e} \underline v)$ with $\dd y = \left (\frac{s-\tau}{\e} \right )^2 \dd \underline v_*$. Therefore, we have
\Be \label{chvarinteg}
\begin{split}
& \left ( \int_{|\underline v_*| < N } \int_{\mathbb{R}^3} \left | f_R (\tau, x - \frac{t-s}{\e} \underline v - \frac{s-\tau} {\e} \underline v_* / \mathbb{Z}^2, v_{**} ) \right |^2 \dd v_{**} \dd v_* \right )^{\frac{1}{2} } \\
&= \left ( \int_{y \in B_{\frac{s-\tau}{\e} N } (x - \frac{t-s}{\e} \underline v) } \int_{\mathbb{R}^3} \left | f_R (\tau, y / \mathbb{Z}^2 , v_{** } ) \right |^2 \left ( \frac{\e }{s-\tau } \right )^2  \dd v_{**} \dd y \right  )^{\frac{1}{2} } \\
& = \left ( \sum_{k \in \mathbb{Z}^2} \int_{y \in ( \left [-\frac{1}{2},\frac{1}{2} \right ]^2 + k) \cap B_{\frac{s-\tau}{\e} N } (x - \frac{t-s}{\e} \underline v) } \int_{\mathbb{R}^3} \left | f_R (\tau, y-k , v_{** } ) \right |^2 \left ( \frac{\e }{s-\tau } \right )^2  \dd v_{**} \dd y \right  )^{\frac{1}{2} } \\
& = \left ( \sum_{k \in \mathbb{Z}^2} \int_{z \in  \left [-\frac{1}{2},\frac{1}{2} \right ]^2  \cap B_{\frac{s-\tau}{\e} N } (x - \frac{t-s}{\e} \underline v - k ) } \int_{\mathbb{R}^3} \left | f_R (\tau, z , v_{** } ) \right |^2 \left ( \frac{\e }{s-\tau } \right )^2  \dd v_{**} \dd z \right  )^{\frac{1}{2} },
\end{split}
\Ee
where $z = y-k$ in each integral. Next, we count the number of $k \in \mathbb{Z}^2$ such that $\left [-\frac{1}{2},\frac{1}{2} \right ]^2  \cap B_{\frac{s-\tau}{\e} N } (x - \frac{t-s}{\e} \underline v - k ) \ne \emptyset$. There are two cases: if $N \frac{s-\tau}{\e} \le 1$, there are $O(1)$ such $k \in \mathbb{Z}^2$. If $N \frac{s-\tau}{\e} > 1$, there are $O \left (\left (N \frac{s-\tau}{\e} \right )^2 \right ) $ such $k \in \mathbb{Z}^2$. Therefore, we have
\Be
\begin{split}\notag
&\eqref{chvarinteg} \lesssim \left ( \max \left ( \left (\frac{\e}{s-\tau} \right )^2, N^2 \right ) \int_{\O} \int_{\mathbb{R}^3}  |f_R(\tau, z, v_{**} ) |^2 \dd v_{**} \dd z \right )^{\frac{1}{2}} \\
& \lesssim_{N, \eta} \frac{1}{\e \kappa} \sup_{\tau \in [0, t]} \| f_R (\tau ) \|_{L^2 (\O \times \mathbb{R}^3 ) }.
\end{split}
\Ee
Choosing $N$ large enough and $\eta$ small enough so that we can bury $I_{5,2}$, $I_{5,1}^{NR}, I_{5,1}^{RN}$, and $I_{5,1}^{RR} $ gives
\Be\notag
\begin{split}
\| h \|_{L^\infty ([0, T] \times \O \times \mathbb{R}^3 )} \lesssim&\   \| h_0 \|_{L^\infty (\O \times \mathbb{R}^3)} + \e^2 \kappa \| \nu^{-1} w H \|_{L^\infty ([0, T] \times \O \times \mathbb{R}^3) }\\
& + \frac{1}{\e \kappa} \| f_R \|_{L^\infty ([0, T] ; L^2 (\O \times \mathbb{R}^3 ) ) },
\end{split}
\Ee
which is the desired conclusion.
\end{proof}

\subsection{Remainder Estimate}\label{sec:re}  To admit far-from-equilibrium initial data, we need to keep the characteristic size of remainder as large as possible. A heuristic calculation suggests that the size $o(\e \kappa)$ for the remainder is the threshold: if the remainder becomes of the size $O(\e \kappa)$, we lose the control for nonlinearity of the remainder equation. Thus, we aim to keep our characteristic size of remainder to be slightly smaller than $\e \kappa$. 

There is only a very slight room for this: the only possible gain is the coercivity of the linearized Boltzmann operator $L$. However, many conventional techniques (averaging lemma, $L^\infty$-estimates) do not rely on it; up to the authors' knowledge the coercivity of $L$ is exploited only in $L^2_v$ estimates. If we rely on other techniques in too early stage, we enormously lose the scale and fail to achieve the goal. 

As a consequence, we need to push the $L^2_v$ estimates as far as possible. The important observation made in \cite{Guo2002} is that even for nonlinear term, we have control by $L^2$-in-$v$ integral of remainders, since nonlinear term is also expressed in terms of an integral with nicely decaying kernel: what lacks is $L^2$ integrability in $x$. This observation naturally leads us to pursue $L^2_v$-estimate for derivatives of remainder and then rely on interpolation - $H^2_x$, but $L^2_v$ estimate. 

It turns out that this idea gives a sharper scaling than many conventional techniques: the commutator $[\![ \p^s, L ]\!]$ between spatial derivatives and $L$ forces us to lose $\sqrt{\kappa}$ scale for each derivative, but we do not lose scale in nonlinearity for 2-dimensional domain. Thus, by setting initial data decaying to 0 in arbitrary slow rate as $\e \rightarrow 0$, we can keep $L^2_x L^2_v$ norms of remainder and its derivatives small, provided that the source terms are also small, which is the main point of the next idea.

Furthermore, we note that $H^2_x L^2_v$ suits very well with our goal to see convergence in a stronger topology: as we can control up to second derivatives of remainder small, we can keep our Boltzmann solution close to the local Maxwellian $M_{1, \e u^\beta, 1}$. Its zeroth and first derivatives may converge - they correspond to the velocity and vorticity. Its second derivatives may blow in general, which represents the formation of singular object, e.g. interfaces.

Now we are ready to prove compactness of $f_R$ in a suitable topology, thereby proving convergence. For a fixed $T>0$ and $t \in (0, T)$, we use the following scaled energy and its dissipation:
\Be \label{ED}
\begin{split}
\mathcal{E}(t) &:= \sum_{s \le 2} \sup_{t' \in (0, t) }\| \kappa^{-1 + \frac{s}{2}} \p^s f_R (t') \|_{ L^2_x L^2 _v  } ^2, \\
\mathcal{D} (t) &:= \sum_{s \le 2} \| \e^{-1} \kappa^{-\frac{3}{2} + \frac{s}{2} } \nu^{\frac{1}{2} } (\mathbf{I} - \mathbf{P} ) \p^s f_R \|_{L^2 ((0, t); L^2 _x L^2_v )  ) } ^2.
\end{split}
\Ee
We also need the following auxilliary norm:
\Be \label{F}
\mathcal{F} (t) := \e \sup_{t' \in (0, t) } \| f_R (t') \|_{L^\infty (\mathbb{T}^2 \times \mathbb{R}^3 ) }. 
\Ee
Also, we will frequently use the following basic inequality:
\Be\notag
\sum_{s \le 2} \| \kappa^{-1 + \frac{s}{2}} \p^s f_R \|_{L^2 ( (0, t); L^2_x L^2 _v ) } ^2 = \int_0 ^t \mathcal{E} \lesssim_{T} \mathcal{E}(t).
\Ee

The main theorem of this section is the following.
\begin{theorem} \label{theo:remainder}
Let $T>0$. Suppose that $\delta_s = \delta_s (\e)$, $s = 0, 1, 2$ satisfy the following:
\begin{align}
&\lim_{\e \rightarrow 0} \delta_0 (\e) ^2  \left (\|\nabla_x u^\beta\|_{L^\infty{((0, T) \times \mathbb{T}^2 }) } ^2 + 2 \right ) \exp \left ( 2 \mathbf{C}_0 \left (\|\nabla_x u^\beta\|_{L^\infty{((0, T) \times \mathbb{T}^2 }) } ^2 + 2 \right ) T \right ) = 0, \label{delta0} \\
& \delta_s (\e) < (\e^{-1} \kappa^{-1/2} )^s, s = 1, 2 \label{deltas}.
\end{align}
Suppose that $f_R (0)$ satisfies 
\Be\notag
\sqrt{\mathcal{E}(0) }, \  \mathcal{F}(0) < \delta_0 (\e), \  \| w_s \p^s f_{R0} \|_{L^\infty (\mathbb{T}^2 \times \mathbb{R}^3 ) } < \delta_s (\e), s = 1, 2.
\Ee
Then \eqref{Remainder} with initial data $f_R(0)$, and $\tilde{u}^\beta (0) \equiv 0$ has a solution $f_R(t)$, $t \in (0, T)$ such that 
\Be
\begin{split}\notag
& 	\mathcal{E}(t) + \mathcal{D}(t)\\
	 \le & \ (\delta_0^2 + \kappa) (1+T)  C(\mathbf{C}_0)\\
	& \times  \left ( 2 \mathbf{C}_0 \left (\|\nabla_x u^\beta\|_{L^\infty((0, T) \times \mathbb{T}^2 ) } ^2 + 2 \right )   \exp \left ( 2 \mathbf{C}_0 \left (\|\nabla_x u^\beta\|_{L^\infty((0, T) \times \mathbb{T}^2 ) } ^2 + 2 \right ) T \right )  \right ) ,
\end{split}
\Ee
and $
\lim_{\e \rightarrow 0} \sup_{t \in (0, T) } (\mathcal{E}(t) + \mathcal{F}(t) ) = 0. $
\end{theorem}

\subsubsection{Energy estimate}

By taking $L^2$ norm for \eqref{Remainder}, \eqref{pRemainder} for $s \le 2$ and integrating over time, we have
\begin{align}
\mathcal{E}(t) &+ \mathcal{D}(t)  \lesssim  \mathcal{E}(0)\notag \\
&+ \| \nabla_x u^\beta\|_{L^\infty_{t,x} } \sum_{s \le 2} \int_{(0, t) \times \mathbb{T}^2 \times \mathbb{R}^3} \left |\frac{\p^s f_R}{\kappa^{1- \frac{s}{2} } } \right |^2 \nu^2 \dd v \dd x \dd t'  \label{momstr1}\\
& + \sum_{s' < s} \kappa^{\frac{s-s'}{2} } V(\beta) \int_{(0, t) \times \mathbb{T}^2 \times \mathbb{R}^3} \left |\frac{\p^{s'} f_R}{\kappa^{1- \frac{s'}{2} } } \right | \left |\frac{\p^s f_R}{\kappa^{1- \frac{s}{2} } } \right | \nu^2 \dd v \dd x \dd t' + \e V(\beta) (\mathcal{E}(t) + \mathcal{D}(t) ) \label{momstr2} \\
&+ \sum_{s \le 2} \int_{(0, t) \times \mathbb{T}^2 \times \mathbb{R}^3} \e^{-2}  \kappa^{-2 + \frac{s}{2}} [\![ \p^s , L ]\!] f_R \frac{\p^s f_R}{\kappa^{1 - \frac{s}{2} } } \dd v \dd x \dd t' \label{Lcommutator} \\
&+ \sum_{s \le 2}\int_{(0, t) \times \mathbb{T}^2 \times \mathbb{R}^3} \e^{-1} \kappa^{-2 + \frac{s}{2} } \p^s \Gamma ( f_R, f_R ) \frac{\p^s f_R}{\kappa^{1 - \frac{s}{2} } } \dd v \dd x \dd t' \label{nonlinearity} \\
& + \sum_{s \le 2} \int_{(0, t) \times \mathbb{T}^2 \times \mathbb{R}^3} \e^{-1} \kappa^{-1 + \frac{s}{2}} \p^s \Gamma (\mathfrak{R}_3, f_R)\frac{\p^s f_R}{\kappa^{1 - \frac{s}{2} } } \dd v \dd x \dd t' \label{linear} \\
&- \sum_{s \le 2}\int_{(0, t) \times \mathbb{T}^2 \times \mathbb{R}^3} \kappa^{-1 + \frac{s}{2} } \left (\e \p^s \mathfrak{R}_1 + \frac{\kappa}{\e} (\mathbf{I} - \mathbf{P} ) \p^s \mathfrak{R}_2 + \frac{\kappa}{\e} \p^s \mathfrak{R}_2 \right ) \frac{\p^s f_R}{\kappa^{1 - \frac{s}{2} } } \dd v \dd x \dd t' . \label{source}
\end{align}

\textit{Step 1. Control of \eqref{source}.} From \eqref{R1point} and \eqref{R2point}, we have
\Be \label{sourceestimate}
\eqref{source} \lesssim \sum_{s \le 2} \left ( \e \kappa^{-1 + \frac{s}{2} } V(\beta) \sqrt{\mathcal{E}(t) } + \kappa^{\frac{3}{2} } V(\beta) \sqrt{\mathcal{D}(t) } + \kappa V(\beta) \sqrt{\mathcal{E}(t)} \right ) \lesssim \kappa^{\frac{1}{2}} \left (\sqrt{\mathcal{E}(t) } + \sqrt{\mathcal{D}(t) } \right ),
\Ee
by \eqref{ekappabeta}.

\textit{Step 2. Control of \eqref{linear}.} We note that
\Be\notag
\p^s \Gamma(\mathfrak{R}_3, f_R ) = \sum_{s_1 + s_2 + s_3 = s} {}_{\p^{s_1}} \Gamma (\p^{s_2} \mathfrak{R}_3 , \p^{s_3} f_R ).
\Ee
There are two cases. First, if $s_1 = 0$, then 
\Be
\begin{split}\notag
\sum_{s_2 + s_3 = s} & \int_{(0, t) \times \mathbb{T}^2 \times \mathbb{R}^3}  \e^{-1} \kappa^{-1 + \frac{s}{2} } \Gamma( \p^{s_2} \mathfrak{R}_3, \p^{s_3} f_R) \frac{\p^s f_R}{\kappa^{1 - \frac{s}{2} } } \dd v \dd x \dd t' \\
& \lesssim \sum_{ s_3 \le s} \kappa^{-\frac{1}{2} + \frac{s}{2} } V(\beta) \| \p^{s_3} f_R \|_{L^2 ((0, t); L^2_x L^2_v ) } \sqrt{\mathcal{D}(t) } \lesssim \kappa^{\frac{1}{2} } V(\beta) \sqrt{\mathcal{E}(t) }\sqrt{\mathcal{D}(t) }.
\end{split}
\Ee
If $s_1 \ge 1$, then by Lemma \ref{lemma_L} we have
\Be\notag
\begin{split}
\sum_{s_1 +s_2 + s_3 = s} & \int_{(0, t) \times \mathbb{T}^2 \times \mathbb{R}^3}  \e^{-1} \kappa^{-1 + \frac{s}{2} } {}_{\p^{s_1} }\Gamma( \p^{s_2} \mathfrak{R}_3, \p^{s_3} f_R) \frac{\p^s f_R}{\kappa^{1 - \frac{s}{2} } } \dd v \dd x \dd t' \\
& \lesssim \sum_{s_3 < s} V(\beta) \kappa^{-1 + \frac{s}{2}} \left ( \| \p^{s_3} f_R \|_{L^2((0, t); L^2_x L^2_v)} + \| \nu^{\frac{1}{2} } (\mathbf{I} - \mathbf{P }) \p^{s_3} f_R \|_{L^2((0, T); L^2_x L^2_v ) } \right ) \\
& \times \left ( \sqrt{\mathcal{E}(t) } + \e \kappa^{\frac{1}{2} } \sqrt{\mathcal{D}(t) } \right ) \lesssim \kappa^{\frac{1}{2} } V(\beta) (\mathcal{E}(t) + \e^2 \kappa \mathcal{D}(t)),
\end{split}
\Ee
since $s_3 < s$. In conclusion, we have
\Be \label{linearestimate}
\eqref{linear} \lesssim \kappa^{\frac{1}{2} }V(\beta) (\mathcal{E}(t) + \mathcal{D}(t) ).
\Ee
\textit{Step 3. Control of \eqref{Lcommutator}.} For $s=0$, $[\![\p^s, L ]\!]=0$. When $s=1$, $[\![\p^s, L ]\!] f_R$ consists of type 1 and type 2 of terms in Lemma \ref{Lpcomm}. When $s=2$, there are exactly one term in $[\![\p^s, L ]\!] f_R$ which is of type 3 in Lemma \ref{Lpcomm}: ${}_{\p} L [\![ \mathbf{P}, \p ]\!] f_R$. For a given $s \le 2$ and type 1 term in Lemma \ref{Lpcomm}, we have an upper bound
\Be \label{Lpcomuppb}
\left ( \|\nabla_x u^\beta\|_{L^\infty_{t,x} } + \kappa^{\frac{1}{2} } V(\beta) \right ) \sqrt{\mathcal{D}(t) } \left ( \sqrt{\int_0 ^t \mathcal{E}  } + \e \kappa^{\frac{1}{2} } \sqrt{\mathcal{D} (t) } \right ) 
\lesssim ( \|\nabla_x u^\beta\|_{L^\infty_{t,x} }^2 + 1) \int_0 ^t \mathcal{E} + o(1) \mathcal{D}(t),
\Ee
where the first $\|\nabla_x u^\beta\|_{L^\infty_{t,x}}$ term corresponds to ${}_{\p^{s_1} } L (\mathbf{I} - \mathbf{P}) \p^{s_2} f_R$ and the second $\kappa^{\frac{1}{2} } V(\beta) $ term corresponds to ${}_{\p^2} L (\mathbf{I} - \mathbf{P} )  f_R$. For example, for $s=2$ with ${}_{\p} L (\mathbf{I} - \mathbf{P}) \p f_R$ term, we have
\Be
\begin{split}\notag
&\int_{(0, t) \times \mathbb{T}^2 \times \mathbb{R}^3 }  \e^{-2} \kappa^{-1}  {}_{\p} L (\mathbf{I} - \mathbf{P} ) \p f_R \p^2 f_R \dd v \dd x \dd t'  \lesssim \| \nabla_x u^\beta \|_{L^\infty_{t,x} } \| \e^{-1} \kappa^{-1} \nu^{\frac{1}{2} } (\mathbf{I} - \mathbf{P} ) \p f_R \|_{L^2((0, t); L^2_x L^2 _v ) }  \\
& \times \left ( \|\p^2 f_R \|_{L^2 ((0, t); L^2_x L^2_v ) } + \e \kappa^{\frac{1}{2} } \|\e^{-1} \kappa^{-\frac{1}{2} } \nu^{\frac{1}{2} } (\mathbf{I} - \mathbf{P} ) \p^2 f_R \|_{L^2 ((0, t); L^2_x L^2_v ) } \right )
\end{split}
\Ee
which is bounded by the right-hand side of \eqref{Lpcomuppb}.

For a given $s \le 2$ and type 2 term in Lemma \ref{Lpcomm}, we have a similar upper bound
\Be\notag
\sum \e^{-1} \kappa^{-\frac{3}{2} + \frac{s}{2}} \| \p \cdots [\![\mathbf{P}, \p ]\!] \cdots \p f_R \|_{L^2((0, t); L^2_x L^2_v)} \sqrt{\mathcal{D}(t) },
\Ee
where summation is over possible combinations of $\p \cdots [\![\mathbf{P}, \p ]\!] \cdots \p$, consisting of $s-1$ $\p$ and one $[\![\mathbf{P}, \p ]\!]$. We note that 
\Be\notag
\begin{split}
&\| \p \cdots [\![\mathbf{P}, \p ]\!] \cdots \p f_R \|_{L^2((0, t); L^2_x L^2_v)} \\
& \lesssim \e \left ( \| \nabla_x u^\beta \|_{L^\infty_{t,x} } \| \p^{s-1} f_R \|_{L^2((0, t); L^2_x L^2_v)} + V(\beta) \sum_{s' < s-1} \| \p^{s'} f_R  \|_{L^2((0, t); L^2_x L^2_v)}\right ),
\end{split}
\Ee
where the former term corresponds to the case that all $s-1$ derivatives $\p$ are applied to $f_R$, and the latter corresponds to the others. Thus, again, we have a bound
\Be\notag
\left ( \|\nabla_x u^\beta\|_{L^\infty_{t,x} } + \kappa^{\frac{1}{2} } V(\beta) \right ) \sqrt{\int_0 ^t \mathcal{E} }  \sqrt{\mathcal{D} (t) } 
\lesssim ( \|\nabla_x u^\beta\|_{L^\infty_{t,x} }^2 + 1) \int_0 ^t \mathcal{E} + o(1) \mathcal{D}(t),
\Ee

Finally, for a type 3 term in Lemma \ref{Lpcomm} (which immediately implies $s=2$), we have
\Be\notag
\| \nabla_x u^\beta \|_{L^\infty_{t,x} }^2 \left \| \frac{f_R}{\kappa} \right \|_{L^2((0, t); L^2_x L^2_v ) } \sqrt{\int_0 ^t \mathcal{E} } \lesssim \| \nabla_x u^\beta \|_{L^\infty_{t,x} }^2 \int_0 ^t \mathcal{E}.
\Ee
To summarize, we have
\Be\label{Lcommutatorestimate}
\eqref{Lcommutator} \lesssim ( \|\nabla_x u^\beta\|_{L^\infty_{t,x} }^2 + 1) \int_0 ^t \mathcal{E} + o(1) \mathcal{D}(t).
\Ee

\textit{Step 4. Control of \eqref{momstr1}, \eqref{momstr2}. }We use the following standard estimate: let $0<\vartheta_2<\vartheta_1<\vartheta_0<\frac{1}{4}$, and let 
\Be \label{weightpfr}
w_j = e^{\vartheta_j|v|^2 - \frac{1}{2} \e u^\beta \cdot \underline v }, j=0, 1, 2.
\Ee
For $s \le 2$, we have
\Be\notag
\begin{split}
&\int_{(0, t) \times \mathbb{T}^2 \times \mathbb{R}^3} \left | \frac{\p^{s} f_R}{\kappa^{1-\frac{s}{2} } } \right |^2 \nu^2 \dd v \dd x \dd t' \lesssim \left \| \frac{ \mathbf{P} \p^s f_R }{\kappa^{1- \frac{s}{2} } } \right \|_{L^2 ((0, t); L^2_x L^2_v ) }^2 + \left \| \nu^1 \frac{ (\mathbf{I} -\mathbf{P}) \p^s f_R }{\kappa^{1- \frac{s}{2} } } \right \|_{L^2 ((0, t); L^2_x L^2_v ) }^2 \\
& \lesssim \int_0 ^t \mathcal{E} + \left \|\mathbf{1}_{ \{ |v - \e u^\beta|>(\e \sqrt{\kappa} )^{-o(1) } \} }\nu^1 \frac{ (\mathbf{I} -\mathbf{P}) \p^s f_R }{\kappa^{1- \frac{s}{2} } }\right   \|_{L^2 ((0, t); L^2_x L^2_v ) }^2 \\
&+ \left \|\mathbf{1}_{ \{ |v - \e u^\beta|\le (\e \sqrt{\kappa} )^{-o(1) } \} }\nu^1 \frac{ (\mathbf{I} -\mathbf{P}) \p^s f_R }{\kappa^{1- \frac{s}{2} } }\right   \|_{L^2 ((0, t); L^2_x L^2_v ) }^2  \\
& \lesssim \int_0 ^t \mathcal{E} + \|\mathbf{1}_{ \{ |v - \e u^\beta|>(\e \sqrt{\kappa} )^{-o(1) } \} } \nu^1 w_s^{-1}  \|_{L^2 ((0, t); L^2_x L^2_v ) }^2 \| w_s \p^s f \|_{L^\infty((0, t); L^\infty(\mathbb{T}^2 \times \mathbb{R}^3 ) )} ^2 \\
& \ \  + \left (\e \sqrt{\kappa} \right )^{1 - o(1) } \mathcal{D}(t) \\
&  \lesssim \int_0 ^t \mathcal{E} +\left (\e \sqrt{\kappa} \right )^{1 - o(1) } \mathcal{D}(t) + e^{-\frac{1}{(\e \sqrt{\kappa})^{o(1)}} } \| w_s \p^s f \|_{L^\infty((0, t); L^\infty(\mathbb{T}^2 \times \mathbb{R}^3 ) )} ^2.
\end{split}
\Ee
Similar calculation for \eqref{momstr2} gives the following bound:
\Be \label{momestimate}
\eqref{momstr1} + \eqref{momstr2} \lesssim (1 + \| \nabla_x u^\beta \|_{L^\infty_{t,x} } ) \int_0 ^t \mathcal{E} + o(1) \mathcal{D}(t)  + e^{-\frac{1}{(\e \sqrt{\kappa})^{o(1)}} } \sum_{s \le 2} \| w_s \p^s f \|_{L^\infty((0, t); L^\infty(\mathbb{T}^2 \times \mathbb{R}^3 ) )} ^2.
\Ee

\textit{Step 5. Control of \eqref{nonlinearity}.} Finally, we control the nonlinear contribution \eqref{nonlinearity}: here we use anisotropic interpolation result (Lemma \ref{anint}) and Lemma \ref{lemma_L}. First, from Lemma \ref{nonhydroLp} and \eqref{ekappabeta}, we remark that
\Be\notag
\begin{split}
\left \| \nu^{\frac{1}{2} } (\mathbf{I} - \mathbf{P} ) \frac{f}{\kappa} \right \|_{L^2 ((0, t); L^4_x L^2_v ) } &+ \left \| \nu^{\frac{1}{2} } (\mathbf{I} - \mathbf{P} ) \frac{\p f}{\sqrt{\kappa} } \right \|_{L^2 ((0, t); L^4_x L^2_v ) } + \left \| \nu^{\frac{1}{2} } (\mathbf{I} - \mathbf{P} ) \frac{f}{\kappa} \right \|_{L^2 ((0, t); L^\infty_x L^2_v ) }\\
& \lesssim \e^{\frac{1}{2} } \left ( \sqrt{\mathcal{D}(t) } + \sqrt{\int_0 ^t \mathcal{E} } \right ).
\end{split}
\Ee
Next, we estimate the following integrals: first, we estimate
\Be
\begin{split}\notag
&\int_{(0, t) \times \mathbb{T}^2 \times \mathbb{R}^3} \frac{1}{\e \kappa^2} \Gamma(f_R, f_R) \frac{f_R}{\kappa} \dd v \dd x \dd t' \\
&  \lesssim \left ( \left \| \frac{f_R}{\kappa} \right \|_{L^2_{tvx}} + \left \| \nu^{\frac{1}{2} } (\mathbf{I} - \mathbf{P} ) \frac{f_R}{\kappa} \right \|_{L^2_{tvx} } \right ) \sqrt{\kappa}^{-1} \| f_R \|_{L^\infty_{tx} L^2_v} \sqrt{\mathcal{D} (t) } \\
& \lesssim \left ( \sqrt{\int_0 ^t \mathcal{E} } + \e \sqrt{\kappa} \sqrt{D(t) } \right ) \sqrt{D(t) } \left \| \frac{f_R}{\kappa} \right \|_{L^\infty_t L^2_{xv} } ^{\frac{1}{2} } \| \p^2 f_R \|_{L^\infty _t L^2_{xv} } ^{\frac{1}{2} } \lesssim \left ( \int_0 ^t \mathcal{E} + \mathcal{D}(t) \right ) \sqrt{\mathcal{E}(t) }.
\end{split}
\Ee
In a similar fashion, we see
\begin{align}\notag
\begin{split}
&\int_{(0, t) \times \mathbb{T}^2 \times \mathbb{R}^3} \frac{1}{\e \kappa^{2 - \frac{s}{2} }} \Gamma(\p^s f_R, f_R) \frac{\p^s f_R}{\kappa^{1 - \frac{s}{2}}} \dd v \dd x \dd t' \\
&\lesssim \sqrt{\mathcal{D}(t) } \frac{1}{\kappa^{\frac{3-s}{2} }} \left [ \left ( \| \p^s f_R \|_{L^2_{txv}} + \| \nu^{\frac{1}{2} } (\mathbf{I} - \mathbf{P} ) \p^s f_R \|_{L^2_{txv}} \right ) \| f_R \|_{L^\infty_{tx} L^2_v } \right. \\
&\left. +  \left ( \|  f_R \|_{L^2_t L^\infty_x L^2_v} + \| \nu^{\frac{1}{2} } (\mathbf{I} - \mathbf{P} )  f_R \|_{L^2_t L^\infty_x L^2_v} \right ) \| \p^s  f_R \|_{L^\infty_{t} L^2_{xv} } \right ] \\
& \lesssim \sqrt{\mathcal{D}(t) } \left [ \left (\sqrt{\int_0^t \mathcal{E} } + \e \sqrt{\kappa} \sqrt{\mathcal{D}(t) } \right ) \frac{1}{\sqrt{\kappa}} \| f_R \|_{L^\infty_{tx} L^2_v } \right. \\
&\left.+ \sqrt{\mathcal{E}(t) } \left ( \frac{1}{\sqrt{\kappa} } \|  f_R \|_{L^2_t L^\infty_x L^2_v}  + \e^{\frac{1}{2} } \left ( \sqrt{\mathcal{D}(t) } + \sqrt{\int_0 ^t \mathcal{E} } \right ) \right ) \right ] \lesssim \left ( \int_0 ^t \mathcal{E} + \mathcal{D}(t) \right ) \sqrt{\mathcal{E}(t) }, s \le 2,
\end{split} \\
\begin{split}\notag
&\int_{(0, t) \times \mathbb{T}^2 \times \mathbb{R}^3} \frac{1}{\e \kappa} \Gamma(\p f_R, \p f_R) \p^2 f_R \dd v \dd x \dd t' \\
&\lesssim \sqrt{\mathcal{D}(t) } \frac{1}{\sqrt{\kappa}}\left (  \| \p f_R \|_{L^2_t L^4_x L^2_v} + \| \nu^{\frac{1}{2} } (\mathbf{I} - \mathbf{P} ) \p f_R \|_{L^2_t L^4_x L^2_v } \right ) \| \p f_R \|_{L^\infty_t L^4_x L^2_v} \lesssim \left ( \int_0 ^t \mathcal{E} + \mathcal{D}(t) \right ) \sqrt{\mathcal{E}(t) }.
\end{split}
\end{align}
Here we have used Lemma \ref{lemma_L} to first bound terms with $L^2_v$ norms with mixed $L^p_x$ norms and then Lemma \ref{anint} to turn back to $L^2_x$ norm. In a similar manner, we have, for $s \le 2$, $s_1 + s_2 = s$, and $s_1 \ge 1$,
\Be
\begin{split}
&\int_{(0, t) \times \mathbb{T}^2 \times \mathbb{R}^3} \frac{1}{\e \kappa^{2 - \frac{s}{2} }} {}_{\p^{s_1}}\Gamma(\p^{s_2} f_R, f_R) \frac{\p^s f_R}{\kappa^{1 - \frac{s}{2}}} \dd v \dd x \dd t' \\
&\lesssim \| \p^{s_1} u^\beta \|_{L^\infty_{t,x} } \sqrt{\int_0 ^t \mathcal{E} } \frac{1}{\kappa^{2-\frac{s}{2} } } \left [ \left ( \| \p^{s_2} f_R \|_{L^2_{txv} } + \| \nu^{\frac{1}{2} } (\mathbf{I} - \mathbf{P} ) \p^{s_2} f_R \|_{L^2_{txv} } \right ) \| f_R \|_{L^\infty_{tx} L^2_v} \right. \\
&\left. + \left ( \|  f_R \|_{L^2_t L^\infty_x L^2_v } + \| \nu^{\frac{1}{2} } (\mathbf{I} - \mathbf{P} )  f_R \|_{L^2_t L^\infty_x L^2_v } \right ) \| \p^{s_2} f_R \|_{L^\infty_t  L^2_{xv}} \right ] \\
& \lesssim \| \p^{s_1} u^\beta \|_{L^\infty_{t,x} } \sqrt{\int_0^t \mathcal{E} } \frac{1}{\kappa^{\frac{1}{2} - \frac{s-s_2}{2}} } \left [ \left ( \sqrt{\int_0 ^t \mathcal{E} } + \e \sqrt{\kappa} \sqrt{\mathcal{D}(t) } \right ) \sqrt{\mathcal{E}(t) } \right. \\
&\left. + \left (\sqrt{\int_0 ^t \mathcal{E} } +\e^{\frac{1}{2}}\left (\sqrt{\mathcal{D}(t)} + \sqrt{\int_0 ^t \mathcal{E} } \right ) \right ) \sqrt{\mathcal{E}(t) } \right ] \\
& \lesssim (\| \nabla_x u^\beta \|_{L^\infty_{tx} } + V(\beta) \sqrt{\kappa} ) \sqrt{\mathcal{E}(t) } \left ( \int_0 ^ t \mathcal{E} + \e \mathcal{D}(t) \right ) \lesssim \sqrt{\mathcal{E} (t) } \left ( \| \nabla_x u^\beta \|_{L^\infty_{tx} } \int_0 ^t \mathcal{E} + \mathcal{D}(t) \right )
\end{split}\notag
\Ee
where the first factor $\| \nabla_x u^\beta \|_{L^\infty_{tx} }$ comes from the case $s_1 = 1$ and the second factor $V(\beta) \sqrt{\kappa} $ comes from the case case $s_2 = 2, s_2 = 0$. Also we have used \eqref{ekappabeta} to bury the contribution of $\|\nabla_x u^\beta \|_{L^\infty_{t,x} }$ in $\mathcal{D}(t) $. 

Therefore, we have
\Be \label{nonlinearestimate}
\eqref{nonlinearity} \lesssim \left ( (\| \nabla_x u^\beta \|_{L^\infty_{tx} }+1) \int_0 ^t \mathcal{E} + \mathcal{D}(t) \right )  \sqrt{\mathcal{E} (t) }
\Ee

Summing up \eqref{sourceestimate}, \eqref{linearestimate}, \eqref{Lcommutatorestimate}, \eqref{momestimate}, \eqref{nonlinearestimate}, we have
\Be \label{Energy}
\begin{split}
\mathcal{E}(t) + \mathcal{D}(t)  &\lesssim  \mathcal{E}(0)+ (\|\nabla_x u^\beta \|_{L^\infty_{t,x} }^2 + 1 + \sqrt{\mathcal{E}(t) } ) \int_0 ^t \mathcal{E} + \kappa  + \sqrt{\mathcal{E}(t) } \mathcal{D}(t)  \\
&+ e^{-\frac{1}{(\e \sqrt{\kappa})^{o(1) } } } \sum_{s \le 2} \| w_s \p^s f \|_{L^\infty ((0, t); L^\infty(\mathbb{T}^2 \times \mathbb{R}^3 ) ) } ^2.
\end{split}
\Ee

\subsubsection{$L^\infty$ control}

From Proposition \ref{prop_infty} and \eqref{Remainder} we obtain the following:
\Be \label{LinfinityRem}
\begin{split}
&\| w_0 f_R \|_{L^\infty ((0, t) ; L^\infty(\mathbb{T}^2 \times \mathbb{R}^3 ) )}
\\
 \lesssim& \  \| w_0 f_{R0} \|_{L^\infty(\mathbb{T}^2 \times \mathbb{R}^3 )} + \frac{1}{\e} \sqrt{\mathcal{E}(t)} + \e \| w_0 f_R \|_{L^\infty ((0, t) ; L^\infty(\mathbb{T}^2 \times \mathbb{R}^3 ) )}  ^2 \\
&+ \e \kappa V(\beta) \| w_0 f_R \|_{L^\infty ((0, t) ; L^\infty(\mathbb{T}^2 \times \mathbb{R}^3 ) )} + \e^3 \kappa V(\beta) + \e \kappa^2 V(\beta)
\end{split}
\Ee
Here we have used Lemma \ref{Gammabound} to bound the right-hand side of \eqref{Remainder}. Proceeding similar argument to \eqref{pRemainder}, for $1 \le s \le 2$, we obtain
\Be\notag
\begin{split}
\| w_s \p^s f_R \|&_{L^\infty ((0, t) ; L^\infty(\mathbb{T}^2 \times \mathbb{R}^3 ) )} \lesssim  \| w_s \p^s f_{R0} \|_{L^\infty(\mathbb{T}^2 \times \mathbb{R}^3 )} + \frac{1}{\e\kappa^{\frac{s}{2}}} \mathcal{E}(t)\\
& + \e \| w_s \p^s f_R \|_{L^\infty ((0, t) ; L^\infty(\mathbb{T}^2 \times \mathbb{R}^3 ) )} \| w_0 f_R \| _{L^\infty ((0, t) ; L^\infty(\mathbb{T}^2 \times \mathbb{R}^3 ) )} \\
& + \e \kappa V(\beta)  \| w_s \p^s f_R \|_{L^\infty ((0, t) ; L^\infty(\mathbb{T}^2 \times \mathbb{R}^3 ) )} \\
& + \e V(\beta) \sum_{s' < s} \| w_{s'} \p^{s'} f_R \|_{L^\infty ((0, t) ; L^\infty(\mathbb{T}^2 \times \mathbb{R}^3 ) )} \\
& + \e V(\beta) \sum_{s_1+s_2 \le s, s_1, s_2 < s} \| w_{s_1} \p^{s_1} f_R \|_{L^\infty ((0, t) ; L^\infty(\mathbb{T}^2 \times \mathbb{R}^3 ) )}\| w_{s_2} \p^{s_2} f_R \|_{L^\infty ((0, t) ; L^\infty(\mathbb{T}^2 \times \mathbb{R}^3 ) )} \\
&+ \e^3 \kappa V(\beta) + \e \kappa^2 V(\beta)
\end{split}
\Ee
Here we have used a pointwise bound $w_0 > \nu^2 w_1 > \nu^4 w_2$ for the third line. Therefore, we have
\Be \label{Linfestimate}
\begin{split}
\mathcal{F}(t) &\lesssim \mathcal{F}(0) + \e^2 + \sqrt{\mathcal{E}(t)} + \mathcal{F}(t)^2, \\
\| w_1 \p^1 f_R \|&_{L^\infty ((0, t) ; L^\infty(\mathbb{T}^2 \times \mathbb{R}^3 ) )} \lesssim  \| w_1 \p^1 f_{R0} \|_{L^\infty(\mathbb{T}^2 \times \mathbb{R}^3 )} + \mathcal{F}(t) \| w_1 \p^1 f_R \|_{L^\infty ((0, t) ; L^\infty(\mathbb{T}^2 \times \mathbb{R}^3 ) )}  \\
&+ \frac{1}{\e \sqrt{\kappa} } \sqrt{\mathcal{E}(t)} + \e V(\beta) \left ( 1 + \frac{\mathcal{F}(t)}{\e} \right )^2, \\
\| w_2 \p^2 f_R \|&_{L^\infty ((0, t) ; L^\infty(\mathbb{T}^2 \times \mathbb{R}^3 ) )} \lesssim  \| w_2 \p^2 f_{R0} \|_{L^\infty(\mathbb{T}^2 \times \mathbb{R}^3 )} + \mathcal{F}(t) \| w_2 \p^2 f_R \|_{L^\infty ((0, t) ; L^\infty(\mathbb{T}^2 \times \mathbb{R}^3 ) )}  \\
&+ \frac{1}{\e \kappa} \sqrt{\mathcal{E}(t)} + \e V(\beta) \left ( 1 + \frac{\mathcal{F}(t)}{\e} +\| w_1 \p^1 f_R \|_{L^\infty ((0, t) ; L^\infty(\mathbb{T}^2 \times \mathbb{R}^3 ) )}   \right )^2
\end{split}
\ee
In particular, giving explicit constants for \eqref{Energy} and \eqref{Linfestimate}, we obtain
\Be\notag
\begin{split}
\mathcal{E}(t) + \mathcal{D}(t) &\le \mathbf{C}_0 \left (\mathcal{E}(0) + \left (\|\nabla_x u^\beta\|_{L^\infty{((0, T) \times \mathbb{T}^2 }) } ^2 + 1 + \sqrt{\mathcal{E}(t) } \right ) \int_0 ^t \mathcal{E} + \kappa + \sqrt{\mathcal{E}(t)} \mathcal{D}(t) \right. \\
&\left. e^{-\frac{1}{(\e \sqrt{\kappa})^{o(1) } } } \sum_{s \le 2} \| w_s \p^s f \|_{L^\infty ((0, t); L^\infty(\mathbb{T}^2 \times \mathbb{R}^3 ) ) } ^2. \right ), \\
\mathcal{F}(t) &\le \mathbf{C}_0 \left ( \mathcal{F}(0) + \e^2 + \sqrt{\mathcal{E}(t)} + \mathcal{F}(t)^2 \right ), \\
\| w_s \p^s f_R \|&_{L^\infty ((0, t) ; L^\infty(\mathbb{T}^2 \times \mathbb{R}^3 ) )} \le \mathbf{C}_0 \big (\| w_s \p^s f_{R0} \|_{L^\infty(\mathbb{T}^2 \times \mathbb{R}^3 )} + \mathcal{F}(t) \| w_s \p^s f_R \|_{L^\infty ((0, t) ; L^\infty(\mathbb{T}^2 \times \mathbb{R}^3 ) )}   \\
&  + \e^{-1} \kappa^{-\frac{s}{2} } \sqrt{\mathcal{E}(t)} + \e V(\beta)  \big ( 1 + \e^{-1} \mathcal{F}(t) + \sum_{1 \le s' < s } \| w_{s'} \p^{s'} f_R \|_{L^\infty ((0, t) ; L^\infty(\mathbb{T}^2 \times \mathbb{R}^3 ) )}  \big)^2  \big )
\end{split}
\Ee
for some constant $\mathbf{C}_0 > 1$.

\subsection{Proof of Theorem \ref{theo:remainder}}

For given any arbitrary positive time $T>0$, choose $T_* \in [0,T]$ such that 
\Be
\begin{split}\label{assumption<T}
T_* =  \sup
\Big\{t>0:
\sqrt{\mathcal{E}(t) } < \frac{1}{10 \mathbf{C}_0}, \mathcal{F}(t) < \frac{1}{10 \mathbf{C}_0}
\Big\}.
\end{split}
\Ee

Then for $t \in [0, T_*]$, 
\Be\notag
\begin{split}
\mathcal{E}(t) + \mathcal{D}(t) &\le 2 \mathbf{C}_0 \mathcal{E}(0) + 2 \mathbf{C}_0 \left (\|\nabla_x u^\beta\|_{L^\infty{((0, T) \times \mathbb{T}^2 }) } ^2 + 2 \right ) \int_0 ^ t \mathcal{E} + 2 \mathbf{C}_0 \kappa \\
& + 2 \mathbf{C}_0 e^{-\frac{1}{(\e \sqrt{\kappa})^{o(1) } } } \sum_{s \le 2} \| w_s \p^s f \|_{L^\infty ((0, t); L^\infty(\mathbb{T}^2 \times \mathbb{R}^3 ) ) } ^2, \\
\mathcal{F}(t) &\le 2 \mathbf{C}_0 \mathcal{F}(0) + 2 \mathbf{C}_0 \e^2 + 2 \mathbf{C}_0 \sqrt{\mathcal{E}(t)},
\end{split}
\Ee
and for $1 \le s \le 2$,
\Be
\begin{split}\notag
\| w_s \p^s f_R \|&_{L^\infty ((0, t) ; L^\infty(\mathbb{T}^2 \times \mathbb{R}^3 ) )} \le 2 \mathbf{C}_0 \big (\| w_s \p^s f_{R0} \|_{L^\infty(\mathbb{T}^2 \times \mathbb{R}^3 )}  + \e^{-1} \kappa^{-\frac{s}{2} }  \\
& + \e V(\beta)  \big ( 1 + \e^{-1}/2 + \sum_{1 \le s' < s } \| w_{s'} \p^{s'} f_R \|_{L^\infty ((0, t) ; L^\infty(\mathbb{T}^2 \times \mathbb{R}^3 ) )}  \big)^2  \big ), \\
\| w_1 \p^1 f_R \|&_{L^\infty ((0, t) ; L^\infty(\mathbb{T}^2 \times \mathbb{R}^3 ) )} \le 2 \mathbf{C}_0  \| w_1 \p^1 f_{R0} \|_{L^\infty(\mathbb{T}^2 \times \mathbb{R}^3 )}  + 4 \mathbf{C}_0 \e^{-1} \kappa^{-\frac{1}{2} }, \\
\| w_2 \p^2 f_R \|&_{L^\infty ((0, t) ; L^\infty(\mathbb{T}^2 \times \mathbb{R}^3 ) )} \le 2\mathbf{C}_0  \| w_2 \p^2 f_{R0} \|_{L^\infty(\mathbb{T}^2 \times \mathbb{R}^3 )} \\
&+ C(\mathbf{C}_0) \left ( \| w_1 \p^1 f_{R0} \|_{L^\infty(\mathbb{T}^2 \times \mathbb{R}^3 )} ^2 + \e^{-1} \kappa^{-1} V(\beta) \right ).
\end{split}
\Ee

Since $\sqrt{\mathcal{E}(0)}, \mathcal{F}(0) < \delta_0 = \delta_0 (\e)$, and $\| w_s \p^s f_{R0} \|_{L^\infty(\mathbb{T}^2 \times \mathbb{R}^3 )} < \delta_s = \delta_s (\e)$ satisfying \eqref{deltas} for $s = 1,2$, we have
\Be
\begin{split}\notag
\| w_s \p^s f_R \|&_{L^\infty ((0, t) ; L^\infty(\mathbb{T}^2 \times \mathbb{R}^3 ) )} \le C(\mathbf{C}_0) (\e^{-1} \kappa^{-1/2} )^s, \\
&\mathcal{E}(t)  + \mathcal{D}(t) \le C( \mathbf{C}_0) (\delta_0^2+ \kappa)  + 2 \mathbf{C}_0 \left (\|\nabla_x u^\beta\|_{L^\infty{((0, T) \times \mathbb{T}^2 }) } ^2 + 2 \right ) \int_0 ^ t \mathcal{E}
\end{split}
\Ee
since $e^{-\frac{1}{(\e \sqrt{\kappa})^{o(1) } } }$ factor decays faster than any algebraic blowups. By Gronwall's lemma, we have
\Be
\begin{split}\notag
&\mathcal{E}(t) + \mathcal{D}(t) \le     C(\mathbf{C}_0)(\delta_0^2 + \kappa) (1+T) \\
& \times  \left ( 2 \mathbf{C}_0 \left (\|\nabla_x u^\beta\|_{L^\infty{((0, T) \times \mathbb{T}^2 }) } ^2 + 2 \right ) \exp \left ( 2 \mathbf{C}_0 \left (\|\nabla_x u^\beta\|_{L^\infty{((0, T) \times \mathbb{T}^2 }) } ^2 + 2 \right ) T \right ) \right ) .
\end{split}
\Ee
Since $\delta_0$ satisfies \eqref{delta0}, we see that for sufficiently small $\e$, $\sqrt{\mathcal{E}(T_*)}, \mathcal{F}(T_*)$ satisfies \eqref{assumption<T}. Therefore, $T_* = T$ and we proved the claim.

\section{Vorticity Convergence of the approximate solutions of Euler}

 \subsection{Stability of regular Lagrangian flow when the vorticity is unbounded}

To study the stability of the regular Lagrangian flow when the vorticities do not belong to $L^\infty$, we adopt the functional used in \cite{ALM,CDe,BC2013}: for $(u^{\beta_i},X^{\beta_i})$ solving \eqref{ODE:X_beta}, 
\Be\label{Lambda}
\Lambda (s;t) =\Lambda^{\beta_1, \beta_2} (s;t): = \int_{\T^2} \log \Big(1 + \frac{|X^{\beta_1}(s;t,x) - X^{\beta_2} (s;t,x)|}{\lambda}  \Big) \dd x,
\Ee
where we again abused the notation 
\Be \label{geodistance}
|X^{\beta_1}(s;t,x) - X^{\beta_2} (s;t,x)| = \mathrm{dist}_{\mathbb{T}^2} (X^{\beta_1}(s;t,x),X^{\beta_2} (s;t,x)), 
\Ee
that is, the geodesic distance between $X^{\beta_1}(s;t,x)$ and $X^{\beta_2} (s;t,x)$. We note that 
\Be\label{Lambda|t}
\Lambda ({t}\color{black};t)=0
\Ee
due to the last condition in both \eqref{dX} and \eqref{ODE:X_beta}. From \eqref{dX} and \eqref{ODE:X_beta},	a direct computation yields that 
\begin{align}
	|	\dot{\Lambda} (s;t) | &\leq \int_{\T^2} \frac{|\dot{X}^{\beta_1} (s)- \dot{X}^{\beta_2}(s)| }{ \lambda + |X^{\beta_1}(s) - X^{\beta_2}(s)|} \dd x \notag\\
	& \leq  \int_{\T^2} \frac{| u^{\beta_1} (s, X^{\beta_1}(s))-  u^{\beta_2}(s, X^{\beta_2} (s)| }{ \lambda + |X^{\beta_1}(s) - X^{\beta_2}(s)|} \dd x\notag\\
	& \leq \int_{\T^2} \frac{| u^{\beta_1} (s, X^{\beta_1}(s))-  u^{\beta_1} (s, X^{\beta_2} (s)| }{ \lambda + |X^{\beta_1}(s) - X^{\beta_2}(s)|} \dd x \label{dL_1}
	\\
	&+ \int_{\T^2} \frac{| u^{\beta_1} (s, X^{\beta_2}(s))-  u^{\beta_2}(s, X^{\beta_2} (s)| }{ \lambda + |X^{\beta_1}(s) - X^{\beta_2}(s)|} \dd x .\label{dL_2}
\end{align}

\begin{proposition}[\cite{CDe,BC2013}]\label{prop_stab}
Let $(u^{\beta_i}, \o^{\beta_i})$ satisfies \eqref{vorticity_beta}, \eqref{BS_beta}, \eqref{Lag_beta}, and $X^{\beta_i}$ be the regular Lagrangian flow of \eqref{ODE:X_beta} for $i=1,2$. Suppose $ \| u^{\beta_1} - u^{\beta_2} \|_{L^1((0,T) ; L^1(\T^2))} \ll 1$. Then 
	\Be\label{stab_rLf}
	\begin{split}
 &	\|  X^{\beta_1}(s;t,\cdot) - X^{\beta_2} (s;t,\cdot) \|_{L^1(\T^2)} \\
  &\lesssim  
	\frac{1+ 	\| \nabla u^{\beta_1}\|_{L^1((0,T) ; L^p(\T^2))} 
}{|\log  \| u ^{\beta_1}- u^{\beta_2} \|_{L^1((0,T) ; L^1(\T^2))}  |} \ \ \ \text{for} \ \ p >1.
\end{split}
	\Ee

For $p=1$, for every $\delta>0$ there exists $C_\delta>0$ such that for every $\gamma>0$
\Be
\begin{split}\label{stab_1}
	&	\mathscr{L}^2 (\{x \in \T^2: |X^{\beta_1}(s;t,x)- X^{\beta_2}(s;t,x)|>\gamma\}) 
	\\
	&\leq  
	\frac{e^{\frac{4C_\delta}{\delta}}}{\frac{4C_\delta}{\delta}} \frac{ \| u^{\beta_1} - u^{\beta_2}\|_{L^1 ((0,T); L^1(\T^2))}}{\gamma}
	+ \e
\end{split}
\Ee
holds.
\end{proposition}
For the convenience of the reader we provide a sketch of the argument. The argument follows the line of \cite{CDe} for $p>1$, and that of \cite{BC2013} for $p=1$.
\begin{proof}
	For \eqref{dL_2}, using \eqref{compression}, we have 
	\Be\label{est:dL_2}
	\begin{split}
		\eqref{dL_2} &\leq \frac{1}{\lambda} \int_{\T^2}  | u^{\beta_1} (s, X^{\beta_2}(s;t,x))-  u^{\beta_2}(s, X^{\beta_2} (s;t,x)|   \dd x \\
		& \leq \frac{\mathfrak{C}}{\lambda} \| u^{\beta_1}(s, \cdot) - u^{\beta_2}(s, \cdot) \|_{L^1(\T^2)}
	\end{split}
	\Ee
	with common compressibility bound $\mathfrak{C} = 1$. In the rest of the proof, we estimate \eqref{dL_1}. 
	
\textit{Step 1. The case of $p>1$. } Recall that the maximal function of $u$ is given by
\Be\label{Max_f}
M u(x) = \sup_{\e>0} \fint_{B_\e (x)}  |u(y) | \dd y=  \sup_{\e>0} \frac{1}{ \mathscr{L}^2 (B_\e (x)) } \int_{B_\e (x)}  |u(y) | \dd y.
\Ee
We have the following (e.g. \cite{H1996}, Section 2):
		\begin{align}
		|u(x) - u(y)| &\lesssim |x-y| \{(M\nabla u) (x) + (M \nabla u) (y)\} \ \  \ \text{a.e. } x,y \in \T^2,\label{DQ}\\
			\| Mw \|_{L^p(\T^2)} &\lesssim \| w\|_{L^p (\T^2)} \ \ \ \text{for} \ \ p \in (1,\infty].
		\label{Max_ineq}
		\end{align}
Now we bound \eqref{dL_1} for $p>1$, using \eqref{DQ} and \eqref{Max_ineq}, as
\Be\label{est:dL_1_p}
\begin{split}
\eqref{dL_1} &\leq \int_{\T^2 }\{  {M}  \nabla u^{\beta_1}    (s, X^{\beta_1} (s;t,x))
+ {M}  \nabla u^{\beta_1}    (s, X^{\beta_2}  (s;t,x))\}
 \dd x \\
 & \lesssim  \| \nabla u^{\beta_1} \|_{L^p(\T^2)} \ \ \text{for} \ \ p \in (1, \infty].
\end{split}\Ee

Using the above \eqref{est:dL_1_p}, \eqref{est:dL_2}, together with \eqref{Lambda|t}, we derive that 
\Be
\begin{split}
\label{est:Lambda}
\Lambda(s;t ) 
& \lesssim  \| \nabla u^{\beta_1} \|_{L^1 ((0,T); L^p(\T^2))}\\
& + \frac{1}{\lambda}
\| u^{\beta_1} - u^{\beta_2} \|_{L^1((0,T) ; L^1(\T^2))} \ \ \text{for all }  \  (s, t) \in [0,t] \times [0,T]. 
\end{split}
\Ee
On the other hand, for any $(s, t) \in [0,t] \times [0,T]$
\Be\label{lower_Lambda}
\begin{split}
	\mathbf{1}_{|X^{\beta_1}(s;t,x) - X^{\beta_2} (s;t,x)| \geq \gamma} 	\log \Big(1 + \frac{|X^{\beta_1}(s;t,x) - X^{\beta_2} (s;t,x)|}{\lambda}  \Big) \geq   \log \Big(1+\frac{ \gamma}{ {\lambda}}\Big)
	.  
\end{split}
\Ee
Then \eqref{lower_Lambda} with $\gamma = \sqrt{\lambda}$, together with the definition \eqref{Lambda}, implies that 
\Be
\begin{split}\label{L2_control}
	 \mathscr{L}^2 (\{ x \in \T^2: |X^{\beta_1}(s;t,x) - X^{\beta_2} (s;t,x)|  \geq \sqrt{\lambda} \}) 
  \leq \frac{1}{|\log \sqrt{\lambda}|} \Lambda (s;t)
\end{split}
\Ee
Therefore, by applying \eqref{est:Lambda} to \eqref{L2_control}, together with $\mathscr{L}^2(\T^2) = 1$ and $|x-y| \leq \sqrt{2}$ for $x,y \in \T^2$, we establish the stability:
\Be\notag
\begin{split}
	&	\|  X^{\beta_1}(s;t,\cdot) - X^{\beta_2}(s;t,\cdot) \|_{L^1(\T^2)} = \int_{\T^2} |X^{\beta_1}(s;t,x) - X^{\beta_2} (s;t,x) |\dd x\\
	& = \int_{|X^{\beta_1}(s;t,\cdot) - X^{\beta_2}(s;t,\cdot)| \leq \sqrt{\lambda} } +  \int_{|X^{\beta_1}(s;t,\cdot) - X^{\beta_2} (s;t,\cdot)| \geq \sqrt{\lambda} } \\
	&\leq \sqrt{\lambda} + \frac{\sqrt{2}}{|\log \sqrt{\lambda}|} \Lambda(s;t) \\
	& \lesssim \sqrt{\lambda} +  \frac{1}{|\log \sqrt{\lambda}|}
	\Big\{
	\| \nabla u^{\beta_1} \|_{L^1((0,T) ; L^p(\T^2))} + \frac{1}{\lambda} \| u ^{\beta_1}- u^{\beta_2} \|_{L^1((0,T) ; L^1(\T^2))}
	\Big\}.
\end{split}
\Ee
Choosing 
\Be\label{choice_lambda}
\lambda = \| u^{\beta_1} - u^{\beta_2} \|_{L^1((0,T) ; L^1(\T^2))},
\Ee
 we have that 
\Be
\begin{split}\label{stability_p}
&\|  X^{\beta_1}(s;t,\cdot) - X^{\beta_2}(s;t,\cdot) \|_{L^1(\T^2)} \\
&\lesssim 
\| u^{\beta_1} - u^{\beta_2} \|_{L^1((0,T) ; L^1(\T^2))}^{1/2}+ 
\frac{	\| \nabla u^{\beta_1} \|_{L^1((0,T) ; L^p(\T^2))}}{|\log  \| u^{\beta_1} - u^{\beta_2} \|_{L^1((0,T) ; L^1(\T^2))}  |}.
\end{split}
\Ee
For $\| u^{\beta_1} - u^{\beta_2} \|_{L^1((0,T) ; L^1(\T^2))}  \ll 1$, we prove \eqref{stab_rLf}.

\medskip

\textit{Step 2. The case of $p=1$. }Note that $p=1$ fails \eqref{Max_ineq}, but $\| M u  \|_{L^{1,\infty} (\T^2)} \lesssim \| u \|_{L^1 (\T^2)}$ only holds instead of \eqref{Max_ineq}. Here, we recall the quasi-norm of the Lorentz space $L^{p, q}$:
\Be\label{Lorentz}
\begin{split}
\| u \|_{L^{p,q}(\O, \mathrm{m})} &: = p^{1/q} \| \lambda  \mathscr{L}^2(\{x \in \T^2: |u(x)|>\lambda\})^{1/p} \|_{L^q (\R_+, \frac{\dd \lambda}{\lambda})},\\
\| u \|_{L^{p, \infty}(\T^2)}^p&= 
\| u \|_{L^{p, \infty}(\T^2, \mathscr{L}^2)}^p = \sup_{\lambda>0}
\{\lambda^p \mathscr{L}^2 (\{ x \in \T^2: |u(x)|> \lambda\})\} .
\end{split}\Ee

For $p=1$, there exists a map $\tilde{M}$, defined as in Definition 3.1 of \cite{BC2013} with choice of functions in Proposition 4.2 of \cite{BC2013} such that (Theorem 3.3 of \cite{BC2013}), 
\Be\label{o_to_M}
\tilde{M}: \o \mapsto \tilde M \nabla (\nabla^\perp (\Delta)^{-1} \o) \geq 0 
 \ \ \text{is bounded in $L^2  (\T^2)\rightarrow L^2(\T^2)$ and $L^1(\T^2)\rightarrow L^{1, \infty}(\T^2)$}. 
\Ee
Note that if $(u, \o)$ satisfying \eqref{BS} in the sense of distributions then $\tilde M \nabla (B*\o)=  \tilde M \nabla  u$. The argument follows the line of \cite{BC2013}, with translation to the periodic domain by Proposition \ref{PeriodicKernel}.

\begin{proposition}
	There exists an operator  $ \o \rightarrow U (\o)$, which will be denoted by $\tilde{M} \nabla u$, defined either on $L^1 (\mathbb{T}^2)$ or $L^2 (\mathbb{T}^2 )$, satisfying
	\Be\notag
	\begin{split}
		& U(\o) (x) \ge 0, \\
		&\| U(\o) \|_{L^{1, \infty} (\mathbb{T}^2 )} \lesssim \| \o \|_{L^1 (\mathbb{T}^2 ) }, \\
		& \| U(\o )\|_{L^2 (\mathbb{T}^2 ) } \lesssim \| \o \|_{L^2 (\mathbb{T}^2 ) }.
	\end{split}
	\Ee
	Also, if $\o \in L^1 (\mathbb{T}^2)$, and $u = B*\o$, then there is a Lebesgue measure 0 set $\mathcal{N}$ such that
	\Be\notag
	|u(x) - u(y) | \le |x-y| (U(x) + U(y) ), x, y \in \mathbb{T}^2 \setminus \mathcal{N}.
	\Ee
\end{proposition}
\begin{proof}
	We first identify $x \in \mathbb{T}^2$ with $x \in \left [ -\frac{1}{2}, \frac{1}{2} \right ]^2 \subset \mathbb{R}^2$, denote $K(y) := \nabla_y ^2 G(y) \chi_{\left [ -\frac{1}{2}, \frac{1}{2} \right ]^2} (y), y \in \mathbb{R}^2$, and define
	\Be\notag
	K_0 (y) = \frac{1}{4\pi} \frac{ y \otimes y - \frac{1}{2} |y|^2 \mathbb{I}_2 }{|y|^4}, y \in \mathbb{R}^2.
	\Ee
	Also, we regard $\o$ and $u$ as a $\mathbb{Z}^2$-periodic function in $\mathbb{R}^2$: $\o(x+m) = \o(x)$, $u(x+m) = u(x)$ for $m \in \mathbb{Z}^2$.  Now, for $x \in \left [ -\frac{5}{2}, \frac{5}{2} \right ]^2\subset \mathbb{R}^2$, $\int_{\mathbb{R}^2} K(y) \o (x-y) dy$ is well-defined as it is exactly $(\nabla^2 G *_{\mathbb{T}^2} \o) (x - m)$ for some $m\in \mathbb{Z}^2$ so that $x -m \in \left [ -\frac{1}{2}, \frac{1}{2} \right ]^2$. Then we see that $D(x)$ defined by
	\Be
	\begin{split}\notag
		D(x) &:= \int_{\mathbb{R}^2} K(y) \o (x-y) - K_0 (y) \o(x-y) \chi_{B_{100} (0) } (x-y) dy \\
		&= \int_{\mathbb{R}^2} (K(y) \chi_{\left [ -\frac{1}{2}, \frac{1}{2} \right ]^2} (y) - K_0 (y) \chi_{B_{100} (x) } (y) ) \o (x-y) dy
	\end{split}
	\Ee
	for $x \in \left [ -\frac{5}{2}, \frac{5}{2} \right ]$, and $D(x) := 0$ for $x \notin \left [ -\frac{5}{2}, \frac{5}{2} \right ]$ is bounded. First, since $ B_{\mathfrak{r} } (0) \subset B_{100} (x) \cap \left [ -\frac{1}{2}, \frac{1}{2} \right ]^2$ and thus $(K(y) \chi_{\left [ -\frac{1}{2}, \frac{1}{2} \right ]^2} (y) - K_0 (y) \chi_{B_{100} (x) } (y) )$ is bounded for $y \in B_{\mathfrak{r} } (0)$. For $y \notin B_{\mathfrak{r} }(0)$, $(K(y) \chi_{\left [ -\frac{1}{2}, \frac{1}{2} \right ]^2} (y) - K_0 (y) \chi_{B_{100} (x) } (y) )$ is bounded as well. Finally, $(K(y) \chi_{\left [ -\frac{1}{2}, \frac{1}{2} \right ]^2} (y) - K_0 (y) \chi_{B_{100} (x) } (y) )$ is supported on $B_{100} (x ) $, thus we have
	\Be\notag
	|D(x) | \lesssim \int_{\mathbb{R}^2} \chi_{B_{100} (x) } (y) |\o (x-y) | dy \lesssim C \| \o \|_{L^1 (\mathbb{T}^2)}.
	\Ee
	Furthermore, this implies $|D | \le C\| \o \|_{L^1 (\mathbb{T}^2)} \chi_{\left [ -\frac{5}{2}, \frac{5}{2} \right ]^2}$ so in fact $D \in L^1$ as well. Therefore, we have 
	\Be \label{nabxbound}
	\nabla_x u (x) = D(x) + K_0 \star_{\mathbb{R}^2} (\o \chi_{B_{100} (0) } ) (x), x \in \left [ -\frac{5}{2}, \frac{5}{2} \right ]^2.
	\Ee
	
	Next, we closely follow the argument of Proposition 4.2 of \cite{BC2013}. Let $\bar{h}$ be a smooth, nonnegative function, supported on $B_{\frac{1}{100} } (0)$ with $\int_{\mathbb{R}^2} \bar{h}(y) dy = 1$. Also, we denote $\bar{h}_r (x) = \frac{1}{r^2} \bar{h} \left (\frac{x}{r} \right )$ for $x \in \mathbb{R}^2$ and $r > 0$. Finally, for $\xi \in \mathbb{S}^1$ and $j=1,2$ we define
	\Be\notag
	\mathfrak{T}^{\xi, j} (w) := h (\frac{\xi}{2} - w) w_j,
	\Ee
	and $\mathfrak{T}^{\xi, j}_r$ is similarly defined for $r>0$. Now let $x, y \in \mathbb{T}^2 = \left [ -\frac{1}{2}, \frac{1}{2} \right ]^2$. Then there exists $\tilde{y} \in \left [ -\frac{3}{2}, \frac{3}{2} \right ]$, $\tilde{y} - y \in \mathbb{Z}^2$, such that the projection of line segment of $\tilde{y}$ and $x$ in $\mathbb{R}^2$ is the geodesic connecting $x, y$ in $\mathbb{T}^2$. Then we have
	\Be\notag
	\begin{split}
	&u(x) - u(y) = u(x) - u (\tilde{y} ) \\
	&= \int_{\mathbb{R}^2} \bar{h}_{|x-\tilde{y}|} \left ( z - \frac{x+\tilde{y} }{2} \right ) (u(x) - u(z) ) dz + \int_{\mathbb{R}^2} \bar{h}_{|x-\tilde{y} |} \left ( z - \frac{x+\tilde{y}}{2} \right ) (u(z) - u(y) )dz.
	\end{split}
	\Ee
	We focus on the first term: the other gives similar contribution. Following the argument of Proposition 4.2 of \cite{BC2013}, we have
	\Be\notag
	\int_{\mathbb{R}^2} \bar{h}_{|x-\tilde{y}|} \left ( z - \frac{x+\tilde{y} }{2} \right ) (u(x) - u(z) ) dz = |x-\tilde{y}| \sum_{j=1} ^2 \int_0 ^1 \int_{\mathbb{R}^2} \mathfrak{T}_{s|x-\tilde{y}|} ^{\frac{x-y}{|x-y|}, j } (w) (\partial_j u) (x-w) dw ds.
	\Ee
	Note that $\mathfrak{T}_{s|x-\tilde{y}|} ^{\frac{x-\tilde{y}}{|x-\tilde{y}|}, j }$ is supported on $B_{\frac{1}{100} s|x-\tilde{y} | } \left (\frac{x-y}{2 |x-\tilde{y}|} \right ),$ and $|x - \tilde{y} | \le \frac{\sqrt{2}}{2}$, so if $w \in B_{\frac{1}{100} s|x-\tilde{y} | } \left (\frac{x-\tilde{y}}{2 |x-\tilde{y}|} \right )$, $|w| \le \frac{2}{3}$ and thus $x - w \in \left [ -\frac{5}{2}, \frac{5}{2} \right ]^2$, which implies that \eqref{nabxbound} is satisfied at $x-w$. (Similar consideration shows that at $\tilde{y} - w$ \eqref{nabxbound} is satisfied.) Therefore,
	\Be\notag
	\begin{split}
		& \left | \int_{\mathbb{R}^2} \bar{h}_{|x-\tilde{y}|} \left ( z - \frac{x+\tilde{y} }{2} \right ) (u(x) - u(z) ) dz \right | \le |x - \tilde{y} |\sum_{j=1} ^2 \int_0 ^1 \left | \left ( \mathfrak{T}_{s|x-\tilde{y} | }^{\frac{x-y}{|x-y|}, j } \star_{\mathbb{R}^2} D \right ) (x) \right | ds \\
		& + |x - \tilde{y} | \sum_{j=1} ^2 \int_0 ^1 \left | \left ( \mathfrak{T}_{s|x-\tilde{y} | }^{\frac{x-y}{|x-y|}, j } \star_{\mathbb{R}^2}  \left ( K_0 \star_{\mathbb{R}^2} \o \chi_{B_{100} }  \right ) \right ) (x) \right | ds \\
		& \le |x - \tilde{y} | \sum_{j=1} ^2 \left ( M_{\left \{ \mathfrak{T}^{\xi, j} | \xi \in \mathbb{S}^1 \right \} }  (D) (x) + M_{\left \{ \mathfrak{T}^{\xi, j} | \xi \in \mathbb{S}^1 \right \} }  ( K_0 \star_{\mathbb{R}^2} \o \chi_{B_{100} } ) (x) \right ),
	\end{split}
	\Ee
	where
	\Be\notag
	M_{\left \{ \mathfrak{T}^{\xi, j} | \xi \in \mathbb{S}^1 \right \} }  ( g ) (x) = \sup_{\left \{ \mathfrak{T}^{\xi, j} | \xi \in \mathbb{S}^1 \right \} } \sup_{r>0} \left |  \left ( \mathfrak{T}^{\xi, j} _r \star_{\mathbb{R}^2} g \right ) (x) \right |, x \in \mathbb{R}^2.
	\Ee
	By Theorem 3.3 of \cite{BC2013}, we have 
	\Be\notag
	\| M_{\left \{ \mathfrak{T}^{\xi, j} | \xi \in \mathbb{S}^1 \right \} }  ( K_0 \star_{\mathbb{R}^2} \o \chi_{B_{100} } )  \|_{L^{1, \infty } (\mathbb{R}^2 ) } \le  C \| \o \chi_{B_{100} } \|_{L^1 (\mathbb{R}^2 ) } \le C \| \o \|_{L^1 (\mathbb{T}^2 ) }.
	\Ee
	Also, by Young's inequality, we have
	\Be\notag
	\| M_{\left \{ \mathfrak{T}^{\xi, j} | \xi \in \mathbb{S}^1 \right \} }  (D)  \|_{L^{1, \infty} (\mathbb{R}^2 ) } \le\| M_{\left \{ \mathfrak{T}^{\xi, j} | \xi \in \mathbb{S}^1 \right \} }  (D)  \|_{L^{\infty} (\mathbb{R}^2 ) }  \le  C \| D \|_{L^\infty (\mathbb{R}^2 ) } \le C \| \o \|_{L^1 (\mathbb{T}^2 )}.
	\Ee
	Finally, for $x \in \mathbb{T}^2$ identified with $\left [ \frac{-1}{2}, \frac{1}{2} \right ]^2$, we define 
	\Be\notag
	U(x) := \sum_{\tilde{x} \in \left [ -\frac{3}{2}, \frac{3}{2} \right ], x - \tilde{x} \in \mathbb{Z}^2 } \sum_{j=1} ^2 \left ( M_{\left \{ \mathfrak{T}^{\xi, j} | \xi \in \mathbb{S}^1 \right \} }  (D) (\tilde{x}) + M_{\left \{ \mathfrak{T}^{\xi, j} | \xi \in \mathbb{S}^1 \right \} }  ( K_0 \star_{\mathbb{R}^2} \o \chi_{B_{100} } ) (\tilde{x}) \right ).
	\Ee
	Then obviously for $x, y \in \mathbb{T}^2$
	\Be\notag
	|u(x) - u(y) | \le d_{\mathbb{T}^2} (x,y) (U(x) + U(y) ), 
	\Ee
	and if $U(x) > \lambda$, then for $\tilde{x}_1, \cdots, \tilde{x}_9 \in \left [ -\frac{3}{2}, \frac{3}{2} \right ]^2$ such that $\tilde{x}_j - x \in \mathbb{Z}^2$, at least one of $\tilde{x}_i$ satisfies $$\sum_{j=1} ^2 \left ( M_{\left \{ \mathfrak{T}^{\xi, j} | \xi \in \mathbb{S}^1 \right \} }  (D) (\tilde{x}_i) + M_{\left \{ \mathfrak{T}^{\xi, j} | \xi \in \mathbb{S}^1 \right \} }  ( K_0 \star_{\mathbb{R}^2} \o \chi_{B_{100} } ) (\tilde{x}_i) \right ) > \frac{\lambda}{9},$$
	and therefore
	\Be\notag
	\begin{split}
		&\left \{ x \in \left [ -\frac{1}{2}, \frac{1}{2} \right ]^2 | U(x) > \lambda \right \} \\
		& \subset \bigcup_{m = (a,b), a,b \in \{ -1, 0, 1 \} } \left \{ y \in \left [ -\frac{1}{2}, \frac{1}{2} \right ]^2 + m \big  | M_{\left \{ \mathfrak{T}^{\xi, j} | \xi \in \mathbb{S}^1 \right \} }  (D) (y) + M_{\left \{ \mathfrak{T}^{\xi, j} | \xi \in \mathbb{S}^1 \right \} }  ( K_0 \star_{\mathbb{R}^2} \o \chi_{B_{100} } ) (y) > \frac{\lambda}{9} \right \} \\
		& \subset \left \{ y \in \mathbb{R}^2 | M_{\left \{ \mathfrak{T}^{\xi, j} | \xi \in \mathbb{S}^1 \right \} }  (D) (y) + M_{\left \{ \mathfrak{T}^{\xi, j} | \xi \in \mathbb{S}^1 \right \} }  ( K_0 \star_{\mathbb{R}^2} \o \chi_{B_{100} } ) (y) > \frac{\lambda}{9} \right \} .
	\end{split}
	\Ee
	Therefore, we see that
	\Be\notag
	\| U \|_{L^{1, \infty} (\mathbb{T}^2 ) } \le C \| \o \|_{L^1 (\mathbb{T}^2 ) }
	\Ee
	
	Also, if $\o \in L^2 (\mathbb{T}^2)$, we see that
	\Be\notag
	\begin{split}
		\| U \|_{L^2 (\mathbb{T}^2 ) } & \le C ( \| M_{\left \{ \mathfrak{T}^{\xi, j} | \xi \in \mathbb{S}^1 \right \} }  ( K_0 \star_{\mathbb{R}^2} \o \chi_{B_{100} } )  \|_{L^{2 } (\mathbb{R}^2 ) } + \| M_{\left \{ \mathfrak{T}^{\xi, j} | \xi \in \mathbb{S}^1 \right \} }  (D)  \|_{L^{2} (\mathbb{R}^2 ) } \\
		&\le C (\| \o \chi_{B_{100} } \|_{L^2 (\mathbb{R}^2 ) } + \|D \|_{L^2 (\mathbb{R}^2 ) } \le C \| \o \|_{L^2 (\mathbb{T}^2 ) }
	\end{split}
	\Ee
	by again Theorem 3.3 of \cite{BC2013} and Young's inequality.
\end{proof}

We return to the proof of \eqref{stab_1}. We have (Proposition 4.2 in \cite{BC2013}) 
\Be\label{DQ1}
|u(x) - u(y)| \leq |x-y| \{\tilde{M} \nabla u(x)  + \tilde{M} \nabla u(y)  \}   \ \ \text{a.e.} \ x,y  \in \T^2.
\Ee

Now we check that $\{\o^\beta\}$ of \eqref{Lag_beta} with \eqref{vorticity_beta} is equi-integrable (in the sense of \eqref{o_beta:EI}). Fix any $\e>0$. We choose $\delta>0$ such that
\Be\label{EI_o}
\text{ if $\mathscr{L}^2 (E^\prime)< \delta$ then $ \int_{E^\prime} |\o_0(x)| \dd x< \frac{\e}{2 \mathfrak{C}}$.}
 \Ee 
 From \eqref{vorticity_beta} and \eqref{compression}, for any Borel set $E \subset \T^2$ with $\mathscr{L}^2 (E) < \delta/ \mathfrak{C}$,
\Be
\begin{split}\label{est:o1_1}
&\| \o^\beta (t,x) \|_{L^1 (E)}  =
\| \o_0^\beta (X^\beta (0;t,x)) \|_{L^1 (\{x \in E\})} \\
&\leq  \mathfrak{C}  \int_{X^\beta (t;0,x) \in E} |\o_0^\beta (x) | \dd x\\
& \leq  \mathfrak{C}   \int_{  \R^2}  \Big(\int_{  \T^2}\mathbf{1}_{X^\beta (t;0,x) \in E}
|\o_0(x-y)| \dd x\Big) \varphi^\beta(y)
 \dd y  ,
\end{split}
\Ee 
where $\o_0$ is regarded as a $\mathbb{Z}^2$-periodic function. For $y \in \R^2$, we define 
\Be\notag
\tilde{E}_y:= \{\tilde{x} \in \mathbb{R}^2: X^\beta (t;0,\tilde{x}+y) \in E + \mathbb{Z}^2\} / \mathbb{Z}^2 \subset \T^2. 
\Ee
From \eqref{compression} and the fact that $x \mapsto x-y$ is measure-preserving for fixed $y$, we have 
\Be\label{est:E}
\mathscr{L}^2 (\{\tilde{x} \in \tilde E_y\})= \mathscr{L}^2 (\{x \in \T^2: X^\beta (t;0,x) \in E\})
\leq \mathfrak{C} \mathscr{L}^2 (E)< \delta.
\Ee 
Therefore, applying \eqref{est:E} to \eqref{EI_o}, we have that, from \eqref{est:o1_1}, 
\Be\label{o_beta:EI}
\text{if } \  \mathscr{L}^2 (E )< \delta/\mathfrak{C} \ \text{ then } \ 
\| \o^\beta (t,\cdot) \|_{L^1(E)} 
\leq 
\| \varphi^\beta \|_{L^1 (\R^2)}
 \sup_{y \in \R^2}\mathfrak{C} \int_{\tilde{x} \in \tilde{E}_y} |\o_0 (\tilde{x})| \dd \tilde{x}
< \e. 
\Ee
Since $\o^\beta$ is equi-integrable, for every $\delta>0$ there exists $C_\delta>0$ and a Borel set $A_\delta \subset \T^2$ such that $\o^\beta = \o^\beta_1 + \o^\beta_2$ such that $\| \o_1^\beta \|_{L^1} \leq \delta$ and $\text{supp}(\o_2^\beta) \subset A_\delta$, $\| \o_2^\beta \|_{L^2} \leq C_\delta$ (Lemma 5.8 of \cite{BC2013}, whose proof can be established by noting that equi-integrability with $\sup_\beta \| \o^\beta\|_{L^1} < \infty$ is equivalent to $\lim_{K\rightarrow \infty} \sup_\beta \int_{\{ |\o^\beta| > K \} \cap \T^2}|\o^\beta| dx = 0$). Now apply \eqref{DQ1} to \eqref{dL_1} and use the decomposition of $u^\beta = u^\beta_1 + u^\beta_2$ with $u^\beta_i = \nabla^\perp (-\Delta)^{-1} \o^\beta_i$ to derive that 	\Be\begin{split}\label{def:U_l}
	\eqref{dL_1} \leq &
	\int_{\T^2}U^\lambda_1(s;t,x)  \dd x + 	\int_{\T^2}U^\lambda_2(s;t,x)  \dd x
	,\\
	U^\lambda_i(s;t,x) :=& \min 
	\Big\{ \frac{|u_i^{\beta_1}(s, X^{\beta_1}  (s;t,x))|}{\lambda} +  \frac{|u_i^{\beta_1}(s, X^{\beta_2}(s;t,x))|}{\lambda}
	,\\
	& \ \ \ \ \ \ \ \ 
	\tilde{M} \nabla u_i^{\beta_1} (s, X^{\beta_1}(s;t,x)) + 	\tilde{M} \nabla u_i^{\beta_1} (s, X^{\beta_2}(s;t,x)) 
	\Big\}\geq 0.
\end{split}	\Ee
\hide
\begin{align}
	U^1_\lambda(s;t,x) =& \min 
	\Big\{ \frac{|u(s, X  (s;t,x))|}{\lambda} +  \frac{|u (s, X^\beta (s;t,x))|}{\lambda}
	,\\
	& \ \ \ \ \ \ \ \ 
	\tilde{M} S \o_1 (s, X(s;t,x)) + 	\tilde{M} S \o_1 (s, X^\beta(s;t,x)) 
	\Big\}\\
	U^2_\lambda(s;t,x) =& \min 
	\Big\{ \frac{|u(s, X  (s;t,x))|}{\lambda} +  \frac{|u (s, X^\beta (s;t,x))|}{\lambda}
	,\\
	& \ \ \ \ \ \ \ \ 
	\tilde{M} S \o_2 (s, X(s;t,x)) + 	\tilde{M} S \o_2 (s, X^\beta(s;t,x)) 
	\Big\}
\end{align}
\unhide
For $U_2^\lambda$, we use \eqref{compression} and \eqref{o_to_M} and simply derive that 
\Be\label{estp:U}
\| U_2^\lambda (s;t, \cdot) \|_{L^2(\T^2)}\leq \mathfrak{C} 
\min\Big\{
\frac{
	2\| u^{\beta_1}_2 (s) \|_{L^2(\T^2)} }{\lambda} , \| \o^{\beta_1}_2 \|_{L^2}
\Big\} \leq \mathfrak{C} C_\delta 
\Ee
For $U^\lambda_1$, using \eqref{o_to_M}
\Be\notag
\| U_1^\lambda (s; t, \cdot) \|_{L^{1, \infty}} \lesssim 
\min 
\{ \frac{ \| u_1 ^\beta (s) \|_{L^{1, \infty}}}{\lambda },  \| \o_1 \|_{L^1 (\T^2)} \}  \leq 
\| \o_1 \|_{L^1 (\T^2)}
\leq \delta
\Ee
\Be\notag
\| U^1_\lambda (s; t, \cdot) \|_{L^{p, \infty}} \lesssim \| U^1_\lambda (s; t, \cdot) \|_{L^{p }}  \lesssim 
\min 
\{ \frac{ \| u_1 ^\beta (s) \|_{L^{p}}}{\lambda },  \| \o_1 \|_{L^p (\T^2)} \} \lesssim \frac{ \| u_1 ^\beta(s) \|_{L^{p}}}{\lambda } \lesssim \frac{\delta}{\lambda}.
\Ee
for some $p \in (1,2)$, using fractional integration.

Using the interpolation $
\| g\|_{L^1(\T^2)} \lesssim \| g \|_{L^{1, \infty}}
\Big\{1+ \log \Big( \frac{ \| g \|_{L^{p,\infty}}}{\|g \|_{L^{1, \infty}}}
\Big)
\Big\}$ (Lemma 2.2 of \cite{BC2013}), we end up with 
\Be\label{est1:U}
\| U^1_\lambda (s;t, \cdot) \|_{L^1}\lesssim 
\| U^1_\lambda (s, \cdot) \|_{L^{1, \infty} }
\Big\{1+ \log_+ \Big( \frac{ \|U^1_\lambda (s,\cdot ) \|_{L^{p,\infty}}}{\|U^1_\lambda (s,\cdot ) \|_{L^{1, \infty}}}
\Big)
\Big\}  \lesssim \delta + \delta |\log \lambda|.
\Ee
where we have used that the map $z \rightarrow z(1+\log_+(K/z))$ is nondecreasing for $z \in [0, \infty)$.

 Together with \eqref{est:dL_2}, \eqref{estp:U} and \eqref{est1:U}, we conclude that  
\Be\notag
\begin{split}\notag
&\Lambda(s; t) \le \int_s ^t |\dot{\Lambda} (\tau;t) | \dd \tau \le   \int_0^t\{\eqref{dL_1} + \eqref{dL_2}\} \dd s \\
&\leq \int_0^t\big\{
\| U_\lambda^1 (s;t,\cdot) \|_{L^1(\T^2) }  + \| U_\lambda^2 (s;t,\cdot) \|_{L^2(\T^2)} 
 + \frac{\mathfrak{C}}{\lambda}
  \| u^{\beta_1}  (s, \cdot )- u^{\beta_2} (s, \cdot )\|_{L^1(\T^2)}\big\}
 \dd s
\\
&
 \leq  
 \mathfrak{C}  C_\delta T  + \delta \{1+ |\log \lambda|\} T+
 \frac{\mathfrak{C}}{\lambda}\| u^{\beta_1} - u^{\beta_2} \|_{L^1((0,T);L^1(\T^2))}.
 \end{split}
\Ee
From this inequality, and \eqref{lower_Lambda}, \eqref{Lambda|t}, we derive that 
\Be
\begin{split}\label{X-X_1}
&	\mathscr{L}^2 (\{x \in \T^2: |X^{\beta_1}(s;t,x)- X^{\beta_2}(s;t,x)|>\gamma\}) 
\\
	&\le \frac{\Lambda(s; t)}{\log \left (1 + \frac{\gamma}{\lambda} \right )}\lesssim  \frac{\| u^{\beta_1} - u^{\beta_2}\|_{L^1L^1}}{ \lambda\log (1+ \frac{\gamma}{\lambda})|}  + \frac{C_\delta }{|\log (1+ \frac{\gamma}{\lambda})|}+ \delta
\end{split}
\Ee
for $\lambda, \gamma \in (0, 1/e)$. Here, for the last term, we have used that, for 
$0<\lambda< 1/e$ and $0<\gamma < 1/e$
\Be\notag
\frac{\delta |\log \lambda|}{|\log (1+ \frac{\gamma}{\lambda})|} =  \delta\frac{|\log \lambda|}{ - \log \lambda + \log (\lambda + \gamma)  } =\delta \frac{|\log \lambda|}{|\log \lambda| - |\log(\lambda + \gamma)| } \leq \delta \frac{|\log \lambda|}{|\log \lambda|  } \leq \delta .
\Ee
Choose
\Be\label{choice_lambda1}
\lambda= \lambda_{\delta,    \gamma} =(e^{\frac{ 4C_\delta}{\delta}}-1)^{-1} \gamma  .
\Ee
Note that $\log (1+ \frac{\gamma}{\lambda_{\e, \gamma}}) = 
\log (e^{\frac{ 4C_\delta}{\delta}})= \frac{4C_\delta}{\delta}$. Then \eqref{X-X_1} yields \eqref{stab_1}.
\end{proof}

\subsection{Convergence of Velocity field $u^\beta$}

\begin{lemma}\label{lemma_AL} Let $T>0$. Assume \eqref{BS_beta} holds and  
	\Be\begin{split}\notag
 \sup_{\beta} \| \o^\beta \|_{L^\infty ((0,T); L^1 (\T^2))}<\infty,\ \
   \sup_{\beta} \| u^\beta \|_{L^\infty ((0,T); L^2 (\T^2))}<\infty.
\end{split}	\Ee
Then there exists a subsequence 
	$\{\beta^\prime\} \subset \{\beta\}$ such that $u^{\beta^\prime}$ is Cauchy in $L^1((0,T); L^1(\T^2))$.
\end{lemma}

\begin{proof}
	The proof is due to the elliptic regularity, the Frechet-Kolmogorov theorem, which states that $W^{s,p} (\mathbb{T}^2) \xhookrightarrow{}  \xhookrightarrow{}  L^q(\mathbb{T}^2)$ for $s>0$ and $1 \leq q \leq p < \infty$, and the Aubin-Lions lemma, which states that for reflexive Banach spaces $X, Y, Z$ such that $Y  \xhookrightarrow{}  \xhookrightarrow{} X\xhookrightarrow{}  Z,$
	\Be
	\begin{split}\label{Aubin-Lions}
		{ W^{1,r} ((0,T); Z) \cap L^1((0,T); Y) \xhookrightarrow{}  \xhookrightarrow{}  L^1((0,T); X), \ \text{for}  \ r>1.
		} 
	\end{split}
	\Ee

	Note that, from $L^1(\O)  \xhookrightarrow{}  H^{s}(\O)$ for any $s<-1$,
	\Be\notag
	\o^\beta \in  C^0([0,T]; H^s(\O)) \  \text{ uniformly-in-$\beta$ for any $s<-1$.} 
	\Ee
	On the other hand, we have $- \Delta p^\beta = \text{div}(\text{div}(u^\beta \otimes u^\beta))$ with $\fint_{\O} p^\beta=0$. Since $u^\beta  \in L^\infty ( (0,T); L^2)$ uniformly-in-$\beta$, $u^\beta \otimes u^\beta  \in L^\infty((0,T);L^1(\O))$ uniformly-in-$\beta$. Using $L^1 (\O) \xhookrightarrow{}  H^{s}(\O)$ for $s<-1$, an elliptic regularity says $
	L^\infty( (0,T); H^{s-1}(\O))\ni
	\text{div}(	u^\beta  \otimes u^\beta )
	\mapsto \nabla p^\beta  \in  L^\infty( (0,T);H^{s-1}(\O))$ uniformly-in-$\beta$. Therefore from $\p_t u^\beta = - \text{div}(u^\beta \otimes u^\beta)- \nabla p^\beta$ we derive that $\p_t u^\beta \in L^\infty ((0,T);H^{s-1})$ uniformly-in-$\beta$ for any $s<-1$. Therefore we conclude that 
	\Be\label{AL1}
	u^\beta \in W^{1, \infty} ((0,T); H^{-5/2} (\O)) \ \text{uniformly-in-}\beta.
	\Ee

		Next, we note that $L^1(\mathbb{T}^2) \xhookrightarrow{}  W^{- \frac{3}{4}, 3} (\T^2)$. This is a consequence of an embedding $W^{ \frac{3}{4}, 3}(\mathbb{T}^2)\xhookrightarrow{}  L^\infty(\mathbb{T}^2) $ (note $\frac{3}{4}> \frac{2}{3}$) and the duality argument $L^1(\mathbb{T}^2) \xhookrightarrow{}  (L^\infty (\mathbb{T}^2))^* \xhookrightarrow{} (W^{ \frac{3}{4}, 3}(\mathbb{T}^2))^* = W^{- \frac{3}{4},\frac{3}{2}} (\mathbb{T}^2)$. Therefore we derive that $\o^\beta \in L^\infty ((0,T); W^{- \frac{3}{4},\frac{3}{2}} (\mathbb{T}^2) )$. Now applying the elliptic regularity theory to \eqref{BS_beta}, we drive that 
		\Be\label{AL2}
		u^\beta \in L^\infty ((0,T); W^{\frac{1}{4}, \frac{3}{2}} (\mathbb{T}^2) ) \  \text{ uniformly-in-$\beta$.} 
		\Ee
		
		Now we set $Y= W^{\frac{1}{4}, \frac{3}{2}} (\mathbb{T}^2), X=L^1(\mathbb{T}^2)  ,Z= H^{-\frac{5}{2}} (\mathbb{T}^2)$. Using the Frechet-Kolmogorov theorem, we have $Y= W^{\frac{1}{4}, \frac{3}{2}} (\mathbb{T}^2)  \xhookrightarrow{}  \xhookrightarrow{} X=L^1(\mathbb{T}^2)   \xhookrightarrow{}Z= H^{-\frac{5}{2}} (\mathbb{T}^2)$. Finally, we prove Lemma \ref{lemma_AL} using the Aubin-Lions lemma \eqref{Aubin-Lions}.\end{proof}
 
\hide

The Biot-Savart law implies that $(u^{\beta^\prime}, \o^{\beta^\prime})$ satisfying \eqref{vorticity_beta} should hold
\Be\notag
u^{\beta^\prime}(t,x) = \int_{\R^2} \mathbf{b}(x-y) \o^{\beta^\prime} (t,y) \dd y, \ \ \ |\mathbf{b}(x)|\leq \frac{C_{\mathbf{b}}}{|x|}. 
\Ee
We introduce smooth functions $\varphi^i_N{\color{red}:\R^2\rightarrow \R_+}$ with $\|\varphi^i_N \|_{C^2} \leq C_N <\infty $ for $i=1,2$, and $N\gg 1$, such that $\varphi^1_N(y) =1$ for $\frac{1}{N}\geq |y| $ and $\varphi^{\color{red}1}_N(y)=0$ for $\frac{2}{N}\leq |y|$; and $\varphi^2_N(y) =1$ for $|y|\geq N $ and $\varphi^{\color{red}2}_N(y)=0$ for $|y|\leq N/2$; and finally $\varphi_N^3:=(1- \varphi^1_N- \varphi^2_N).$ We decompose $\mathbf{b} = \varphi_N^1 \mathbf{b} +  \varphi_N^2\mathbf{b} +  \varphi^3_N  \mathbf{b}$. {\color{red}We estimate each contributions separately. We first note that, for $r \in [1,\infty]$ and any $\beta^\prime_1, \beta^\prime_2 \in \{\beta^\prime\}$,
	\Be	\begin{split}\label{u^beta:i}
		&\Big\|\int_{\R^2} \varphi_N^i( y) \mathbf{b}( y) \big(\o^{\beta^\prime_1}(t,x-y) - \o^{\beta^\prime_2}(t,x-y)\big) \dd y \Big\|_{L^r( \{x \in \O\})}\\
		&\leq\Big\|\int_{\R^2} \varphi_N^i( y) \mathbf{b}( y) \mathbf{1}_{ |(x-y)+ y| \leq \text{diam} (\O)}\big(\o^{\beta^\prime_1}(t,x-y) - \o^{\beta^\prime_2}(t,x-y)\big) \dd y \Big\|_{L^r( \{x \in \R^2\})}
		\\
		&\leq 2 \| \varphi^i_N  \mathbf{b}  \|_{L^r (\R^2)}   \| \o^{\beta^\prime_1}(t, \cdot ) - \o^{\beta^\prime_2}  (t, \cdot ) \|_{L^1(\O)}  ,
	\end{split}
	\Ee
	where we have used the Young's inequality and the fact that $\o$ is spatially periodic.
}

{\color{red}Using \eqref{u^beta:i} and the Fubini's theorem,}
\Be
\begin{split}\label{u^beta:1}
	&\Big\|\int_{\R^2} \varphi_N^1( y) \mathbf{b}( y) \big(\o^{\beta^\prime_1}(t,x-y) - \o^{\beta^\prime_2}(t,x-y)\big) \dd y \Big\|_{L^1( \{x \in \O\})}\\
	&\leq 2 \| \varphi^1_N  \mathbf{b}  \|_{L^1 (\R^2)}   \| \o^{\beta^\prime_1}(t, \cdot ) - \o^{\beta^\prime_2}  (t, \cdot ) \|_{L^1(\O)}  
	\leq \frac{10 C_ \mathbf{b}}{N} \sup_{\beta} \| \o ^\beta (t, \cdot ) \|_{L^1(\O)}.
\end{split}
\Ee
and, for $r >4$ from $\mathscr{L}^2(\O)=1$,
\Be
\begin{split}\label{u^beta:2}
	& \| \varphi_N^2 \mathbf{b} \star (\o^{\beta^\prime_1} (t)- \o^{\beta^\prime_2}(t)) \|_{L^1 (\{x \in \O\})}
	\\
	&\leq  \| \varphi_N^2 \mathbf{b} \star (\o^{\beta^\prime_1} (t)- \o^{\beta^\prime_2}(t)) \|_{L^r (\{x \in \O\})}
	\\
	&\leq \|  \varphi_N^2 \mathbf{b}   \|_{L^r (\R^2)} N^2 \big\|\o^{\beta^\prime_1}(t,\cdot  ) - \o^{\beta^\prime_2}(t,\cdot )\big\|_{L^1(\O)}  
	\\
	&\leq  \frac{C_{\mathbf{b}}}{N^{r-4}}   \sup_{\beta} \| \o ^\beta (t, \cdot ) \|_{L^1(\O)}.
\end{split}
\Ee 

{\color{red}Finally, the last term is bounded by}
\Be
\begin{split}\label{u^beta:3}
	& \| \varphi_N^3 \mathbf{b} \star (\o^{\beta^\prime_1} (t)- \o^{\beta^\prime_2}(t)) \|_{L^1 (\{x \in \O\})}
	\\
	&\leq\Big\|\int_{\R^2} \varphi_N^3( y) \mathbf{b}( y) {\color{red}\mathbf{1}_{|(x-y) - x| \leq \text{diam} (\O)}}
	\big(\o^{\beta^\prime_1}(t,x-y) - \o^{\beta^\prime_2}(t,x-y)\big) \dd y \Big\|_{L^\infty( \{x \in \O\})}\\ 
	& \leq   \| \varphi_N^3  \mathbf{b} \|_{H^{|s|}(\R^2)}  N^2 \|\o^{\beta^\prime_1}(t,\cdot) - \o^{\beta^\prime_2}(t,\cdot) \|_{H^{s}(\O)} .
\end{split}
\Ee  
For any $\e>0$, we choose large $N=N ( \sup_{\beta} \| \o ^\beta (t, \cdot ) \|_{L^1(\O)})\gg 1$ so that $\eqref{u^beta:1} + \eqref{u^beta:2}< \e/2$. For such $N$, now we choose $\beta^\prime\gg_{N,  \mathbf{b}} 1$ such that $ \eqref{u^beta:3}< \e/2$.\unhide

\hide
\begin{proof}
	The proof is due to the Biot-Savart law (Theorem 10.1 in \cite{BM}); and the Aubin-Lions lemma (Lemma 10.4 in \cite{BM}): $C^{0,1} ([0,T]; H^m) \cap C([0,T]; H^s) \xhookrightarrow{}  \xhookrightarrow{}  C([0,T]; H^r )$ for $m \leq r< s$. Note that, from $L^1(\T^2)  \xhookrightarrow{}  H^{s}(\T^2)$ for any $s<-1$,
	\Be
	\o^\beta \in  C^0([0,T]; H^s(\T^2)) \  \text{ uniformly-in-$\beta$ for any $s<-1$.} 
	\Ee
	On the other hand, we have $- \Delta p^\beta = \text{div}(\text{div}(u^\beta \otimes u^\beta))$ with $\fint_{\T^2} p^\beta=0$. Since $u^\beta  \in L^\infty ( (0,T); L^2)$ uniformly-in-$\beta$, $u^\beta \otimes u^\beta  \in L^\infty((0,T);L^1(\T^2))$ uniformly-in-$\beta$. Using $L^1 (\T^2) \xhookrightarrow{}  H^{s}(\T^2)$ for $s<-1$, an elliptic regularity says $
	L^\infty( (0,T); H^{s-1}(\T^2))\ni
	\text{div}(	u^\beta  \otimes u^\beta )
	\mapsto \nabla p^\beta  \in  L^\infty( (0,T);H^{s-1}(\T^2))$ uniformly-in-$\beta$. Therefore from $\p_t u^\beta = - \text{div}(u^\beta \otimes u^\beta)- \nabla p^\beta$ we derive that $\p_t u^\beta \in L^\infty ((0,T);H^{s-1})$ uniformly-in-$\beta$ for any $s<-1$. Therefore we conclude that $u^\beta \in C^{0,1} ([0,T]; H^{s-1} (\T^2))$ uniformly-in-$\beta$ .
	\Be\label{AL_o}
	\o^{\beta^\prime} \text{ is Cauchy in $C^0([0,T]; H^r)$ for any $r(<s)<-1$.}
	\Ee 
	\color{red}{Joonhyun's correction started}\color{black}{}
	
	Next, we show that $u^\beta$ is uniformly bounded in $L^\infty ((0, T); W^{\frac{1}{2}, \frac{5}{4} } (\mathbb{T}^2))$. We use the fractional Sobolev space: one can refer to, for example, \color{red}{Taylor, Partial Differential Equations, Volume 3, Section 13.6.}\color{black} We note that $L^\infty (\mathbb{T}^2) \subset W^{s, p}(\mathbb{T}^2)$ continuously, provided that $sp > 2$. We pick $s = \frac{1}{2}, p = 5$ for simplicity. Therefore, we have
	\Be
	L^1 (\mathbb{T}^2) \subset (L^\infty (\mathbb{T}^2))^* \subset W^{-\frac{1}{2}, \frac{5}{4} }(\mathbb{T}^2)
	\Ee
	and thus $\o^\beta$ is uniformly bounded in $L^\infty ((0, T); W^{-\frac{1}{2}, \frac{5}{4} }(\mathbb{T}^2))$. Therefore, we have $\nabla_x u^\beta \in L^\infty ((0, T); W^{-\frac{1}{2}, \frac{5}{4} } )$ uniformly as well, and by Poincare inequality, $u^\beta \in L^\infty ((0, T); W^{\frac{1}{2}, \frac{5}{4} } (\mathbb{T}^2))$ uniformly. Then, since
	\Be
	W^{\frac{1}{2}, \frac{5}{4} } (\mathbb{T}^2) \xhookrightarrow{} \xhookrightarrow{} L^1 (\mathbb{T}^2)
	\Ee
by Rellich-Kondrachov, by Aubin-Lions lemma we see that $\{u^\beta\}$ is precompact in $C([0, T]; L^1(\mathbb{T}^2))$.

	\color{red}{Joonhyun's correction ended}\color{black}
	
\hide	
	The Biot-Savart law implies that $(u^{\beta^\prime}, \o^{\beta^\prime})$ satisfying \eqref{vorticity_beta} should hold
	\Be\notag
	u^{\beta^\prime}(t,x) = \int_{\R^2} \mathbf{b}(x-y) \o^{\beta^\prime} (t,y) \dd y, \ \ \ |\mathbf{b}(x)|\leq \frac{C_{\mathbf{b}}}{|x|}. 
	\Ee
	We introduce a smooth functions $\varphi^i_N$ with $\|\varphi^i_N \|_{C^2} \leq C_N <\infty $ for $i=1,2$, such that $\varphi^1_N(y) =1$ for $\frac{1}{N}\geq |y| $ and $\varphi_N(y)=0$ for $\frac{2}{N}\leq |y|$; and $\varphi^2_N(y) =1$ for $|y|\geq N $ and $\varphi_N(y)=0$ for $|y|\leq N/2$; and finally $\varphi_N^3:=(1- \varphi^1_N- \varphi^2_N).$ We decompose $\mathbf{b} = \varphi_N^1 \mathbf{b} +  \varphi_N^2\mathbf{b} +  \varphi^3_N  \mathbf{b}$. Choose any $\beta^\prime_1, \beta^\prime_2 \in \{\beta^\prime\}$. We estimate each contributions separately: from the Funini's theorem 
	\Be
	\begin{split}\label{u^beta:1}
		&\Big\|\int_{\R^2} \varphi_N^1( y) \mathbf{b}( y) \big(\o^{\beta^\prime_1}(t,x-y) - \o^{\beta^\prime_2}(t,x-y)\big) \dd y \Big\|_{L^1( \{x \in \T^2\})}\\
		&\leq  \| \varphi^1_N  \mathbf{b}  \|_{L^1 (\R^2)}   \| \o^{\beta^\prime_1}(t, \cdot ) - \o^{\beta^\prime_2}  (t, \cdot ) \|_{L^1(\T^2)}  
		\leq \frac{10 C_ \mathbf{b}}{N} \sup_{\beta} \| \o ^\beta (t, \cdot ) \|_{L^1(\T^2)};
	\end{split}
	\Ee
	and, for $r >4$ from $\mathscr{L}^2(\T^2)=1$
	\Be
	\begin{split}\label{u^beta:2}
		& \| \varphi_N^2 \mathbf{b} * (\o^{\beta^\prime_1} (t)- \o^{\beta^\prime_2}(t)) \|_{L^1 (\{x \in \T^2\})}
		\leq  \| \varphi_N^2 \mathbf{b} * (\o^{\beta^\prime_1} (t)- \o^{\beta^\prime_2}(t)) \|_{L^r (\{x \in \T^2\})}
		\\
		&\leq \|  \varphi_N^2 \mathbf{b}   \|_{L^r (\R^2)} N^2 \big\|\o^{\beta^\prime_1}(t,\cdot  ) - \o^{\beta^\prime_2}(t,\cdot )\big\|_{L^1(\T^2)}  
		\\
		&\leq  \frac{C_{\mathbf{b}}}{N^{r-4}}   \sup_{\beta} \| \o ^\beta (t, \cdot ) \|_{L^1(\T^2)};
	\end{split}
	\Ee 
	and 
	\Be
	\begin{split}\label{u^beta:3}
		& \| \varphi_N^3 \mathbf{b} * (\o^{\beta^\prime_1} (t)- \o^{\beta^\prime_2}(t)) \|_{L^1 (\{x \in \T^2\})}
		\\
		&\leq\Big\|\int_{\R^2} \varphi_N^3( y) \mathbf{b}( y) \big(\o^{\beta^\prime_1}(t,x-y) - \o^{\beta^\prime_2}(t,x-y)\big) \dd y \Big\|_{L^\infty( \{x \in \T^2\})}\\ 
		& \leq   \| \varphi_N^3  \mathbf{b} \|_{H^{|s|}(\R^2)}  N^2 \|\o^{\beta^\prime_1}(t,\cdot) - \o^{\beta^\prime_2}(t,\cdot) \|_{H^{s}(\T^2)} .
	\end{split}
	\Ee  
	For any $\e>0$, we choose large $N=N ( \sup_{\beta} \| \o ^\beta (t, \cdot ) \|_{L^1(\T^2)})\gg 1$ so that $\eqref{u^beta:1} + \eqref{u^beta:2}< \e/2$. For such $N$, now we choose $\beta^\prime\gg_{N,  \mathbf{b}} 1$ such that $ \eqref{u^beta:3}< \e/2$.
	
\unhide	
	
\end{proof}\unhide

\subsection{Rate of Convergence of $u^\beta$: Localized Yudovich solutions}
We use the following version of the theorem, presented in \cite{CS2021}. The theorem in \cite{CS2021} provided the modulus of continuity for $u$ which we will use, and explicitly stated that the unique solution is regular Lagrangian.

We begin with introducing localized Yudovich class of vorticity. Intuitively, the localized Yudovich class consists of vorticities with moderate growth of $L^p$ norm as $p \rightarrow \infty$. Existence and uniqueness results of Yudovich class of vorticity extends to localized Yudovich class. We refer to \cite{CS2021} and references therein for further details. 
\Be\notag
\| \o \|_{Y_{\mathrm{ul} }^\Theta  (\O) } := \sup_{1 \le \pp < \infty} \frac{ \| \o \|_{L^\pp (\O ) } }{\Theta (\pp) }.
\Ee
In this paper, we focus on the growth function with the following condition, which gives quantitative bounds on the behavior of velocity field $u$; it would be interesting to see if one can generalize the presented results to arbitrary admissible growth functions. We assume that $\Theta: \mathbb{R}_{\ge 0} \rightarrow \mathbb{R}_{\ge 0}$ satisfies the following: there exists $m \in \mathbb{Z}_{>0}$ such that
\Be \label{Thetap}
\Theta (\pp) =  \prod_{k=1} ^m \log_k \pp ,
\Ee
for large $\pp >1$, where $\log_k \pp$ is defined inductively by $\log_1 \pp =  \log \pp$ and 
\Be\notag
\log_{k+1} \pp = \log \log_{k} \pp .
\Ee
Also, we adopt the convention that $\log_0 \pp = 1.$ We remark that we are only interested in the behavior of $\Theta$ for large $\pp$. Also, we denote the inverse function of $\log_m (\pp) $ (defined for large $\pp$) by $e_m$. Finally, we note that
\Be\notag
\int_{e_m(1) } ^\infty \frac{1}{\pp \Theta (\pp ) } = \infty,
\Ee
which turns out to be important in uniqueness of the solution.

\begin{theorem}[\cite{CS2021}] \label{thm:localYudovichwellposedness}
If $\o_0 \in Y_{\mathrm{ul}}^\Theta (\O)$, for every $T>0$ there exists a unique weak solution $\o \in L^\infty ([0, T]; Y_{\mathrm{ul}}^\Theta (\O))$ with $u \in L^\infty ([0, T]; C_b^{0, \varphi_{\Theta} }(\O, \mathbb{R}^2 ) )$, which is regular Lagrangian. Here, the function space $C_b^{0, \varphi_{\Theta} }(\O, \mathbb{R}^2 )$ is defined by
\Be\notag
C_b ^{0, \varphi_{\Theta} } (\O, \mathbb{R}^2 ) = \left \{ v \in L^\infty (\O, \mathbb{R}^2 )  | \sup_{x\ne y}\frac{|v(x) - v(y) | }{\varphi_\Theta (d(x,y)) < \infty} \right \},
\Ee
where $d(x,y)$ is the geodesic distance on the torus $\O = \mathbb{T}^2$, and $\varphi_\Theta$ is defined by
\Be\notag
\varphi_{\Theta} (r) = 
\begin{cases}
0, r=0, \\
r (1- \log r) \Theta( 1 - \log r), r \in (0, e^{-2} ), \\
e^{-2} 3 \Theta(3), r \ge e^{-2}.
\end{cases}
\Ee
Also, $\| \o\|_{L^\infty ([0, T]; Y_{\mathrm{ul}}^\Theta (\O))}$ and $\| u \|_{C_b ^{0, \varphi_\Theta} (\O, \mathbb{R}^2 ) }$ depends only on $\|\o_0 \|_{Y_{\mathrm{ul}}^\Theta (\O)}$ and $T$. The dependence is non-decreasing in both $\|\o_0 \|_{Y_{\mathrm{ul}}^\Theta (\O)}$ and $T$.
\end{theorem}

In this subsection, we prove the following proposition:
\begin{proposition}
Let $\omega_0 \in Y_{\mathrm{ul}}^{\Theta} (\T^2)$. There exist constants $M$, depending only on $m$ and $\sup_{t \in [0, T]} \|u(t) \|_{L^\infty}$ (and therefore $\|\o_0\|_{L^3}$) (and dimension $d=2$), and $C$($C=2e$ works), which is universal, such that
\Be \label{Rateftn}
\sup_{0 \le t \le T} \| u^\beta (t) - u (t) \|_{L^2 (\O) } ^2 \le \frac{M}{e_m \left ( \left ( \log_m \left (\frac{M}{\beta^2 \| \o_0 \|_{L^2 (\O ) }^2  } \right ) \right )^{e^{-C \| \o_0\|_{Y_{\mathrm{ul}}^\Theta} T } } \right ) } =: \mathrm{Rate}(\o_0 ; \beta).
\Ee
Note that $\lim_{\beta\rightarrow 0^+} \mathrm{Rate}(\o_0 ; \beta) = 0$.
\end{proposition}
In particular, the case $m=0$ corresponds to the Yudovich class, with $\mathrm{Rate}(\beta) = \beta^{2e^{-C \| \o_0\|_{Y_{\mathrm{ul}}^\Theta} T }} $. 

\begin{proof}
We follow the proof of \cite{Y1995}. By letting $v =  u^\beta - u$, we have
\Be\notag
\partial_t v + u^\beta \cdot \nabla_x v - v \cdot \nabla_x u + \nabla_x (p^\beta - p) = 0.
\Ee
Noting that $v$ is incompressible and taking $L^2$ norm of $v$, we obtain
\Be\notag
\frac{d}{2dt} \| v \|_{L^2 (\O) } ^2 \le \int_{\O} v \cdot \nabla_x u \cdot v \dd x,
\Ee
or
\Be\notag
\| v(t) \|_{L^2 (\O)}^2 \le \| v(0) \|_{L^2 (\O) } + 2 \int_0 ^t \int_{\O} |\nabla_x u| |v|^2 \dd x
\Ee
Next, we note that by Sobolev embedding
\Be\notag
\| v \|_{L^\infty (\O)}^2 \le  2(\| u \|_{L^\infty (\O)}^2 + \| u^\beta \|_{L^\infty (\O) }^2 ) \le 2 C \| \o_0 \|_{L^3 (\O)} ^2,
\Ee
while energy conservation gives 
\Be\notag
\| v (t) \|_{L^2 (\O)}^2 \le 2 (\| u^\beta (t) \|_{L^2  (\O)}^2 + \| u(t) \|_{L^2 (\O )} ^2 ) \le 4 \| u_0 \|_{L^2 (\O) }^2.
\Ee
Therefore, there exists a constant $M$, explicitly given by
\Be\notag
M:= 1+ 4 \| u_0 \|_{L^2 (\O)}^2  e_{m} (1) + 2C \| \o_0 \|_{L^3} ^2
\Ee
satisfies 
\Be\notag
\frac{M}{\| v(t) \|_{L^2 (\O)}^2} > e_{m} (1),  \|v (t) \|_{L^\infty (\O)}^2 \le M.
\Ee
Then, by the definition of $Y_{\mathrm{ul}}^\Theta$ and the Calderon-Zygmund inequality
\Be\notag
\| \nabla_x u \|_{L^\pp (\O)} \le C \pp \| \o \|_{L^\pp (\O)}
\Ee
for $p \in (1, \infty)$, we have
\Be\notag
\| \nabla_x u \|_{L^\pp (\O) } \le \|\o_0 \|_{Y_{\mathrm{ul}}^\Theta } \pp \Theta(\pp) := \|\o_0 \|_{Y_{\mathrm{ul}}^\Theta }\phi(\pp),
\Ee
where we have used the conservation of $\|\o\|_{L^\pp (\O) }$ for every $1 \le \pp < \infty$. We first treat the case of $m \ge 1$. By H\"{o}lder's inequality, for each $\epsilon \in (0, \frac{1}{e_{m-1} (1)})$ ($\frac{1}{e_{m-1} (1) } \le 1$) we have
\begin{align}
&\int_{\T^2} |\nabla_x u | |v|^2 dx \le \| v \|_{L^\infty (\O ) } ^{2\epsilon } \int |v|^{2 (1-\epsilon ) } |\nabla_x u | \dd x \notag \\
&\le M^{\epsilon} \left (\int_{\O} |v|^2 \dd x \right )^{1-\epsilon}\left ( \int_{\O} |\nabla_x u|^{\frac{1}{\epsilon } } \dd x \right )^\epsilon \notag \\
&\le M^\epsilon \left (\| v \|_{L^2 (\O ) }^2 \right )^{1-\epsilon} \| \o_0 \|_{Y_{\mathrm{ul} }^\Theta } \phi \left (\frac{1}{\epsilon}\right ) = \| \o_0 \|_{Y_{\mathrm{ul} }^\Theta } \| v \|_{L^2 (\O ) }^2 \left (\frac{M}{\|v \|_{L^2 (\O ) }^2 } \right )^\epsilon \phi \left ( \frac{1}{\epsilon } \right ).\notag
\end{align}
Now choose
\Be\notag
\epsilon^* = \frac{1}{\log \frac{M}{\|v(t) \|_{L^2 (\O ) }^2 }}.
\Ee
Then since $\frac{M}{\|v(t) \|_{L^2 (\O ) } ^2 } > e_{m} (1)$,  $\log \left (\frac{M}{\|v(t) \|_{L^2 (\O ) } ^2 } \right ) > \log (e_{m} (1) )  = e_{m-1} (1)$ so $\epsilon^* \in (0, \frac{1}{e_{m-1} (1) } )$. There, we have
\begin{align}
&\left (\frac{M}{\|v \|_{L^2 (\O ) }^2 } \right )^{\epsilon^*} \phi \left ( \frac{1}{{\epsilon^*} } \right )\notag \\
&= e \log \left (\frac{M}{\|v(t) \|_{L^2 (\O ) } ^2 } \right ) \log \left ( \log \left (\frac{M}{\|v(t) \|_{L^2 (\O ) } ^2 } \right )  \right ) \cdots \log_m \left ( \log \left (\frac{M}{\|v(t) \|_{L^2 (\O ) } ^2 } \right )  \right )\notag \\
&= e \Theta \left (\frac{M}{\|v(t) \|_{L^2 (\O ) } ^2 } \right ).\notag
\end{align}
For $m=0$ (Yudovich case), $\epsilon \rightarrow \left (\frac{M}{\|v \|_{L^2 (\O ) }^2 } \right )^\epsilon \phi \left ( \frac{1}{\epsilon } \right ) = \left (\frac{M}{\|v \|_{L^2 (\O ) }^2 } \right )^\epsilon \frac{1}{\epsilon}$ attains its minimum at $\epsilon ^* = \frac{1}{\log \left (\frac{M}{\|v \|_{L^2 (\O ) }^2 } \right )}$, so we choose $M$ such that $\epsilon^* < 1$.

Therefore, we have
\Be\notag
\int_\O |\nabla_x u| |v|^2 \dd x \le e \| \o_0 \|_{Y_{\mathrm{ul} }^\Theta } \| v \|_{L^2 (\O ) }^2  \Theta \left (\frac{M}{\|v \|_{L^2 (\O ) } ^2 } \right ).
\Ee
To sum up, we have
\Be\notag
\| v(t) \|_{L^2 (\O ) }^2 \le \| v_0 \|_{L^2 (\O ) }^2 + \int_0 ^t 2 e \| \o_0 \|_{Y_{\mathrm{ul } }^\Theta } \Psi ( \| v(s) \|_{L^2 (\O ) } ^2 ) \dd s,
\Ee
where 
\Be\notag
\Psi (r) = r \Theta \left (\frac{M}{r} \right ).
\Ee
Then by Osgood's lemma, we have
\Be\notag
-\mathcal{M} (\| v(t) \|_{L^2 (\O ) }^2 ) +\mathcal{M} (\| v_0 \|_{L^2 (\O ) }^2 ) \le 2e\| \o_0 \|_{Y_{\mathrm{ul} }^\Theta} t,
\Ee
where 
\begin{align}
&\mathcal{M} (x) = \int_x ^a \frac{\dd r}{\Psi (r) } = \int_x ^a \frac{\dd r}{r \prod_{k=1} ^m \log_k \left (\frac{M}{r} \right )} \notag\\
& = \int_{\frac{M}{a}} ^{\frac{M}{x} } \frac{\dd z}{z \prod_{k=1} ^m \log_k \left (z \right )} = \int_{\log_m (\frac{M}{a} ) } ^{\log_m (\frac{M}{x} ) } \frac{\dd y}{y} = \log_{m+1} \left (\frac{M}{x} \right ) - \log_{m+1} \left (\frac{M}{a} \right ).\notag
\end{align}
where $a = 2 \| u_0 \|_{L^2(\O)}^2$ and we have used the substitution $z = \frac{M}{r}$ for the third identity and $y = \log_m (z)$ with 
\Be\notag
\frac{\dd y}{\dd z} = \frac{1}{z \prod_{k=1} ^{m-1} \log_k (z) }
\Ee
for the forth identity. In particular, we have
\Be\notag
\begin{split}
&\log_{m+1} \left (\frac{M}{\|v(t) \|_{L^2 (\O )}^2 } \right )\\
& \ge \log_{m+1} \left (\frac{M}{\|v_0 \|_{L^2 (\O )}^2 } \right ) - C \| \o_0 \|_{Y_{\mathrm{ul } } ^\Theta } t  = \log \left ( \log_m\left (\frac{M}{\|v_0 \|_{L^2 (\O )}^2 } \right ) e^{-C t \| \o_0 \|_{Y_{\mathrm{ul } }^\Theta} } \right ),
\end{split}
\Ee
and taking $e_{m+1}$ and reciprocal gives the desired conclusion. Certainly $\mathrm{Rate}(\beta)$ is a continuous function of $\beta$, and it converges to 0 as $\beta \rightarrow 0$ as $\mathcal{M}(0)=\infty$.
\end{proof}

\subsection{Convergence  of $\o^\beta$}



\begin{proposition}\label{theo:convergence}For any fixed $\pp \in [1,\infty]$, suppose $\o_0 \in L^\pp (\T^2)$. Recall the regularization of the initial data $\o_0^\beta$ in \eqref{vorticity_beta}. Let $(u^\beta, \o^\beta)$ and $(u, \omega)$ be Lagrangian solutions of \eqref{vorticity_eqtn}-\eqref{BS_beta} and \eqref{vorticity}-\eqref{BS}, respectively. For any $T>0$ and the subsequence $\{\beta^\prime\} \subset \{\beta\}$ in Lemma \ref{lemma_AL}, we have 
	\Be\label{stab_o_beta}
	\sup_{t \in [0,T]}\|	\o^{\beta^\prime} (t, \cdot)- \o  (t, \cdot) \| _{L^\pp (\T^2)} \rightarrow 0 \ \ \text{as} \ \  \beta^\prime \rightarrow \infty.
	\Ee
\end{proposition}

\begin{proof}
	For the subsequence $\{\beta^\prime\} \subset \{\beta\}$ in Lemma \ref{lemma_AL},
	\begin{align}
		&|\o (t,x) - \o^{\beta^\prime} (t,x)|\notag\\
		& = |\o_0(X(0;t,x)) - \o^{\beta^\prime} _0 (X^{\beta^\prime} (0;t,x))|\notag\\
		& \leq |\o_0(X(0;t,x)) - \o^\ell_0 (X(0;t,x))| 
		+ |\o_0^\ell (X^{\beta^\prime}  (0;t,x)) - \o_0^{\beta^\prime}  (X^{\beta^\prime}  (0;t,x))|
		\label{diff_o_1}
		\\
		&+ |\o^\ell_0(X(0;t,x)) - \o^\ell_0 (X^{\beta^\prime}  (0;t,x))|. \label{diff_o_2}
	\end{align}
	Using the compressibility \eqref{compression}, we derive that, for $\pp \in [1,\infty]$ 
	\Be\label{est:diff_o_1}
	\| \eqref{diff_o_1}\|_{L^p} \leq 2 \mathfrak{C} \| \o_0 - \o_0^\ell \|_{L^p}.
	\Ee 
	
	For the last term, we need a stability of the Lagrangian flows:
	\Be
	\begin{split}\label{est:diff_o_2}
		\|\eqref{diff_o_2}\|_{L^{\pp} (\T^2)} &\leq \| \nabla \o_0^\ell \|_{L^\infty} \|X(0;t,\cdot) - X^{\beta^\prime} (0;t,\cdot)\|_{L^\pp (\T^2)}\\
		& \leq   \| \nabla  \varphi^\ell \|_{L^\infty} \| \o_0 \|_{L^1} \|X(0;t,\cdot) - X^{\beta^\prime}  (0;t,\cdot)\|_{L^\pp (\T^2)}\\
		& \leq \frac{1}{ \ell^3}  \| \nabla  \varphi \|_{L^\infty (\T^2)}  \| \o_0 \|_{L^1} \|X(0;t,\cdot) - X^{\beta^\prime}  (0;t,\cdot)\|_{L^\pp (\T^2)},
	\end{split}
	\Ee
	where we have used \eqref{growth:Do_ell}.

	For $\pp >1$, we use \eqref{stab_rLf} in Proposition \ref{prop_stab} and Lemma \ref{lemma_AL} to have 
	\Be
	\begin{split}\label{est1:diff_o_2}
		\eqref{est:diff_o_2} & \lesssim  \frac{1}{\ell^3} 	\frac{1+ 	\| \nabla u^{\beta^\prime}\|_{L^1((0,T) ; L^\pp(\T^2))} 
		}{|\log  \| u  - u^{\beta ^\prime} \|_{L^1((0,T) ; L^1(\T^2))}  |} 
	\end{split}	\Ee
	Now we choose 
	\Be\label{choice_ell}
	\ell= \ell(\beta^\prime) \sim |\log  \| u  - u^{\beta ^\prime} \|_{L^1((0,T) ; L^1(\T^2))}  |^{-\frac{1}{10}}
	\ \ \text{for each } \beta^\prime, \Ee
	such that 
	\Be\notag
	\begin{split}
		\ell=	\ell(\beta^\prime) \downarrow 0 \ \ \text{as} \ \beta^\prime \downarrow 0,\\
		\ell^3 |\log  \| u  - u^{\beta ^\prime} \|_{L^1((0,T) ; L^1(\T^2))}  | \rightarrow  \infty \ \ \text{as} \ \beta^\prime \downarrow 0.
	\end{split}	\Ee
	Therefore, for $\pp>1$, we prove $\eqref{est1:diff_o_2} \rightarrow 0$ as $\beta^\prime \downarrow 0$. Combining this with \eqref{est:diff_o_1}, we conclude \eqref{stab_o_beta} for $\pp >1$.

	For $p=1$, 
	there exists $C_\e>0$ for any $\e>0$ such that 
	\Be
	\begin{split}\label{stab_1}
		&	\mathscr{L}^2 (\{x \in \T^2: |X^{\beta_1}(s;t,x)- X^{\beta_2}(s;t,x)|>\gamma\}) 
		\\
		&\leq  
		\frac{e^{\frac{4C_\e}{\e}}}{\frac{4C_\e}{\e}} \frac{ \| u^{\beta_1} - u^{\beta_2}\|_{L^1 ((0,T); L^1(\T^2))}}{\gamma}
		+ \e \ \ \ \text{for any} \ \ \gamma >0.
	\end{split}
	\Ee

	For $\pp =1$, using \eqref{stab_1}, we have
	\Be
	\begin{split}\notag
		&\|X(0;t,\cdot) - X^{\beta^\prime}(0;t,\cdot)\|_{L^1 (\T^2)}\\
		&
		\leq  \int_{|X(0;t,\cdot) - X^{\beta^\prime} (0;t,\cdot)| \leq \gamma} 
		|X(0;t,x) - X^{\beta^\prime} (0;t,x)| \dd x\\
		&
		\  \ +  \int_{|X(0;t,\cdot) - X^{\beta^\prime}(0;t,\cdot)| \geq \gamma} |X(0;t,x) - X^{\beta^\prime} (0;t,x)| \dd x\\
		& \leq \gamma + 
		\frac{e^{\frac{4C_\e}{\e}}}{\frac{4C_\e}{\e}} \frac{ \| u - u^{\beta^\prime}
			\|_{L^1 ((0,T); L^1(\T^2))}}{\gamma} + \e ,
	\end{split}
	\Ee
	and hence
	\Be\label{est:diff_X}
	\eqref{est:diff_o_2}\lesssim \frac{1}{\ell^3}\Big\{
	\gamma + 
	\frac{e^{\frac{4C_\e}{\e}}}{\frac{4C_\e}{\e}} \frac{ \| u - u^{\beta^\prime}
		\|_{L^1 ((0,T); L^1(\T^2))}}{\gamma} + \e
	\Big\}.
	\Ee
	For each $\e>0$, we choose $\gamma=\e$ and $\ell= \e^{\frac{1}{10}}$. And $\beta^\prime\gg_\e1$ such that $\frac{e^{\frac{4C_\e}{\e}}}{\frac{4C_\e}{\e}} \frac{1}{\e^{\frac{13}{10}}} \| u - u^{\beta^\prime}
	\|_{L^1 ((0,T); L^1(\T^2))} \rightarrow 0$. Combining with \eqref{est:diff_o_1}, we conclude \eqref{stab_o_beta} for $\pp=1$.\end{proof}


\subsubsection{When $\o_0$ has no regularity}

If $\o_0 \in Y_{\mathrm{ul}} ^\Theta (\O)$ and no additional regularity is assumed, one cannot expect a convergence rate which is uniform over $\o_0$: the rate crucially depends on how fast $\o_0 ^\beta$ converges to $\o_0$. Suppose that $\o(t)$ is the Lagrangian solution with initial data $\o_0$. Then we have
\begin{align}
&|\o (t,x) - \o^\beta (t,x) | = | \o_0 (X(0; t, x) ) - \o_0 ^\beta (X^\beta (0; t, x)) | \notag\\
&\le |\o_0 (X(0; t, x)) - \o_0 ^\ell (X(0; t, x))| + |\o_0 ^\ell (X^\beta (0; t, x)) - \o_0 ^\beta (X^\beta (0; t, x))| \notag \\
& + |\o_0 ^\ell (X(0; t, x)) - \o_0 ^\ell (X^\beta (0; t, x))|,\notag
\end{align}
where $\o_0 ^\ell$ is the initial data regularization of $\o_0$ with parameter $\ell$. Therefore, by the compression property we have
\begin{align}
&\| \o(t) - \o^\beta (t) \|_{L^\pp (\O ) } \le \mathfrak{C} \| \o_0 - \o_0 ^\ell \|_{L^\pp (\O) } + \| \o_0 ^\ell - \o_0 ^\beta \|_{L^\pp (\O ) } \notag\\
&+ \| \o_0 ^\ell (X(0, t; \cdot ) - \o_0 ^\ell (X^\beta (0; t, \cdot ) ) \|_{L^\pp (\O)}.\notag
\end{align}
Using \eqref{est:diff_o_2}, we can estimate the first two terms:
\Be
\mathfrak{C} \| \o_0 - \o_0 ^\ell \|_{L^\pp (\O) } +( \| \o_0 ^\ell - \o_0 ^\beta \|_{L^\pp (\O ) } \le (\mathfrak{C} +1 ) \| \o_0 - \o_0 ^\ell \|_{L^\pp (\O) } + \| \o_0 ^\beta - \o_0 \|_{L^\pp (\O ) }.\notag
\Ee
The last term is estimated by \eqref{est:diff_o_2} and \eqref{est1:diff_o_2}:
\Be
\| \o_0 ^\ell (X(0, t; \cdot )) - \o_0 ^\ell (X^\beta (0; t, \cdot ) ) \|_{L^\pp (\O)} \le \frac{C (1 + \pp \| \o_0 \|_{L^\pp (\O) } t )}{ \ell^3 | \log \mathrm{Rate}(\o_0; \beta) | }.\notag
\Ee
Choosing $\ell = |\log \mathrm{Rate} (\beta) |^{-\frac{1}{4}}$ gives that for $t \in [0, T]$
\Be \label{Rateftnvort}
\begin{split}
\| \o(t) - \o^\beta (t) \|_{L^\pp (\O) } &\lesssim \| \o_0 ^\beta - \o_0 \|_{L^\pp (\O)} + \| \o_0 - \o_0 ^{|\log \mathrm{Rate} (\o_0; \beta) |^{-\frac{1}{4}} } \|_{L^\pp (\O) } + \frac{1+ \pp \| \o_0 \|_{L^\pp (\O ) } T}{|\log \mathrm{Rate} (\o_0; \beta ) |^{\frac{1}{4} } } \\
& =: \mathrm{Rate}_\o (\o_0; \beta).
\end{split}
\Ee
Since there is no explicit rate for the convergence of $\| \o_0 ^\beta - \o_0 \|_{L^\pp (\O ) }$, the first two terms dominate the rate of convergence in general. 

\subsubsection{When $\o_0$ has some regularity}

An important class of localized Yudovich vorticity functions belong to Besov space of positive regularity index: for example, $f(x) = \log( \log |x| ) \varphi(x) \in Y_{\mathrm{ul} }^\Theta$ with $\Theta (\pp) = \log \pp$, where $\varphi(x)$ is a smooth cutoff function, belongs to $W^{1, r} (\O)$ where $r < 2$, and thus in Besov space $B^s_{2, \infty}$ with $s < 1$. Of course, vortex patches $\chi_D$ with box-counting dimension of the boundary $d_F (\partial D)<2$ belongs to $B^{\frac{2-d_F (\partial D)}{p}}_{p, \infty}$ for $1 \le p <\infty$ (\cite {CW1996}) and thus vortex patch with a mild singularity in the interior of $D$ also belongs to a certain Besov space with positive regularity.

In this subsection, we provide the rate of convergence of vorticity when $\o_0 \in Y_{\mathrm{ul}} ^\Theta (\O) \cap B^s_{2, \infty} (\O)$ or $\o_0 \in L^\infty (\O) \cap B^s_{2, \infty} (\O)$. Unlike Yudovich $\o_0 \in L^\infty (\O)$ case, if $\o_0$ is in localized Yudovich class $Y_{\mathrm{ul} }^\Theta (\O)$, even if initial vorticity has additional Besov regularity, that is, $\o_0 \in Y_{\mathrm{ul}}^\Theta (\O) \cap B^s_{2, \infty} (\O)$ for some $s>0$, the Besov regularity of vorticity $\o(t)$ may not propagate, even in the losing manner. The key obstruction is failure of generalization of propagation of regularity result. We will explain this after proving the result, following the argument of \cite{CDE2019}, \cite{BCD2011}, and \cite{V1999}. 

\begin{proposition}
If $\o_0 \in Y_{\mathrm{ul}}^\Theta (\O) \cap B^s _{2, \infty } (\O)$ for some $s>0$, then we have
\Be \label{vorticityratelocYud}
\begin{split}
&\| \o_0 - \o_0 ^\beta \|_{L^2 (\O ) } \\
&\le C(T, \| \o_0 \|_{L^2 (\O) }, \| \o_0 \|_{B^s_{2, \infty } (\O ) } ) \left (\beta^{\frac{s'}{1+s'}} +  \left ( \frac{1}{|\log \mathrm{Rate} (\o_0; \beta ) | } \right )^{\frac{s'}{3+4s'} } \right ) \\
& =: \mathrm{Rate}_{\o, s, loc-Y} (\beta)
\end{split}
\Ee
for any $s' \in (0, s)$.
Moreover, if $\o_0 \in L^\infty (\O) \cap B^s _{2, \infty } (\O)$, 
\Be \label{vorticityrateYud}
\| \o^\beta (t) - \o(t) \|_{L^2 (\O) } \le C(s, T, \|\o_0\|_{B^s_{2, \infty} (\O ) } ) \beta^{C(s) e^{-C(\|\o_0 \|_{L^\infty (\O ) }) T } } =: \mathrm{Rate}_{\o, s, Y} (\beta).
\Ee
In particular, if $\o_0$ is Yudovich with some Besov regularity, the vorticity converges with an algebraic rate $\beta^\alpha$.
\end{proposition}

\begin{proof}
First, we prove the rate for $\o_0 \in Y_{\mathrm{ul}}^\Theta (\O) \cap B^s _{2, \infty } (\O)$. We rely on the above rate:
\Be\notag
\| \o(t) - \o^\beta (t) \|_{L^2 (\O) }  \le C ( \| \o_0 - \o_0 ^\ell \|_{L^2 (\O ) } + \| \o_0 - \o_0 ^\beta \|_{L^2 (\O ) } ) + \frac{C(1+T \| \o_0 \|_{L^2 (\O ) } )}{ \ell^3 | \log \mathrm{Rate}(\beta) | }.
\Ee
Since $\o_0 \in B^s _{2, \infty} (\O)$, we may use the following interpolation: 
\begin{align}
&\| \o_0 - \o_0 ^\beta \|_{L^2 (\O)}\notag \\
&\le \| \o_0 - \o_0 ^\beta \|_{H^{-1} (\O)} ^{\frac{s'}{1+s'} }\| \o_0 - \o_0 ^\beta \|_{H^{s'} (\O)} ^{\frac{1}{1+s'} } \le \| \o_0 - \o_0 ^\beta \|_{H^{-1} (\O)} ^{\frac{s'}{1+s'} }\| \o_0 - \o_0 ^\beta \|_{B^s_{2, \infty} (\O) } ^{\frac{1}{1+s'} },\notag
\end{align}
for arbitrary $s' \in (0, s)$, where we have used that $H^s = B^s_{2,2}$ and $B^s_{p,q} (\O) \subset B^{s'}_{p, q'} (\O)$ for $s'<s$ and arbitrary $q, q'$. (The proof for whole space, which is standard, can be easily translated to periodic domain $\O$.) Since
\Be
\| \o_0 - \o_0 ^\beta \|_{H^{-1} (\O)} \le \| u_0 - u_0 ^\beta \|_{L^2 (\O ) } \le C \beta \| \o_0 \|_{L^2 (\O )}, \notag
\Ee
we have
\Be
\| \o_0 - \o_0 ^\beta \|_{L^2 (\O)} \le C \beta^{\frac{s'}{1+s'}} \| \o_0 \|_{L^2 (\O ) }^{\frac{s'}{1+s'}} \| \o_0 \|_{B^s_{2,\infty} (\O) } ^{\frac{1}{1+s'}},\notag
\Ee
and similarly
\Be
\| \o_0 - \o_0 ^\ell \|_{L^2 (\O)} \le C \ell^{\frac{s'}{1+s'}} \| \o_0 \|_{L^2 (\O ) }^{\frac{s'}{1+s'}} \| \o_0 \|_{B^s_{2,\infty} (\O) } ^{\frac{1}{1+s'}}.\notag
\Ee
Finally, we match $\ell$ and $\beta$ to find a rate of convergence: we match $\ell$ so that
\Be
\frac{1}{\ell^3 | \log \mathrm{Rate}(\beta) | } = \ell^{\frac{s'}{1+s'} }.\notag
\Ee
Then we have
\Be\notag
\ell^{\frac{s'}{1+s'}} = \frac{1}{\ell^3 | \log \mathrm{Rate}(\beta) | } = \left ( \frac{1}{|\log \mathrm{Rate} (\beta ) | } \right )^{\frac{s'}{3+4s'} } \rightarrow 0,
\Ee
as $\beta\rightarrow 0$. To summarize, we have
\Be\notag
\| \o_0 - \o_0 ^\beta \|_{L^2 (\O ) } \le C(T, \| \o_0 \|_{L^2 (\O) }, \| \o_0 \|_{B^s_{2, \infty } (\O ) } ) \left (\beta^{\frac{s'}{1+s'}} +  \left ( \frac{1}{|\log \mathrm{Rate} (\beta ) | } \right )^{\frac{s'}{3+4s'} } \right ),
\Ee
as desired. Note that in the Yudovich class, $\mathrm{Rate}(\beta) = \beta^C$, and thus this rate is dominated by $\frac{1}{|\log \beta |^\alpha}$, which is much slower than algebraic rate $\beta^\alpha$. 

Next, we prove the improved rate for the Yudovich initial data $\o_0 \in L^\infty(\O)$. First, we calculate the rate of distance $d(X^\beta (0; t, x), X^\beta (0; t, y))$ with respect to $d(x,y)$, which is uniform in $\beta$. For the later purpose, we calculate the rate for localized Yudovich class as well: $m=0$ corresponds to $\o_0 \in L^\infty (\O)$.

If $\o_0 \in Y_{\mathrm{ul}}^\Theta (\O) \cap B^s_{2, \infty} (\O)$, then so is $\o_0^\beta \in Y_{\mathrm{ul}}^\Theta (\O) \cap B^s_{2, \infty} (\O)$, with
\Be\notag
\sup_\beta ( \| \o_0^\beta  \|_{Y_{\mathrm{ul}}^\Theta (\O)} + \|\o_0 ^\beta \|_{B^s_{2, \infty} (\O)} ) \le  (\| \o_0  \|_{Y_{\mathrm{ul}}^\Theta (\O)} + \|\o_0  \|_{B^s_{2, \infty} (\O)} ).
\Ee
We first estimate the modulus of continuity for $u^\beta$ with $\o_0 \in Y_{\mathrm{ul}}^\Theta (\O)$, given by Theorem \ref{thm:localYudovichwellposedness}. 
\Be\notag
\varphi_\Theta (r) \le
\begin{cases}
0, r=0, \\
r (1 - \log r)  \prod_{k=1} ^m \log_k (1 - \log r), 0 < r < \frac{1}{e^{e_m(1) -1 } }, \\
C(\Theta), r \ge \frac{1}{e^{e_m(1) -1 } }
\end{cases}
\Ee
where $C(\Theta)$ is a constant depending on $\Theta$.

We have
\begin{align}\notag
&| X^\beta (0; t, x)- X^\beta (0; t, y) | \le |x- y| + \int_0 ^t  \left |\frac{d}{ds} X^\beta(s; t, x)- \frac{d}{ds} X^\beta(s; t, y)\right | \dd s \notag\\
&= |x-y| + \int_0 ^t | u(X^\beta (s; t, x), s) - u(X^\beta(s; t, y) , s) | \dd s \notag\\
& \le |x-y| + \int_0 ^t \varphi_\Theta (|X^\beta(s; t, x)- X^\beta(s; t, y) | ) B \dd s.\notag
\end{align}
Here, by the Theorem \ref{thm:localYudovichwellposedness}, $C$ is uniform in $\beta$. Then by Osgood's lemma, we have
\Be\notag
-\mathcal{M} (\left |X^\beta (0; t, x), X^\beta(0; t, y) \right | ) + \mathcal{M} (|x-y| ) \le Bt, 
\Ee
where 
\Be\notag
\mathcal{M} (x) = \int_x ^1 \frac{\dd r}{\varphi_\Theta (r) } = \begin{cases}
\int_x ^{\exp(\frac{1}{{e_m(1) - 1 } })} \frac{1}{r (1-\log r) \prod_{k=1} ^m \log_k (1-\log r) } \dd r + \int_{\frac{1}{e^{e_m(1) - 1 } }} ^1 \frac{\dd r}{\varphi^\Theta (r) }, x < \exp(\frac{1}{e^{e_m(1) - 1 }}), \\
 \int_{x } ^1 \frac{\dd r}{\varphi_\Theta (r) }, x \ge  \exp(\frac{1}{e^{e_m(1) - 1 }}),
\end{cases}
\Ee
and $B$ is an upper bound for $\|u^\beta \|_{L^\infty ([0, T]; C_b ^{0, \varphi_\Theta} (\O, \mathbb{R}^2 ) ) }$. For future purpose, we take $B$ so that $e^{BT} > e_m (1)$. 
Thus, if $x \ge  \exp(\frac{1}{e^{e_m(1) - 1 }})$, $\mathcal{M} (x) \le C_0$ for some positive constant $C_0$. If $x < \exp(\frac{1}{e^{e_m(1) - 1 }})$, then
\Be\notag
\int_x ^{\exp(\frac{1}{{e_m(1) - 1 } })} \frac{1}{r (1-\log r) \prod_{k=1} ^m \log_k (1-\log r) } \dd r = \log_{m+1} (1- \log x)
\Ee
using the substitution $y = \log_m (1-\log r)$, and thus
\Be\notag
\mathcal{M} (x) \in [\log_{m+1} (1-\log x), \log_{m+1} (1-\log x) + C_0]
\Ee
for a (possibly larger) positive constant $C_0$. Therefore, if $|x-y|$ is sufficiently small so that $\log_{m+1} (1 - \log |x-y|) - BT > C_0$, then since
\Be\notag
\mathcal{M} ( | X^\beta (0; t, x)- X^\beta(0; t, y)  |) \ge \mathcal{M} (|x-y|) - Bt \ge \log_{m+1} (1 - \log |x-y| ) - BT,
\Ee
$|X^\beta (0; t, x)- X^\beta(0; t, y) |<  \exp(\frac{1}{e^{e_m(1) - 1 }})$, and therefore we have
\Be\notag
\log_{m+1} (1 - \log (|X^\beta (0; t, x)- X^\beta(0; t, y) | ) ) \ge \log_{m+1} (1 - \log |x-y| ) - BT - C_0,
\Ee
which gives
\Be\notag
1 - \log (|X^\beta (0; t, x)- X^\beta(0; t, y) | ) \ge e_{m+1} (\log_{m+1} (1 - \log |x-y| ) - BT - C_0),
\Ee
or
\Be\notag
|X^\beta (0; t, x)- X^\beta(0; t, y) |  \le e \exp ( - (e_{m+1} \left (\log_{m+2} (\frac{e}{|x-y|}) - BT - C_0\right )  ) ),
\Ee
which is uniform in $\beta$. 

From now on, we assume $m=0$. We closely follow the proof of \cite{CDE2019} (and \cite{BCD2011}). We rewrite the above as
\Be\notag
|X^\beta (0; t, x)- X^\beta(0; t, y) | \le e \left ( \frac{|x-y| }{e} \right )^{e^{-(BT + C_0 ) }} =: C(T) (|x-y|)^{\alpha(T)},
\Ee
where $\alpha(T) = \exp (-(BT + C_0) )$ which is deteriorating in time and $C(T) = \exp(1 - e^{-(BT+ C_0) } )$ which increases in time.

Next, we introduce the space $F^s_\pp (\O)$, which belongs to the family of Triebel-Lizorkin spaces $F^s_\pp = F^s_{\pp, \infty}$ for $\pp>1$:
\Be \label{TriebelLizorkin}
\begin{split}
 F^s_{\pp} (\O)  
&= \{ f \in L^\pp (\O) | \text{ there exists } g \in L^\pp (\O) \text{ such that for every } x, y \in \O,\\
& \ \ \ \  \frac{|f(x) - f(y) | }{|x-y|^s} \le g(x) + g(y) \},\end{split}
\Ee
and its seminorm $[\cdot ]_{F^s_\pp}$ is defined by
\Be\notag
[f]_{F^s_\pp} :=  \inf_{g \in L^\pp (\O)} \{ \| g \|_{L^\pp (\O)} | |f(x) - f(y) |\le (|x-y|)^s (g(x) + g(y) ), \text{ for every } x, y \in \O \}.
\Ee
The norm on $F^s_{\pp} (\O)$ is naturally defined by $\| \cdot \|_{L^\pp (\O ) } + [\cdot ]_{F^s_\pp}$. 

Now we argue that solution in Yudovich class propagates Besov regularity. First, we use the following embeddings: for $s_3 > s_2 > s_1$, we have continuous embeddings (the proof for whole space, which is standard, can be easily translated to periodic domain $\O$.)
\Be \label{Besovembedding}
B_{\pp, \infty} ^{s_3} (\O) \subset B_{\pp, 1} ^{s_2} (\O) \subset W^{s_2, \pp}  (\O) \subset F_{\pp} ^{s_1} (\O)  \subset B_{\pp, \infty} ^{s_1} (\O).
\Ee
Therefore, since $\o_0 \in B^s_{2, \infty} (\O)$ for some $s>0$, we have $\o_0 \in F^{s_1}_{2} $ for some $s_1 \in (0, s)$, and thus so are $\o_0 ^\beta$s with uniform bounds on $F^{s_1}_2$ norm. Then for any $\beta \ge 0$ (we introduce the convention that $X^0  = X$ and $\o^0 = \o$) we have
\begin{align}
&\frac{|\o^\beta (x,t) - \o^\beta (y, t) |}{(|x-y|)^{s_1 \alpha (T)} } = \frac{| \o_0 ^\beta (X^\beta (0; t, x) ) - \o_0 ^\beta (X^\beta (0; t, y) ) | } { (|x-y|)^{s_1 \alpha (T)}  } \notag\\
&=  \frac{| \o_0 ^\beta (X^\beta (0; t, x) ) - \o_0 ^\beta (X^\beta (0; t, y) ) | } {d(X^\beta (0; t, x), X^\beta (0; t, y) ) ^{s_1} } \frac{(|X^\beta (0; t, x)- X^\beta (0; t, y) | )^{s_1} }{ (|x-y|)^{s_1 \alpha (T)}  } \notag \\
& \le \left (g(X^\beta (0; t, x)) + g(X^\beta (0; t, y) ) \right ) C(T),\notag
\end{align}
for any $g \in L^2 (\O)$ satisfying \eqref{TriebelLizorkin}. Therefore, $C(T) g \circ X^\beta (0; t , \cdot )$ satisfies defining condition for \eqref{TriebelLizorkin} and thus $\o^\beta (t) \in F^{s_1 \alpha(T)}_2$ with 
\Be
\| \o^\beta (t) \|_{F^{s_1 \alpha(T) } _2 } \le C(T)  \| \o_0 \|_{F^{s_1}_2}.\notag
\Ee
Therefore, using \eqref{Besovembedding}, we have
\Be\notag
\| \o^\beta (t) \|_{B^{s_1 \alpha (T) } _{2, \infty } (\O) } \le C \| \o^\beta (t) \|_{F^{s_1 \alpha(T) } _2 (\O) } \le C(T) \| \o_0 \|_{F^{s_1}_2 (\O) } \le C(T) \| \o_0 \|_{B^s_{2, \infty} (\O)}.
\Ee
Now we use the interpolation inequality:
\Be\notag
\| \o^\beta (t) - \o(t) \|_{L^2 (\O) } \le \| \o^\beta (t) - \o(t) \|_{H^{-1} (\O) } ^{\frac{s_0}{1+s_0 } } \| \o^\beta (t) - \o(t) \|_{B^{s_1 \alpha(T) } _{2, \infty } }^{\frac{1}{1+s_0 } }.
\Ee
for some $s_0 < s_1 \alpha(T)$. Therefore, we have
\Be\notag
\| \o^\beta (t) - \o(t) \|_{L^2 (\O) } \le \| u^\beta (t) - u(t) \|_{L^2 (\O ) } ^{\frac{s_0}{1+s_0 } } C(T, \|\o_0\|_{B^s_{2, \infty} (\O ) } ) \le C(T, \|\o_0\|_{B^s_{2, \infty} (\O ) } ) \beta^{C e^{-C(\|\o_0 \|_{L^\infty (\O ) }) T } },
\Ee
by noting that the rate function for Yudovich case is algebraic: that is, $\mathrm{Rate}(\beta) = \beta^{2 e^{-C(\|\o_0 \|_{L^\infty (\O ) }) T }}$.
\end{proof}

\begin{remark}

One may naturally ask if one can obtain faster rate than \eqref{vorticityratelocYud}, analogous to \eqref{vorticityrateYud}. It seems that the argument we presented for \eqref{vorticityrateYud} does not extend to localized Yudovich space.

First, if $m>0$ for the modulus of continuity given by 
\Be \label{modcontinuityX}
\mu (|x-y|, T) = \exp \left ( - e_{m+1} \left ( \log_{m+2} \frac{e}{|x-y|} - (BT + C_0 ) \right ) \right )
\Ee
cannot be bounded by any H\"{o}lder exponent $|x-y|^\alpha$ for any $\alpha \in (0,1)$. Thus, we cannot continue the argument from there. To see this, suppose that there exists a $\alpha>0$ and $C>0$ such that 
\Be\notag
\mu(r, T) \le C r^\alpha,
\Ee
for any $r<1$ very small. This amounts to say that
\Be\notag
\log_{m+2} \frac{e}{r} - \log_{m+2} \frac{1}{Cr^\alpha} \ge BT + C_0.
\Ee
Taking exponential, we have
\Be\notag
\frac{\log_{m+1} \frac{e}{r} }{\log_{m+1} \frac{1}{Cr^\alpha} } \ge e^{BT + C_0}.
\Ee
Since both denominator and numerator diverges as $r \rightarrow 0^+$, we may apply L'Hopital's rule:
\begin{align}
&\frac{d}{dr} \log_{m+1} \frac{e}{r} = \frac{1}{\prod_{k=1} ^m \log_k \frac{e}{r} } \left ( -\frac{1}{r} \right ),\notag \\
&\frac{d}{dr} \log_{m+1} \frac{1}{Cr^\alpha} = \frac{1}{\prod_{k=1} ^m \log_k \frac{1}{Cr^\alpha} } \left ( -\frac{\alpha}{r} \right ).\notag
\end{align}
Inductively, we have
\begin{align}
&\lim_{r\rightarrow 0^+} \frac{\log_{0+1} \frac{e}{r} }{\log_{0+1} \frac{1}{Cr^\alpha} } = \frac{1}{\alpha},\notag \\
&\lim_{r\rightarrow 0^+} \frac{\log_{1+1} \frac{e}{r} }{\log_{1+1} \frac{1}{Cr^\alpha} } = \lim_{r\rightarrow 0^+} \frac{\log_1 \frac{1}{Cr^\alpha} }{\log_1 \frac{e}{r} } \frac{1}{\alpha} = \frac{\alpha}{\alpha} = 1,\notag \\
&\cdots \notag \\
&\lim_{r\rightarrow 0^+} \frac{\log_{m+1} \frac{e}{r} }{\log_{m+1} \frac{1}{Cr^\alpha} } = \prod_{k=1} ^m \frac{\log_k \frac{1}{Cr^\alpha} } {\log_k \frac{e}{r} } \frac{1}{\alpha} = 1.\notag
\end{align}
Therefore, except for $m=0$, where the limit is given by $\frac{1}{\alpha}$, for any $\alpha>0$ and $C>0$, there exists small $r>0$ such that $\mu(r, T) > Cr^\alpha$. Thus, control of vorticity in Triebel-Lizorkin space $F_p ^{s(t)}$ is not available.

There are other methods for propagation of regularity (in losing manner), but it seems that they also suffer from similar issue: flows generated by localized Yudovich class does not propagate enough regularity.

The argument of \cite{BCD2011} does not extend to localized Yudovich class as well: when $\o_0$ is locally Yudovich, the modulus of continuity for $u$ is weaker than log-Lipschitz: it is known that the norm defined by 
\Be\notag
\|u\|_{LL'} = \| u\|_{L^\infty} + \sup_{j \ge 0} \frac{\| \nabla S_j u \|_{L^\infty} }{(j+1)},
\Ee
where $S_j u = \sum_{k=-1} ^j \Delta_k u$, is equivalent to the norm of Log-Lipschitz space (proposition 2.111 of \cite{BCD2011}, which is for the whole case, but can be adopted the periodic domain easily). However, if $\o_0 \in Y_{\mathrm{ul} }^\Theta$, then $u$ has the modulus of continuity $\varphi_\Theta$, and the norm for $C_b ^{0, \varphi_\Theta} (\O)$ is equivalent to
\Be\notag
\|u\|_{L^\infty} + \sup_{j \ge 0} \frac{\| \nabla S_j u \|_{L^\infty} }{\prod_{k=1} ^{m+1} \log_k (e 2^j ) } ,
\Ee
which is less than $\|u\|_{LL'}$. However, the critical growth rate for the denominator in applying the linear loss of regularity result (for example, Theorem 3.28 of \cite{BCD2011}) is $j+1$, which is the rate of Log-Lipschitz norm. Therefore, we cannot rely on the argument of \cite{BCD2011} to conclude that $\o(t)$ has certain Besov regularity.

Finally, a borderline Besov space $B_\Gamma$, introduced by Vishik in \cite{V1999}) has a certain regularity (in the sense that $B_\Gamma$ restricts the rate of growth of frequency components) propagates, but it is not clear how to use this to obtain convergence rate for vorticity. For simplicity, we focus on one particular form of growth function: let 
\Be\notag
\Gamma (r) = (r+2) \frac{\log (r+3)}{\log 2}, \Gamma_1 (r) = \frac{\log (r+3)}{\log 2}
\Ee
for $r \ge -1$ and $\Gamma(r) = \Gamma_1 (r) = 1$ for $r \le -1$. We define the space $B_\Gamma$ by
\Be\notag
B_\Gamma =\left  \{ f | \| f \|_{\Gamma} := \sup_{N \ge -1} \frac{\sum_{j=-1} ^N\| \Delta_j f \|_{L^\infty }}{\Gamma(N)} < \infty \right \}
\Ee
and we define $B_{\Gamma_1}$ in a similar manner. In \cite{V1999}, the following was proved:
\begin{theorem}[\cite{V1999}]
If $\o_0 \in L^{\pp_0} \cap L^{\pp_1} \cap B_{\Gamma_1}$, for $1 < \pp_0 < 2 < \pp_1 < \infty$, then for any $T>0$ there uniquely exists a weak solution $\o(t)$ of Euler equation satisfying
\Be\notag
\| \o(t) \|_{\Gamma} \le \lambda(t),
\Ee
where $\lambda(t)$ depends only on the bounds on $\|\o_0 \|_{L^{\pp_0} \cap L^{\pp_1} \cap B_{\Gamma_1}}$. 
\end{theorem}

Therefore, one can prove uniform boundedness of vorticity in $B_{\Gamma}$ space. However, it is not clear how one can interpolate $B_{\Gamma}$ space and the velocity space (where we have rate of convergence) to obtain rate for $L^p$ norm of the vorticity.

Indeed, it was recently shown that if the velocity field is worse than Lipschitz ($u \in W^{1,p}$ for $p<\infty$), then it is possible for smooth data to lose all Sobolev regularity instantaneously from the transport by $u$ (\cite{ACM2019}). Instead, only a logarithm of a derivative can be preserved (see, e.g. \cite{BN2018}), and this loss of regularity prohibits faster convergence.

\end{remark}

\section{Proof of the Main theorems}

\begin{lemma}
	\Be\label{est:dec:F-M}
	\begin{split}
&	\left\| \frac{ F^\e(t) - M_{1, \e u(t) , 1}}{\e 
		\sqrt{M_{1,0,1}}} \right\|_{L^p_xL^2_v} \\
	&
	\lesssim   e^{ \frac{ \e^2}{4}  \| u^\beta \|^2_{ \infty  }} \Big\{ \|u^\beta  (t)- u (t)\|_{L^p _x}  
	e^{\e^2 \|u-u^\beta\|^2_{\infty}}  
	+ \kappa^{ \min \{1, \frac{p+2}{2p}\}} \sqrt{\mathcal{E} (t)}
	+ \e \kappa V(\beta)
	\Big\}.\end{split}
	\Ee
	\Be\label{est:dec:DF-M}
	\begin{split}
	&	\left\| \frac{ \nabla_x (F^\e - M_{1, \e u(t) , 1})}{\e (1+ |v|)
			\sqrt{M_{1,0,1}}} \right\|_{L^p_xL^2_v}  \\
		&
		\lesssim \big\{ \|\nabla_x u^\beta -\nabla_x  u\|_{L^p_x}  +
		\e \|\nabla_x u \|_{L^p_x}  + \e \|\nabla_x u^\beta \|_{L^p_x} \big\} 
		e^{\e^2 \|u-u^\beta\|^2_{\infty}}  e^{\e^2 \| u^\beta\|^2_\infty}\\
		&+  e^{ \frac{ \e^2  \| u^\beta \|^2_{L^\infty (\O)}}{4}} \{\kappa^{ \min \{\frac{1}{p}, \frac{1}{2}\}} \sqrt{\mathcal{E}(t)}+ \e \kappa V(\beta) \}	.
	\end{split}	\Ee
\end{lemma}

\begin{proof} 	\hide
	Recall the notation of a Maxwellian associates to the Lagrangian solutions $u^\beta$ of \eqref{vorticity_eqtn}-\eqref{vorticity_beta}:
	\Be\label{mu_beta}
	\mu =\mu  (t,x,v) = M_{1, \e u^\beta (t,x), 1}  (v).
	\Ee
	For the simplicity, we will use the following notation for the global Maxwellian: 	
	\Be\label{mu_0}
	\mu_0 = \mu_0 (v)= M_{1,0,1} (v). 
	\Ee

	\textit{Step 1. Proof of \eqref{est:dec:F-M}. }	We decompose
	\Be\label{dec:F-M}
	\begin{split}
		&\left\| \frac{ F^\e - M_{1, \e u  , 1}}{\e    \sqrt{M_{1,0 ,1}}} \right\|_{L^p_x L^2_v } \\
		&	\leq \left\| \frac{   M_{1, \e u^\beta , 1} - M_{1, \e u  , 1} }{\e  \sqrt{M_{1,0 ,1}}} \right\|_{L^p_x L^2_v }+
		\left\|
		\sqrt{	\frac{M_{1, \e u^\beta, 1}}{   M_{1,0,1}} }
		\right\|_{L^\infty_{x,v} } 
		\left\| \frac{ F^\e - M_{1, \e u^\beta , 1}}{\e \sqrt{M_{1,\e u^\beta ,1}}} \right\|_{L^p_x L^2_v } 
		\\
		&  = \eqref{dec:F-M}_1 + \eqref{dec:F-M}_2  \eqref{dec:F-M}_3.
	\end{split}
	\Ee
	
	The bound of $\eqref{dec:F-M}_1$ raises the need for consideration of $\sqrt{M_{1, A , 1}/ M_{1,0,1}}$ for $A \in \R^3$:
	\Be\label{M/M_0}
	\sqrt{M_{1, A , 1}/ M_{1,0,1}}  \lesssim
	e^{\frac{-|v-A|^2 + |v|^2}{4}}
	\leq 
	e^{\frac{|A|^2}{4}}
	.
	\Ee
	Using \eqref{M/M_0} and the Taylor expansion, we derive that, 
	\Be\label{taylor}
	\begin{split}
		&\frac{| M_{1, \e u^\beta ,1} -M_{1, \e u , 1} |}{ \e \sqrt{M_{1,0,1}}}\\
		&= \frac{1}{\e}\Big| \int^\e_0
		\big(
		(v-\e u) + a (u- u^\beta )
		\big)
		\cdot (u^\beta - u) 
		\frac{M_{1, \e u - a (u- u^\beta),1}}{\sqrt{M_{1,0,1}}}
		\dd a  \Big|\\
		&\leq |u^\beta - u| \frac{1}{\e} \int^\e_0
		|(v-\e u) + a (u- u^\beta ) | | M_{1, \e u - a (u- u^\beta),1}  |^{\frac{1}{2}}
		e^{\frac{ |\e u -a (u-u^\beta)|^2 }{4 }}
		\dd a \\
		&\lesssim   |u^\beta - u|  e^{\e^2 |u-u^\beta|^2}  e^{\e^2 | u^\beta|^2}
		\frac{1}{\e} \int^\e_0
		| M_{1, \e u - a (u- u^\beta),1 } (v)|^{\frac{1 }{4}}
		\dd a
		,
	\end{split}
	\Ee 
	where we have used $ |(v-\e u) + a (u- u^\beta ) |  |M_{1, \e u - a (u- u^\beta),1}|^{\frac{1 }{2}} \lesssim |M_{1, \e u - a (u- u^\beta),1}|^{\frac{1   }{4}} 
	$ and $|\e u -a (u-u^\beta)| = |(\e -a) u - (\e -a ) u^\beta  + \e u^\beta|
	\leq |\e -a | |u- u^\beta| + \e |u^\beta| \leq \e \{|u- u^\beta|  + |u^\beta|\}$.
	
	Now taking an $L^p_x L^2_v$-norm to \eqref{taylor}, we conclude that 
	\Be\label{est:dec:F-M_1}
	\begin{split}
		\eqref{dec:F-M}_1 &\lesssim \|u^\beta - u\|_{L^p_x}  
		e^{\e^2 \|u-u^\beta\|^2_{\infty} +\e^2 \| u^\beta\|^2_\infty} 
		\frac{1}{\e} \sup_{x \in \O}\Big( \int^\e_0
		\|  M_{1, \e u - a (u- u^\beta),1  } \|_{L^2 (\R^3)}^{\frac{1 }{4}} 
		\dd a\Big)\\
		&\lesssim \|u^\beta - u\|_{L^p _x}  
		e^{\e^2 \|u-u^\beta\|^2_{\infty}}  e^{\e^2 \| u^\beta\|^2_\infty}
		\sup_{x \in \O}  \| M_{1, 0, 1} \|_{L^2(\R^3)}^{1/4}\\
		&\lesssim \|u^\beta - u\|_{L^p _x}  
		e^{\e^2 \|u-u^\beta\|^2_{\infty}}  e^{\e^2 \| u^\beta\|^2_\infty} .
	\end{split}\Ee

	From \eqref{M/M_0}, clearly we have 
	\Be\label{est:dec:F-M_2}
	\eqref{dec:F-M}_2 \lesssim e^{ \frac{ \e^2  \| u^\beta \|^2_{L^\infty (\O)}}{4}}.
	\Ee

	Using the expansion \eqref{F_e}, we can bound $\eqref{dec:F-M}_3$:
	\Be\label{est:dec:F-M_3}
	\eqref{dec:F-M}_3 \lesssim  \| f_R \|_{L^p _x L^2_v } + \e \kappa V(\beta)
	\lesssim \| \nabla_x f_R \|_{L^2_{x, v }}^{\frac{p-2}{p}} \|   f_R \|_{L^2_{x, v }} ^{\frac{2}{p}}+ \e \kappa V(\beta)
	\Ee

	\unhide
	
		We only prove \eqref{est:dec:DF-M}, as the proof of \eqref{est:dec:F-M} is similar and simpler. We decompose
	\Be\label{dec:F-M}
	\begin{split}
		&\left\| \frac{  \nabla_x (F^\e - M_{1, \e u  , 1})}{\e (1+ |v|)   \sqrt{M_{1,0 ,1}}} \right\|_{L^p_x L^2_v } \\
		&	\leq \left\| \frac{   \nabla_x ( M_{1, \e u^\beta , 1} - M_{1, \e u  , 1}) }{\e  (1+ |v|)   \sqrt{M_{1,0 ,1}}} \right\|_{L^p_x L^2_v }+
		\left\|
		\sqrt{	\frac{M_{1, \e u^\beta, 1} ^{1+o(1) }}{   M_{1,0,1}} }
		\right\|_{L^\infty_{x,v} } 
		\left\| \frac{ \nabla_x (F^\e - M_{1, \e u^\beta , 1})}{\e (1+ |v|)  \sqrt{M_{1,\e u^\beta ,1}^{1+o(1) } }} \right\|_{L^p_x L^2_v } 
		\\
		&  = \eqref{dec:F-M}_1 + \eqref{dec:F-M}_2  \eqref{dec:F-M}_3.
	\end{split}
	\Ee
	
	The bound of $\eqref{dec:F-M}_1$ raises the need for consideration of $\sqrt{M_{1, A , 1}^{1+o(1) } / M_{1,0,1}}$ for $A \in \R^3$:
	\Be\label{M/M_0}
	\sqrt{M_{1, A , 1}^{1+o(1) } / M_{1,0,1}}  \lesssim
	e^{\frac{-(1+o(1) )|v-A|^2  + |v|^2}{4}}
	\leq 
	e^{\frac{|A|^2}{4}}
	.
	\Ee
	Using \eqref{M/M_0} and the Taylor expansion, we derive that, 
	\Be\label{taylor}
	\begin{split}
		&\frac{|  \nabla_x (M_{1, \e u^\beta ,1} -M_{1, \e u , 1} )|}{ \e \sqrt{M_{1,0,1}}}\\
		&= \frac{1}{\e}\Big| \int^\e_0
		\nabla_x \Big( \big(
		(v-\e u) + a (u- u^\beta )
		\big)
		\cdot (u^\beta - u) 
		\frac{M_{1, \e u - a (u- u^\beta),1}}{\sqrt{M_{1,0,1}}}\Big)
		\dd a  \Big|\\
		&\lesssim   \{|\nabla_x u^\beta - \nabla_x u| + \e|\nabla_x u| + \e |\nabla_x u^\beta|\} e^{\e^2 |u-u^\beta|^2}  e^{\e^2 | u^\beta|^2}
		\frac{1}{\e} \int^\e_0
		| M_{1, \e u - a (u- u^\beta),1 } (v)|^{\frac{1 }{4}}
		\dd a
		,
	\end{split}
	\Ee 
	where we have used $ |(v-\e u) + a (u- u^\beta ) |  |M_{1, \e u - a (u- u^\beta),1}|^{\frac{1 }{2} - o(1)/2 } \lesssim |M_{1, \e u - a (u- u^\beta),1}|^{\frac{1   }{4}} 
	$ and $|\e u -a (u-u^\beta)| = |(\e -a) u - (\e -a ) u^\beta  + \e u^\beta|
	\leq |\e -a | |u- u^\beta| + \e |u^\beta| \leq \e \{|u- u^\beta|  + |u^\beta|\}$.
	
	Now taking an $L^p_x L^2_v$-norm to \eqref{taylor}, we conclude that 
	\Be\label{est:dec:F-M_1}
	\begin{split}
		\eqref{dec:F-M}_1 
		&\lesssim  \big\{ \|\nabla_x u^\beta -\nabla_x  u\|_{L^p_x}  +
		\e \|\nabla_x u \|_{L^p_x}  + \e \|\nabla_x u^\beta \|_{L^p_x} \big\} 
		e^{\e^2 \|u-u^\beta\|^2_{\infty}}  e^{\e^2 \| u^\beta\|^2_\infty} .
	\end{split}\Ee

	From \eqref{M/M_0}, clearly we have 
	\Be\label{est:dec:F-M_2}
	\eqref{dec:F-M}_2 \lesssim e^{ \frac{ \e^2  \| u^\beta \|^2_{L^\infty (\O)}}{4}}.
	\Ee

	Using the expansion \eqref{F_e}, we can bound $\eqref{dec:F-M}_3$:
	\Be\label{est:dec:F-M_3}
	\begin{split}
		\eqref{dec:F-M}_3&
		\lesssim \| \nabla_x f^\e \|_{L^p_x L^2_x}+ 
		\e \| u^\beta \|_{\infty}\|  f^\e \|_{L^p_x L^2_x}+ \e \kappa V(\beta)
		\\
		& 
		\lesssim 
		\| \nabla_x^2 f_R \|_{L^2_{x, v }}^{\frac{p-2}{p}} \|  \nabla_x f_R \|_{L^2_{x, v }} ^{\frac{2}{p}}+ \e \| u^\beta \|_{\infty}
		\| \nabla_x f_R \|_{L^2_{x, v }}^{\frac{p-2}{p}} \|   f_R \|_{L^2_{x, v }} ^{\frac{2}{p}}+ \e \kappa V(\beta)\\
		& \lesssim \kappa^{ \min \{\frac{1}{p}, \frac{1}{2}\}} \sqrt{\mathcal{E}(t)}+ \e \kappa V(\beta).
	\end{split}
	\Ee
We finish the proof by applying \eqref{est:dec:F-M_1}-\eqref{est:dec:F-M_3} to \eqref{dec:F-M}. 
\end{proof}

We claim that 
\begin{lemma}
	\Be
	\| \o_B^\e(t) - \o(t) \|_{L^\pp (\T^2)} \lesssim  \| \o^\beta (t) - \o(t) \|_{L^\pp (\T^2)}
	+ \kappa^{ \min \{\frac{1}{2}, \frac{1}{\pp}\}} \sqrt{\mathcal{E} (t)} + \e \kappa V(\beta).
	\Ee
\end{lemma}
\begin{proof}Recall $F^\e$ in \eqref{F_e}. Note that  
	\begin{align} 
		&	\o^\e_B(t,x)- \o (t,x)  = \nabla^\perp \cdot  u^\e_B(t,x)  -  \nabla^\perp \cdot  u (t,x)  \notag\\
		& = \frac{1}{\e}   \int_{\R^3}    \munderbar{v} \cdot  \nabla^\perp (F^\e(t,x,v) - M_{1,\e u ,1} (v)) \dd v\notag \\
		&
		=  \frac{1}{\e}   \int_{\R^3}    \munderbar{v} \cdot  \nabla^\perp (
		M_{1, \e u^\beta,1} (v)
		- M_{1,\e u ,1} (v)) \dd v \ \  \  (=\o^\beta - \o)
		\label{diff:o1}
		\\
		& \ \ 	 +      \int_{\R^3}   \nabla^\perp 	f_R(t,x,v)   \cdot  \munderbar{v}    \sqrt{\mu}
		\dd v 
		\label{diff:o3}\\
		& \ \ +      \nabla^\perp \cdot   \int_{\R^3}   
		\{ \e^2 p^\beta \mu - \e^2 \kappa (\nabla_x u^\beta) : \mathfrak{A} \sqrt{\mu} + \e \kappa \tilde{u}^\beta \cdot (v-\e u^\beta) \mu + \e^2 \kappa \tilde{p}^\beta \mu  \} \dd v .
		\label{diff:o2}
	\end{align} 
	Clearly 
	\Be\notag
	\| \eqref{diff:o1}\|_{L^\pp (\T^2)} =\|  \o^\beta(t) - \o(t)\|_{L^\pp (\T^2)} .
	\Ee
	From Theorem \ref{theo:remainder}, we conclude that 
	\Be\notag
	\|	\eqref{diff:o3} \|_{L^\pp(\T^2)} \lesssim \begin{cases} \| \nabla_x f_R (t) \|_{L^2 (\T^2 \times \R^3)} \lesssim  \sqrt{\kappa} \sqrt{\mathcal{E} (t)} \ \ for \ \pp  \in [1,2],\\
		\| \nabla_x^2 f_R(t) \|_{L^2 (\T^2 \times \R^3) }^{\frac{\pp -2}{\pp}} 	   \| \nabla_x f_R(t) \|_{L^2 (\T^2 \times \R^3) }^{\frac{2}{\pp}} \lesssim \kappa^{\frac{1}{\pp}} \sqrt{\mathcal{E} (t)}
		\ \ for \ \pp  \in (2, \infty),
	\end{cases}
	\Ee
	where we have used (anisotropic) Gagliardo-Nirenberg interpolation for the second, whose proof is analogous to Lemma \ref{anint}.
	
	Using Theorem \ref{Vbound}, we get that $
	\| \eqref{diff:o2} \|_{L^\pp (\O)} \lesssim  \e \kappa V(\beta). $\end{proof}

Equipped with Proposition \ref{thm:Dptu}, Proposition \ref{theo:convergence}, and Proposition \ref{theo:remainder}, we are ready to prove the main theorem of this paper. 

\begin{theorem}\label{main_theorem1} 
	Choose an arbitrary $T \in (0, \infty)$. Suppose $(u_0, \o_0) \in L^2(\O) \times L^\pp (\O)$ for $\pp \in [1,  \infty)$ and $(u,\o)$ be a Lagrangian solution of \eqref{vorticity}-\eqref{BS}-\eqref{vorticity_initial}. Assume the initial data $F_0$ to \eqref{eqtn_F} satisfies conditions in Theorem \ref{theo:remainder}. Then there exists a family of Boltzmann solutions $F^\e(t,x,v)$ to \eqref{eqtn_F} in $[0,T]$ such that 
	\Be 
	\sup_{t \in [0,T]}	\left\| \frac{ F^\e(t) - M_{1, \e u(t) , 1}}{\e 
		\sqrt{M_{1,0,1}}} \right\|_{L^2 (\T^2 \times \R^3)} \rightarrow 0.
	\Ee
	Moreover the Boltzmann vorticity converges to the Lagrangian solution $\o$:
	\Be\label{conv:vorticity_B}
	\sup_{0 \leq t \leq T}	\|\o^\e_B(t,\cdot)- \o (t,\cdot)\|_{L^\pp(\mathbb{T}^2)} \rightarrow 0.
	\Ee
	\end{theorem}

\begin{theorem}\label{main_theorem2}
Choose an arbitrary $T \in (0, \infty)$. Suppose $\o_0 \in Y_{\mathrm{ul}}^{\Theta} (\T^2)$ for some $\Theta$ in \eqref{Thetap} with $m \in \mathbb{Z}_{\ge 0}$, and let $(u, \o)$ be the unique weak solution of \eqref{vorticity}-\eqref{BS}-\eqref{vorticity_initial}. Assume the initial data $F_0$ to \eqref{eqtn_F} satisfies conditions in Theorem \ref{theo:remainder}. Then there exists a family of Boltzmann solutions $F^\e (t,x,v)$ to \eqref{eqtn_F} in $[0, T]$ such that 
	\Be 
	\sup_{t \in [0,T]}	\left\| \frac{ F^\e(t) - M_{1, \e u(t) , 1}}{\e 
		\sqrt{M_{1,0,1}}} \right\|_{L^2 (\T^2 \times \R^3)} \rightarrow 0.
	\Ee
	Moreover the Boltzmann velocity and vorticity converges to the solution $\o$ with an explicit rate as in \eqref{Rateftn}, \eqref{Rateftnvort}:
	\Be \label{conv:vorticity_B_rate}
	\begin{split}
	\sup_{0 \leq t \leq T} \| u^\e_B (t, \cdot) - u(t, \cdot ) \|_{L^2 (\mathbb{T}^2 ) } &\lesssim \mathrm{Rate} (\beta), \\
	\sup_{0 \leq t \leq T} \| \o^\e_B (t, \cdot) - \o(t, \cdot ) \|_{L^\pp (\mathbb{T}^2 ) } &\lesssim \mathrm{Rate}_\o (\beta).
	\end{split}
	\Ee
Furthermore, if $\o_0 \in Y_{\mathrm{ul}}^{\Theta} (\T^2) \cap B^s_{2, \infty} (\mathbb{T}^2)$ for some $s>0$, Boltzmann vorticity converges to the solution $\o$ with a rate which is uniform in $\o_0$ as in \eqref{vorticityratelocYud}, \eqref{vorticityrateYud}:
	\Be \label{conv:vorticity_B_uniform}
	\begin{split}
	\sup_{0 \leq t \leq T} \| \o^\e_B (t, \cdot) - \o(t, \cdot ) \|_{L^\pp (\mathbb{T}^2 ) } &\lesssim \mathrm{Rate}_{\o,s,loc-Y} (\beta), m>0 \text{ (localized Yudovich)},\\
	\sup_{0 \leq t \leq T} \| \o^\e_B (t, \cdot) - \o(t, \cdot ) \|_{L^\pp (\mathbb{T}^2 ) } &\lesssim \mathrm{Rate}_{\o,s,Y} (\beta),  m=0 \text{ (Yudovich)}.
	\end{split}
	\Ee
\end{theorem}
 
\hide
\begin{proof}[Proof of Corollary \ref{cor:vorticity}]
	Recall the Lagrangian solution $(u^\beta, \o^\beta )$ to \eqref{vorticity_eqtn}-\eqref{vorticity_beta}; and the local Maxwellian  associates to $u^\beta$, $\mu= M_{1, \e u^\beta, 1}$ in \eqref{mu_beta}.
	From \eqref{hydro_limit} and \eqref{F_e}, 
	\Be\begin{split}\label{dec:o-o}
		&	\o^\e_B(t,x)- \o (t,x)  = \nabla^\perp u^\e_B(t,x)  -  \nabla^\perp u (t,x)  \\
		& = \frac{1}{\e} \nabla^\perp \left( \int_{\R^3}    \munderbar{v}  (F^\e(t,x,v) - M_{1,\e u ,1} (v)) \dd v\right) \\
		& = \frac{1}{\e} \nabla^\perp \left( \int_{\R^3}    \munderbar{v}  (F^\e(t,x,v)- M_{1, \e u^\beta, 1}) \dd v\right)  + \frac{1}{\e} \nabla^\perp \left(  \int_{\R^3}    \munderbar{v}   (M_{1, \e u^\beta, 1} - M_{1,\e u ,1} (v)) \dd v\right) \\ 
		& =  \nabla^\perp \left( \int_{\R^3}    \munderbar{v}  (\e  f_2   + \delta   f _R  )\sqrt{\mu } \dd v\right) + 
		(	\o^\beta - \o ) \\
		& = 
		\Big(
		\e  \nabla^\perp  u_2 + 
		\delta  \nabla^\perp  \int_{\R^3}    \munderbar{v}    f_R  \sqrt{\mu } \dd v 
		\Big)+ 
		(	\o^\beta - \o ).
	\end{split}
	\Ee
	We derive the convergence of the later term $(	\o^\beta - \o )$ using \eqref{stab_o_beta} in Proposition \ref{theo:convergence}. For the $\e \nabla^\perp u_2$, we use (??) and easily derive that $\|\e \nabla^\perp u_2\|_{L^\pp}  \rightarrow 0$. 
	
	For the remainder $f_R$, we express that 
	\Be
	\begin{split}
		&\delta \left\|	 \nabla^\perp  \int_{\R^3}    \munderbar{v}    f _R \sqrt{\mu } \dd v \right\|_{L^\pp}\\
		&
		\lesssim
		\delta \left\|	   \int_{\R^3}   \{- v_1 \p_{x_2} f_R+ v_2 \p_{x_1}f_R\} \sqrt{\mu } \dd v \right\|_{L^\pp}
		+ \e \delta \left\|	   \int_{\R^3}     f_R   \o^\beta  |v- u^\beta|  \sqrt{\mu } \dd v \right\|_{L^\pp}
		\\
		&\lesssim
		\delta \| \nabla_x f _R |  \munderbar{v}   | \sqrt{\mu} \|_{L^\pp}
		+ \e \delta  \| f_R \|_{L^\infty} \| \o^\beta \|_{L^\pp} .
	\end{split}
	\Ee

\end{proof}\unhide

\appendix


\hide	=&0, 
\end{align*}
where we have used that $B_i$ is odd in $v_i$. Hence we check (\ref{Bous_ell}) for $v\cdot \nabla_x L^{-1} (\mathbf{I} - \mathbf{P}) (v\cdot \nabla_x u_\kappa \cdot v \sqrt{\mu})$ contribution. \unhide

%
%
\hide
From $L^{-1} \Gamma (\mathbf{P}f, \mathbf{P}f)= \frac{( \mathbf{P} f)^2}{2\sqrt{\mu}}$, we derive that 
\begin{align}\notag
	L^{-1} \Gamma (u_\kappa \cdot v \sqrt{\mu}, u_\kappa \cdot v \sqrt{\mu})  
	= (\mathbf{I} - \mathbf{P}) ( \frac{1}{2 \sqrt{\mu}} \{u_\kappa \cdot v \sqrt{\mu}\}^2) 
	=  \frac{1}{2}(\mathbf{I} - \mathbf{P})
	\left(  \sum_{i,j} u_{\kappa,i} u_{\kappa, j} v_i v_j \sqrt{\mu} \right) 
	=  \frac{1}{2} 
	\sum_{i,j} u_{\kappa,i} u_{\kappa, j} \{v_i v_j - \frac{|v|^2}{3} \delta_{ij} \}\sqrt{\mu}  ,
\end{align}
where we have used   
\[
(\mathbf{I} - \mathbf{P})( v_i v_j \sqrt{\mu} ) 
=  \Big\{
v_iv_j - \frac{|v|^2}{3}  \delta_{ij}
\Big\} \sqrt{\mu}.
\]

Clearly from oddness of the functions we derive that 
\begin{align*}
	\left\langle
	\sqrt{\mu}, v\cdot \nabla_x L^{-1} (\Gamma (u_\kappa\cdot v \sqrt{\mu},u_\kappa\cdot v \sqrt{\mu} ))
	\right\rangle 
	= \left\langle
	\sqrt{\mu}, v\cdot \nabla_x   \frac{1}{2} 
	\left(  \sum_{i,j} u_{\kappa,i} u_{\kappa, j} \{v_i v_j - \frac{|v|^2}{3} \delta_{ij} \} \sqrt{\mu} \right)\right\rangle 
	=  0,\\
	\left\langle
	|v|^2 \sqrt{\mu}, v\cdot \nabla_x L^{-1} (\Gamma (u_\kappa\cdot v \sqrt{\mu},u_\kappa\cdot v \sqrt{\mu} ))
	\right\rangle 
	= \left\langle
	|v|^2  \sqrt{\mu}, v\cdot \nabla_x   \frac{1}{2} 
	\left(  \sum_{i,j} u_{\kappa,i} u_{\kappa, j} \{v_i v_j - \frac{|v|^2}{3} \delta_{ij} \} \sqrt{\mu} \right)\right\rangle 
	=  0,
\end{align*}
and hence we check (\ref{Bous_ell}) for $v\cdot \nabla_x   \frac{1}{2} 
\left(  \sum_{i,j} u_{\kappa,i} u_{\kappa, j} \{v_i v_j - \frac{|v|^2}{3} \delta_{ij} \} \sqrt{\mu} \right) $ contribution.

Now we compute 
\begin{align} 
	\left\langle
	v_\ell \sqrt{\mu}, \sum_k v_k \p_k L^{-1} (\Gamma (u_\kappa\cdot v \sqrt{\mu},u_\kappa\cdot v \sqrt{\mu} ))
	\right\rangle  &= 
	\sum_k  \left\langle v_\ell v_k \sqrt{\mu}
	,  
	(\mathbf{I} - \mathbf{P})
	\left(  \sum_{i,j}  \p_k u_{\kappa,i} u_{\kappa, j} v_i v_j \sqrt{\mu} \right) 
	\right\rangle\notag \\
	&=  \sum_k  \left\langle
	\{v_\ell v_k  - \frac{|v|^2}{3}
	\delta_{\ell k}
	\}\sqrt{\mu}
	,  
	\sum_{i,j}  \p_k u_{\kappa,i} u_{\kappa, j} \{v_i v_j - \frac{|v|^2}{3} \delta_{ij}
	\}
	\sqrt{\mu} 
	\right\rangle\notag \\
	& = \sum_{i,j,k} 
	\left\langle
	\{v_\ell v_k  - \frac{|v|^2}{3}
	\delta_{\ell k}
	\}\sqrt{\mu}
	,    \{v_i v_j - \frac{|v|^2}{3} \delta_{ij}
	\}
	\sqrt{\mu} 
	\right\rangle  \p_k u_{\kappa,i} u_{\kappa, j}
	\label{vLG}
	.
\end{align}
From direct computations and $\int_{\R} \frac{e^{- {(v_1)^2}/{2}}}{(2\pi )^{1/2}} \dd v_1=1=\int_{\R}  (v_1)^2\frac{e^{- {(v_1)^2}/{2}}}{(2\pi )^{1/2}} \dd v_1,$ and $\int_{\R}  (v_1)^4\frac{e^{- {(v_1)^2}/{2}}}{(2\pi )^{1/2}} \dd v_1=3$, we have  
\begin{align}\notag
	& \left\langle
	\{v_\ell v_k  - \frac{|v|^2}{3}
	\delta_{\ell k}
	\}\sqrt{\mu}
	,    \{v_i v_j - \frac{|v|^2}{3} \delta_{ij}
	\}
	\sqrt{\mu} 
	\right\rangle = \frac{4}{3} \delta_{\ell k } \delta_{i j} \delta_{i\ell}   - \frac{2}{3} \delta_{\ell k} \delta_{ij} (1- \delta_{i \ell})
	+  \delta_{\ell i} \delta_{k j} (1- \delta_{\ell k}) 
	+   \delta_{\ell j} \delta_{ki}(1- \delta_{\ell k}) .
\end{align}
Then we derive 
\Be\begin{split}
	(\ref{vLG}) &=
	\frac{4}{3} \p_\ell u_{\kappa, \ell} u_{\kappa, \ell} - \frac{2}{3} \sum_{j (\neq \ell)} \p_ \ell u_{\kappa, j} u_{\kappa, j} + \sum_{j (\neq \ell)} \p_j u_{\kappa, \ell} u_{\kappa, j} 
	+ u_\ell \sum_{i (\neq \ell) } \p_i u_{\kappa, i}
	\\
	&=(\frac{4}{3} -2 ) \p_\ell u_{\kappa, \ell} u_{\kappa, \ell}- \frac{2}{3} \sum_{j (\neq \ell)} \p_ \ell u_{\kappa, j} u_{\kappa, j}
	+ \sum_{j  } \p_j u_{\kappa, \ell} u_{\kappa, j} 
	+ u_\ell \sum_{i   } \p_i u_{\kappa, i}
	\\
	&= - \frac{2}{3} \sum_{j  } \p_ \ell u_{\kappa, j} u_{\kappa, j}
	+ \sum_{j  } \p_j u_{\kappa, \ell} u_{\kappa, j} 
	+ u_\ell \sum_{i   } \p_i u_{\kappa, i}
	\\
	&= - \frac{1}{3} \p_\ell|u_\kappa|^2
	+  u_\kappa \cdot \nabla_x u_{\kappa,\ell}  +  (\nabla \cdot u_\kappa )u_{\kappa, \ell}.\label{vLG1}
\end{split}\Ee\unhide
%
%
\hide From we end up with reduced equation 
\Be
\begin{split}
	\p_t f_R + \frac{1}{\e} v\cdot \nabla_x f_R+ \frac{1}{\e^2 \kappa } Lf_R 
	\\
	=
	\frac{  \delta }{\e\kappa}\Gamma(f_R,f_R)
	- \frac{\e }{\delta} \p_t f_2  - \frac{\kappa}{\delta} (\mathbf{I} - \mathbf{P}) ( v\cdot \nabla_x L^{-1} (v\cdot \nabla_x u_\kappa \cdot v \sqrt{\mu}) ) \\
	+
	\frac{1}{\delta\kappa} \Gamma(f_1,f_2) + \frac{  1}{\e\kappa} \Gamma(f_1,f_R) 
	+    \frac{\e }{\delta\kappa} \Gamma({f_2}, f_2) +  \frac{ 1}{\kappa} \Gamma({f_2}, f_R).
\end{split}
\Ee\unhide
\hide For $(\mathbf{I} - \mathbf{P}) \eqref{eqtn_f_3}$, from \eqref{f_2}, it is bounded by the sum of the two terms:
\Be\begin{split}
	\frac{1}{\delta}(\mathbf{I} - \mathbf{P}) \big( (v-\e u) \cdot \nabla_x \mathbf{P} f_2\big)&=  O\Big(\frac{1}{\delta} \Big)\{|(  \nabla\tilde{\rho},\nabla \tilde{u},\nabla \tilde{\theta})| \\
	& \ \  + \e |\nabla_x u|  |(\tilde{\rho}
	, \tilde{u}, \tilde{\theta})| \} \langle v-\e u \rangle^4  \mu^{\frac{1}{2}}, \label{r2}
\end{split}\Ee
\begin{align}
	\frac{1}{\delta} (\mathbf{I} - \mathbf{P}) \big( (v-\e u) \cdot \nabla_x  (\mathbf{I} - \mathbf{P})f_2\big)
	=  O\Big(\frac{\kappa}{\delta}\Big)\{ |\nabla^2 u|  + \e |\nabla u|^2  \} \langle v-\e u \rangle^4 \mu^{\frac{1}{2}}.\label{r3}
\end{align}

\textit{Expansion of (\ref{eqtn_f_4}):} From (\ref{f_2}) 
\Be
\begin{split}
	|(\ref{eqtn_f_4})| \lesssim& \frac{\e}{\delta} \Big\{
	|\p_t ( \tilde{\rho}, \tilde{u}, \tilde{\theta})| + |u|  |\nabla( \tilde{\rho}, \tilde{u}, \tilde{\theta})|
	+ \e |  ( \tilde{\rho}, \tilde{u}, \tilde{\theta})|( |\p_t u| + |u| |\nabla u|)\\
	& \ \ \ 
	+ \kappa (|\p_t \nabla u| + \e |\nabla u| |\p_t u|)
	+ \kappa |u| (|  \nabla^2 u| + \e |\nabla u| ^2)
	\Big\} \langle v-\e u \rangle^6 \sqrt{\mu}  .\label{r_4}
\end{split}
\Ee

\textit{Expansion of (\ref{eqtn_f_5}):} Directly we compute and bound
\Be
\begin{split}\label{r_5}
	|\eqref{eqtn_f_5}| & \lesssim 
	\Big\{
	\frac{\e^2}{\delta} |\p_t u| + \frac{\e}{\delta} |\nabla_x u|
	\Big\} \langle v-\e u \rangle^2  |f_2|\\
	&\lesssim \frac{\e}{\delta } \Big[
	\e|   \p_t u  |  + |   \nabla  u  | 
	\Big]
	\Big[
	\kappa  |    \nabla u| + |(\tilde{\rho}, \tilde{u}, \tilde{\theta})|
	\Big]  \langle v-\e u \rangle^4 \sqrt{\mu}.
\end{split} \Ee

\textit{Expansion of (\ref{eqtn_f_6}):} From \eqref{f_2} and \eqref{est:f_2}, 
\Be\label{est:eqtn_f_6} 
\begin{split}
	|	\eqref{eqtn_f_6}|= 
	|(\mathbf{I} - \mathbf{P}) \eqref{eqtn_f_6}| \lesssim 
	\frac{\e}{\delta \kappa } (\kappa^2 |\nabla u|^2 + |(\tilde{\rho}, \tilde{u}, \tilde{\theta})|^2)\langle v-\e u \rangle ^2 \sqrt{\mu} .
\end{split}
\Ee

\textit{Expansion of \eqref{eqtn_f_1}:} Note that, for $\eqref{basis}$, 
\Be\label{vDmu}
\e^{-1} (v- \e u )\cdot \nabla_x \mu = \sum_{\ell, m=1}^3 \varphi_\ell \p_\ell u_m \varphi_m  \mu.
\Ee
This implies $\frac{\e^{-1} (v-\e u) \cdot \nabla_x \mu^{1/2}}{\sqrt{\mu}}= \frac{1}{2} \sum_{\ell, m=1}^3 \varphi_\ell  \varphi_m \p_\ell u_m$
\Be\begin{split}\label{IF_energy}
	&	\frac{ [ \p_t  + \e^{-1} v\cdot \nabla_x] \sqrt{\mu}}{\sqrt{\mu}} f_R    \\
	&
	=\frac{\e^{-1} (v-\e u)\cdot \nabla_x \sqrt{\mu}}{\sqrt{\mu}} f_R   + \frac{[ \p_t  + u\cdot \nabla_x ]\sqrt{\mu}}{\sqrt{\mu}} f_R   \\
	&=\frac{1}{2} \sum_{\ell, m=1}^3 \varphi_\ell \varphi_m \p_\ell u_m f_R  
	+ \frac{\e}{2} \sum_{\ell, m=1}^3 (\p_t u_m + u_\ell \p_\ell u_m) \varphi_m f_R ,
\end{split}\Ee
where we have used the computation $u \cdot \nabla_x \sqrt{\mu}=\frac{\e}{2} \sum_{\ell,m=1}^3 u_\ell \p_\ell u_m \varphi_m \sqrt{\mu}$.

We note that from \eqref{est:f_2}
\Be\label{est:Gf2fR}
\frac{2}{\kappa} \Gamma(f_2, f_R) =  \frac{2}{\kappa}  (\mathbf{I} - \mathbf{P})\Gamma(f_2, f_R) =
O\Big(\frac{1}{\kappa}\Big)(\kappa |\nabla u| + |(\tilde{\rho}, \tilde{u}, \tilde{\theta})| ) |f_R|.
\Ee

\end{proof}\unhide

\bibliographystyle{abbrv}
\bibliography{Bibliography} 

\hide

\unhide
\end{document}